%% file: CDF_Uniform.tex
\documentclass{amsart}

 \usepackage{amsmath}
\usepackage{amsfonts}
\usepackage{url}
\usepackage{fullpage}
\usepackage{amssymb}
\usepackage[dvips]{graphicx}
\usepackage{color}
\usepackage{float} 
\usepackage{epsfig} 

\newcommand{\transp}{{\scriptscriptstyle \top}}
\newtheorem{thm}{\bf{Theorem}}[section]
\newtheorem{lemma}[thm]{\bf{Lemma}}

\newtheorem{rem}[thm]{\bf{Remark}}

\def\argmin{{\mathop{\hbox{\rm argmin}\,}}}

\begin{document}

\title[]{{Exact computation of the cumulative distribution function of the Euclidean distance between a point and a random variable uniformly distributed in disks, balls, or polygones and application to Probabilistic Seismic Hazard Analysis}}

\maketitle 

\begin{center}
Vincent Guigues\\
School of Applied Mathematics, FGV\\
Praia de Botafogo, Rio de Janeiro, Brazil\\ 
{\tt vguigues@fgv.br}
\end{center}

\date{}

\begin{abstract} 
We consider a random variable expressed as the Euclidean distance between an arbitrary point 
and a random variable uniformly distributed
 in a closed and bounded set of 
 a three-dimensional Euclidean space.
 Four cases are considered for this set: a union of disjoint disks, a union of disjoint balls, a union of disjoint line segments, and the boundary of a polyhedron. In the first three cases, we provide closed-form expressions of the cumulative distribution
 function and the density. In the last case, we propose an algorithm
 with complexity $O(n \ln n)$, $n$ being the number of edges of the polyhedron,
 that computes exactly the cumulative distribution function.
 An application of these results to probabilistic seismic hazard analysis and extensions 
 are discussed.\\
\end{abstract}

\par {\textbf{Keywords:} Computational Geometry, Geometric Probability, Distance to a random variable, Uniform distribution, Green's theorem, PSHA.\\

\par MSC2010 subject classifications: 60D05, 65D99, 51N20, 65D30, 86A15.\\

\section{Introduction}

Consider a closed and bounded set 
$\mathcal{S} \subset \mathbb{R}^3$ and a random variable
$X:\Omega \rightarrow \mathcal{S}$ uniformly distributed in $\mathcal{S}$.
Given an arbitrary point $P \in \mathbb{R}^3$,
we study the distribution of the Euclidean distance $D:\Omega \rightarrow \mathbb{R}_{+}$ between $P$ and
$X$ defined by $D(\omega)=\|\overrightarrow{P X(\omega)}\|_2$ for any $\omega \in \Omega$.

Denoting respectively the density 
and the cumulative distribution function (CDF) 
of $D$ by $f_D(\cdot)$ and $F_D(\cdot)$, we have $f_D(d)=F_D(d)=0$ 
if $d<\displaystyle \min_{Q \in \mathcal{S}}\;\|\overrightarrow{PQ}\|_2$ while
$f_D(d)=0$ and $F_D(d)=1$
if $d>\displaystyle \max_{Q \in \mathcal{S}}\;\|\overrightarrow{PQ}\|_2$.
For $\displaystyle \min_{Q \in \mathcal{S}}\;\|\overrightarrow{PQ}\|_2 \leq d \leq \displaystyle \max_{Q \in \mathcal{S}}\;\|\overrightarrow{PQ}\|_2$, we have
$$
F_D(d)=
\mathbb{P}(D\leq d)=\frac{\mu(\mathcal{B}(P,d) \cap \mathcal{S})}{\mu(\mathcal{S})}
$$
where $\mu(A)$ is the Lebesgue measure of the set $A$
and $\mathcal{B}(P, d)$ is the ball of center $P$ and radius $d$.
As a result, the computation of the CDF of $D$ 
amounts to a problem of computational geometry, namely
computing the Lebesgue measures  of $\mathcal{S}$
and of $\mathcal{B}(P,d) \cap \mathcal{S}$ for any $d \in \mathbb{R}_{+}$.

We consider four cases for $\mathcal{S}$, represented in Figure \ref{figintro} and denoted by (A), (B), (C), and (D)
in this figure: (A) a disk, (B) a ball, (C) a line segment, and (D) the boundary of a polyhedron.
The cases where $\mathcal{S}$ is a union of disks, a union of balls, or a union of line segments 
are straightforward extensions of cases (A), (B), and (C).
\begin{figure}
\begin{tabular}{l}
\input{Intro.pstex_t}
\end{tabular}
\caption{Different supports $\mathcal{S}$ for random variable $X$.}
\label{figintro}
\end{figure}
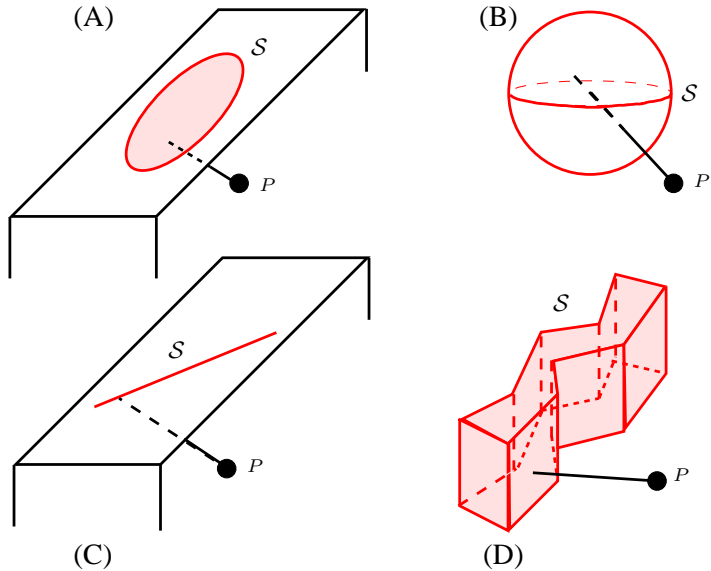

The study of these four cases is useful for Probabilistic Seismic Hazard Analysis (PSHA)
to obtain the distribution of the distance between a given location on earth and
the epicenter of an earthquake which, in a given seismic zone, is usually assumed to have a uniform distribution
in that zone modelled as a union of disks, a union of balls, a union of line segments, or the boundary of a 
polyhedron in $\mathbb{R}^3$.
This application, which motivated this study, is described in Section \ref{secapp} 
following the lines of the seminal papers \cite{cornell}, \cite{mcguire}, which paved the way for PSHA. PSHA involves several approximations and models
and therefore, as in 
\cite{imp1}, \cite{kosti2015}, \cite{loffkrev10}, \cite{imp4},\cite{imp5}, \cite{imp3}, \cite{imp2}, our algorithms perform geometric computations over inexact inputs.

In this context, the outline of the paper is as follows. 
In Section \ref{diskdist}, we consider case (A), the case where $\mathcal{S}$
is a disk. 
In Section \ref{secball} and Subsection \ref{secls}, we consider respectively
case (B), where $\mathcal{S}$ is a ball, and case (C), where $\mathcal{S}$ is a line segment.
In these three cases (A), (B), and (C), we obtain closed-form expressions for the 
CDF and the density of $D$.
The main mathematical contribution of this paper is 
Subsection \ref{secpoly1} which provides for case (D), i.e., the case where
$\mathcal{S}$ is the boundary of a polyhedron, an algorithm  with complexity $O(n \ln n)$ where 
$n$ is the number of edges of the polyhedron, that computes
exactly the CDF of $D$. An approximate density for $D$ can
then be obtained.
                        
We are not aware of other papers with these results.
However, particular cases have been discussed: in \cite{Baker2008}, cases (A) and (C) are considered taking 
for $P$ respectively the center of the disk and a point on the perpendicular bisector of the line segment.
In the recent paper \cite{stewartzhang2012}, as a particular case of (D), a rectangle is considered
for $\mathcal{S}$ while $P$ is the center of the rectangle.
In the case where $\mathcal{S}$ is the boundary of a polyhedron, to our knowledge, the current 
versions of the most popular softwares for PSHA 
(OPENQUACK \cite{openquake}, CRISIS 2012 \cite{ordaz20012}) do not compute exactly the CDF of 
$D$. For instance, CRISIS 2012 uses an approximate algorithm that performs a spatial integration subdividing the 
boundary of the polyhedron into 
small triangles.

Numerical experiments are presented in
Section \ref{sec:numexp} while
extensions of our results, in particular to 
handle the case of a general  polyhedron 
and the case where the $\ell_2$-norm is replaced by either the 
$\ell_1$-norm or the $\ell_{\infty}$-norm, are discussed in the last Section \ref{conclusion}.

Throughout the paper, we use the following notation. 
For a point $A$ in $\mathbb{R}^{3}$,
we denote its coordinates with respect to a given Cartesian coordinate system by $x_A, y_A$, and  $z_A$.
For two points $A, B \in \mathbb{R}^{3}$, $\overline{AB}$ is the line segment 
joining points $A$ and $B$, i.e., 
$\overline{AB}=\{tA + (1-t)B \;:\;t \in [0,1]\}$,
$(AB)=\{tA + (1-t)B \;:\;t \in \mathbb{R}\}$ is the line passing through
$A$ and $B$, and
$\overrightarrow{AB}$ is the vector whose
coordinates are $(x_B-x_A, y_B-y_A, z_B-z_A)$. Given two vectors $x, y \in \mathbb{R}^3$, we denote
the usual scalar product of $x$ and $y$ in $\mathbb{R}^3$ by $\langle x, y \rangle = x^\transp y$.
For $P \in \mathbb{R}^2$, we denote the circle and the disk of center $P$ and radius $d$
by respectively $\mathcal{C}(P,d)$ and $\mathcal{D}(P,d)$.

\section{Overview of the four steps of PSHA} \label{secapp}

An important problem in civil engineering is to determine the level of ground shaking
a given structure can withstand. In regions with high levels of seismic activity,
it makes sense to invest in structures able to resist high levels of ground shaking.
On the contrary, in regions without seismic activity during the structure
lifetime, we should not invest in such structures. More precisely, 
it would be reasonable to design structures able to resist up to a Peak Ground Acceleration
$A^* m.s^{-2}$ that is very rarely exceeded, say with a small probability $\varepsilon$, over a given time
window.
This approach is used in PSHA: the confidence level $\varepsilon$ and the time window being fixed (say of $t$ years), the main task of PSHA 
is to estimate
at a given location $P$, the Peak Ground Acceleration (PGA) $A^*$ such that the probability of the event
\begin{equation} \label{eventEt}
\begin{array}{lll}
 E_t(A^*, P)&=&\{\mbox{There is at least an earthquake causing a PGA}  \\
 &&  \;\;\mbox{greater than $A^*$ at }P  \mbox{ in the next }t \mbox{ years}\} \\
\end{array}
\end{equation}
is $\varepsilon$. We present the approach introduced by \cite{cornell}, \cite{mcguire},  to
model and solve this problem. In this approach, we consider the seismic zones
that could have an impact on the PGA at $P$ (see Figure \ref{figzones}
for an example of 4 zones with $P$ belonging to one of these zones).
These zones are bounded sets that do not overlap: typically disks, line segments, or
simple polygones.
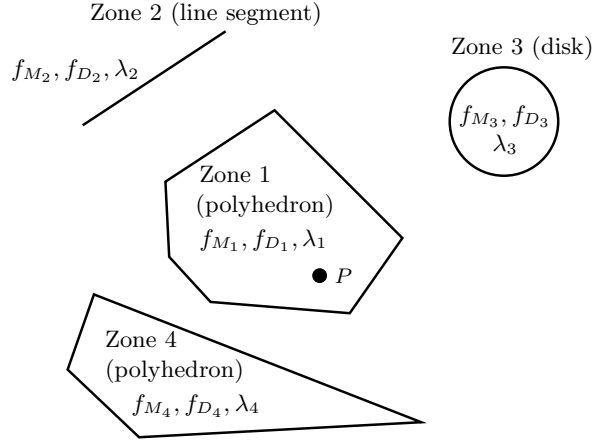
\begin{figure}
\begin{tabular}{l}
\input{Zones.pstex_t}
\end{tabular}
\caption{Seismic zones around a given point $P$.}
\label{figzones}
\end{figure}
The number of earthquakes provoking PGAs
at $P$ greater than $A^*$ over the next $t$ years depends on the frequency of earthquakes in each zone.
As for the ground acceleration at $P$ provoked by the earthquakes of a given zone,
it will depend on the magnitudes of these earthquakes, which are random, and the locations of their
epicenters, which are random too.
To take these factors into account, PSHA uses a four-step process (see Figure \ref{figzones}):
\begin{itemize}
\item[(i)] in zone $i$, the process of earthquake arrivals is modelled as  
a Poisson process with rate $\lambda_i$.
We will assume that the earthquake arrival processes in the different zones are independent.
\item[(ii)] In zone $i$, the magnitude of earthquakes is modelled as a random variable $M_i$ with density $f_{M_i}(\cdot)$.
\item[(iii)] The distance between $P$ and the epicenter of the earthquakes of zone $i$ is modelled as a random variable $D_i$ with density $f_{D_i}(\cdot)$.
\item[(iv)] A ground motion prediction model is chosen expressed as a regression of the ground acceleration
on magnitude, distance, and possibly other factors.
\end{itemize}
We now detail these steps and explain how to combine them to achieve the main task of PSHA: compute
the probability of event \eqref{eventEt} for any $A^*$. The ability to compute this probability
for any $A^*$ makes possible the estimation, by dichotomy, of an acceleration
$A^*$ satisfying $\mathbb{P}\Big(E_t(A^*, P) \Big) = \varepsilon$.

From (i), we obtain that the distribution of the number of earthquakes $N_{t i}$ in zone $i$ on a time window of $t$
time units is given by
$$
\mathbb{P}(N_{t i}=k)=e^{-\lambda_i t} \frac{(\lambda_i t)^k}{k!},\;k \in \mathbb{N},
$$
where the rate $\lambda_i$ represents the mean number of earthquakes in zone $i$
per time unit, say per year.
From now on, we fix an acceleration $A^*$ and introduce the event
\begin{equation} \label{formulapi}
\begin{array}{lll}
E(A^*, P, i)  & =  & \{ \mbox{An earthquake from zone }i \mbox{ causes  a PGA} \\
&  &  \;\;\mbox{greater than }A^* \mbox{ at }P \}
\end{array}
\end{equation}
with its probability $p_i=\mathbb{P}\Big(E(A^*, P, i)\Big)$.
For each earthquake in zone $i$, either event $E(A^*, P, i)$ occurs for this earthquake,
i.e., this earthquake causes a PGA greater than $A^*$ at $P$, or not.
As a result, we can define two new counting processes for zone $i$: the process $\tilde N_{t i}$ counting the
earthquakes causing $PGA>A^*$ at $P$ (events represented by black balls in Figure \ref{figpoisson}) and the
process counting the earthquakes causing $PGA \leq A^*$ at $P$.
\begin{figure}
\begin{tabular}{l}
\input{Poisson.pstex_t}
\end{tabular}
\caption{Splitting of the process of earthquake arrivals in zone $i$ into a
process of earthquakes causing $PGA>A^*$ at $P$ (arrivals represented by black balls) and a process of earthquakes 
causing $PGA \leq A^*$ at $P$.}
\label{figpoisson}
\end{figure}
To proceed, we need the following well-known lemma:
\begin{lemma} Consider a Poisson process $N_t$ with arrival rate $\lambda$.
Assume that arrivals are of two types I and II: type I with 
probability $p$ and type II with probability $1-p$.
We also assume that the arrival types are independent.
Then the process $\tilde N_t$ of type I arrivals is a Poisson process with rate
$\lambda p$.
\end{lemma}
\begin{proof} We compute for every $k \in \mathbb{N}$,
$$
\begin{array}{lll}
\mathbb{P}\Big(\tilde N_{t} =k \Big) & = & \displaystyle \sum_{j=k}^{+\infty} \mathbb{P}\Big(\tilde N_{t} =k | N_{t} =j\Big)  \mathbb{P} \Big( N_{t} =j  \Big)\;\;\;\mbox{[Total Probability Theorem]}\\
&=& \displaystyle \sum_{j=k}^{+\infty} C_j^k p^k (1-p)^{j-k} e^{-\lambda t} \frac{(\lambda t)^j}{j!} \\
&=&e^{-\lambda t} \frac{(\lambda p t)^k}{k!} \displaystyle \sum_{j=0}^{+\infty} \frac{\Big[\lambda (1-p) t\Big]^{j}}{j!} = e^{-\lambda p t} \frac{(\lambda p t)^k}{k!},
\end{array}
$$
which shows that $\tilde N_{t}$ is a Poisson random variable with parameter $\lambda p t$. 
We conclude using the independence of the arrival types on disjoint time windows.\hfill
\end{proof}
This lemma shows that the process $(\tilde N_{t i})_t$ is a Poisson process with rate $\lambda_i p_i$.
Denoting by $\mathcal{N}$ the number of zones, it follows that
the probability to have $k$ earthquakes causing a PGA greater than $A^*$ at $P$
over the next time window of $t$ years is
$$
\displaystyle
\begin{array}{lll}
\mathbb{P}\Big(\displaystyle \sum_{i=1}^{\mathcal{N}} \tilde N_{t i} = k\Big)
&= &\displaystyle \sum_{x_1+\ldots +x_{\mathcal{N}}=k} \mathbb{P}\Big(\tilde N_{t 1} = x_1; \ldots; \tilde N_{t \mathcal{N}} = x_{\mathcal{N}} \Big)\\
&= &\displaystyle{\sum_{x_1+\ldots +x_{\mathcal{N}}=k}  \prod_{i=1}^{\mathcal{N}}}\mathbb{P}\Big(\tilde N_{t i} = x_i \Big)\\
&= &\displaystyle \sum_{x_1+\ldots +x_{\mathcal{N}}=k} \prod_{i=1}^{\mathcal{N}} e^{-\lambda_i p_i t} \frac{(\lambda_i p_i t)^{x_i}}{x_{i}!}
\end{array}
$$
where for the second equality we have used 
the independence of  ${\tilde N}_{t 1}, \ldots, {\tilde N}_{t \mathcal{N}}$.
Taking $k=0$ in the above relation, we obtain
\begin{equation} \label{probinterest}
1-\mathbb{P}( E_t(A^*, P) )= \mathbb{P}(\overline{E_t(A^*, P)})=e^{-(\sum_{i=1}^{\mathcal{N}} \lambda_i p_i) t}.
\end{equation}
Setting $\tilde N_t=\sum_{i=1}^{\mathcal{N}} \tilde N_{t i}$,
the expectation of $\tilde N_t$
which is the mean number of earthquakes causing
a PGA greater than $A^*$ at $P$ over the next $t$ years, can be expressed as
\begin{equation} \label{formlambdatA}
\lambda_t(A^*, P)=\mathbb{E}\Big[ \tilde N_t \Big] =\sum_{i=1}^{\mathcal{N}} \mathbb{E}\Big[ \tilde N_{t i}\Big] = (\sum_{i=1}^{\mathcal{N}} \lambda_i p_i ) t.
\end{equation}
Using this relation and \eqref{probinterest}, the probability of event $E_t(A^*, P)$
can be rewritten
$$
\mathbb{P}( E_t(A^*, P) )=1-e^{-\lambda_t(A^*, P)}
$$
with $\lambda_t(A^*, P)$ given by \eqref{formlambdatA}.

It remains to explain how the probability $p_i$ of event $\eqref{formulapi}$ is computed.
This computation is based on a ground motion prediction model (step (iv) above)
which is a regression equation representing the PGA 
induced by an 
earthquake of magnitude $M$ at distance
$D$ of its epicenter. This relation takes the form
\begin{equation} \label{pgamodel}
\ln PGA = \overline{\ln PGA}(M, D, \theta) + \sigma(M, D, \theta) \varepsilon.
\end{equation}
 In this relation,
$\overline{\ln PGA}(M, D, \theta)$ (resp. $\sigma(M, D, \theta)$) is the 
conditional mean (resp. standard deviation) of $\ln PGA$ given the magnitude $M$
and distance $D$ to the epicenter while $\varepsilon$ is a standard Gaussian
random variable. 
We see that the PGA depends on the magnitude, the distance
to the epicenter and other parameters, generally referred to as $\theta$ (such 
as the ground conditions).
More precisely, the mean $\overline{\ln PGA}(M, D, \theta)$ should increase
with $M$ (the higher the magnitude, the higher the PGA)
and decrease with $D$ (the larger the distance, the lower the PGA).
As an example, the ground motion prediction model in \cite{cornell} is of the form
$$
\ln PGA = 0.152 + 0.859M - 1.803\ln (D+25)  + 0.57 \varepsilon
$$
which amounts to take 
$\overline{\ln PGA}(M, D, \theta)=0.152 + 0.859M - 1.803\ln (D+25)$
and $ \sigma(M, D, \theta)=0.57$.

The density $f_{M_i}(\cdot)$ used for the distribution of the magnitude 
of the earthquakes of
zone $i$ depends on the history of the magnitudes of the earthquakes of that zone.
For a large number of seismic zones, the density proposed by Gutenberg and Richter
\cite{gutenbergrichter} has shown appropriate. It is of the form
$$
f_{M_i}(m) = \frac{\beta_i e^{-\beta_i (m-M_{\min}(i))}}{1-e^{-\beta_i (M_{\max}(i)-M_{\min}(i))}}
$$
for some parameter $\beta_i>0$ where the support of $M_i$
is $[M_{\min}(i), M_{\max}(i)]$.

In each zone, the epicenter has a uniform distribution in that zone.
The seismic zones usually considered in PSHA are disks, balls, line segments, or
the boundary of a polyhedron. As a result, the determination of the density
$f_{D_i}(\cdot)$ of the distance $D_i$ between $P$ and the epicenter in zone $i$ can be determined
analytically or approximately using Sections \ref{diskdist}, \ref{secball}, \ref{secls}, and \ref{secpoly1}.

Gathering the previous ingredients, assuming that $D_i$ and $M_i$ are independent, and using the Total Probability Theorem, we obtain
$$
p_i= \displaystyle \int_{m_i =M_{\min}(i)}^{M_{\max}(i)}
\int_{x_i =0}^{\infty}  \mathbb{P}\Big(PGA>A^{*}| M_i = m_i; D_i = x_i \Big)f_{M_i}(m_i) f_{D_i}( x_i ) dm_i dx_i 
$$
where $\mathbb{P}\Big(PGA>A^{*}| M_i = m_i; D_i = x_i \Big)$ is given by the ground motion prediction model
\eqref{pgamodel}. For implementation purposes, the above integral is generally estimated
discretizing the continuous distributions of  magnitude $M_i, i=1,\ldots, \mathcal{N}$, and distance
$D_i, i=1,\ldots, \mathcal{N}$.

Finally, we mention the existence of an alternative, zoneless approach to PSHA introduced
by \cite{frankel} and \cite{woo}.

\section{Distance to a random variable uniformly distributed in a disk} \label{diskdist}

Let $\mathcal{S}=\mathcal{D}(S_0, R_0)$ be a disk of center $S_0$
and radius $R_0>0$ and let $P$ be a point in the plane containing $\mathcal{S}$ 
at Euclidean distance $R_1$ of $S_0$. We first consider the case where $R_1=0$.
If $0 \leq d \leq R_0$, we get $F_D(d)=\frac{\pi d^2}{\pi R_0^2}=(d/R_0)^2$ and
$f_D(d)=2 \frac{d}{R_0^2}$, if $d>R_0$ we have $F_D(d)=1$ and $f_D(d)=0$ while
if $d<0$ we have $F_D(d)=f_D(d)=0$.
Let us now consider the case where $R_1\geq R_0$.
If $d>R_1+R_0$ we have $F_D(d)=1$ and $f_D(d)=0$ while
if $d<R_1-R_0$ we have $F_D(d)=f_D(d)=0$.
Let us now take $R_1-R_0 \leq d \leq R_1+R_0$. The intersection of the disks $\mathcal{D}(S, R_0)$ and
$\mathcal{D}(P, d)$ is the union of two lenses
having a line segment $\overline{AB}$ in common (see Figures \ref{fig0} and \ref{fig1}).
\begin{figure}
\begin{tabular}{l}
\input{Lens.pstex_t}
\end{tabular}
\caption{Lenses of height $h$ in a disk of radius $R$.}
\label{fig0}
\end{figure}
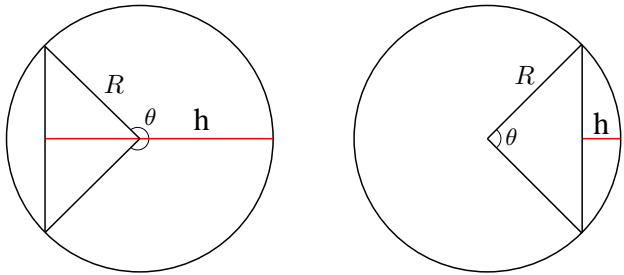
\begin{figure}
\begin{tabular}{l}
\input{Circle_Circle.pstex_t}
\end{tabular}
\caption{Random variable $X$ uniformly distributed in a ball of radius $R_0$ and center $S_0$. Case where $R_1\geq R_0>0$.}
\label{fig1}
\end{figure}
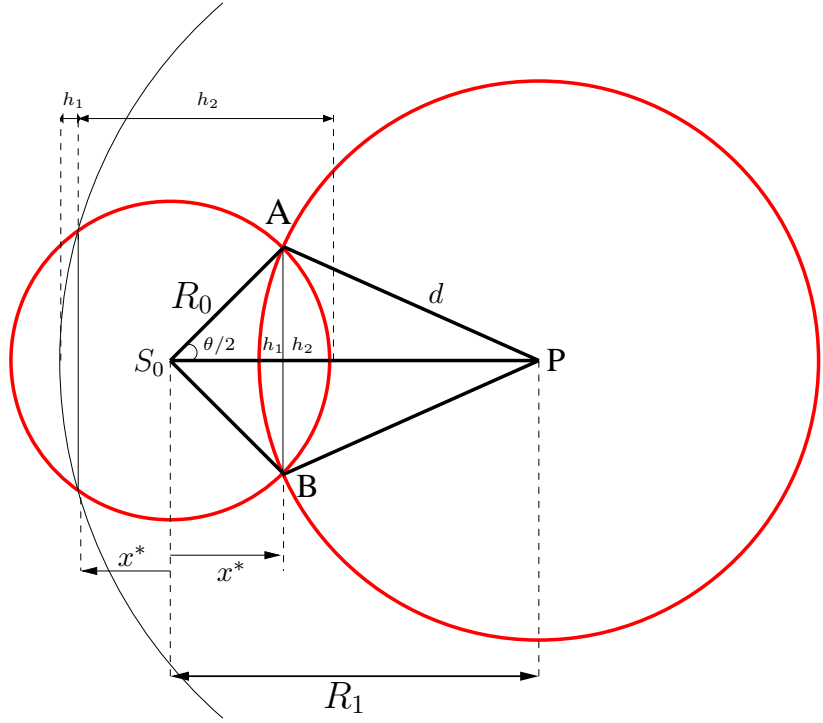
Without loss of generality, assume that $(S_0 P)$ is the $x$-axis
and that the equations of the boundaries of the disks
are given by $x^2 + y^2=R_0^2$ and $(x-R_1)^2 + y^2=d^2$.
From these equations, we obtain that the abscissa of the intersection points $A$ and $B$
of the boundaries of the disks is $x^*=\frac{R_0^2 + R_1^2 - d^2}{2 R_1}$. Note that $A=B$
if and only if $d=R_1\pm R_0$. In Figure \ref{fig1}, we represented
a situation where $x^* \geq 0$ and a situation where $x^*<0$. 
In both cases, $\mathcal{D}(S_0, R_0) \cap \mathcal{D}(P, d)$  
is the union of a lens of height $h_1(d)$ in a disk of radius $d$ (the disk $\mathcal{D}(P, d)$)
and of a lens of height $h_2(d)$ in a disk of radius $R_0$ (the disk $\mathcal{D}(S_0, R_0)$)
where
\begin{equation} \label{h1h2def}
\begin{array}{lll}
h_1(d)  & = & d-R_1+ x^* = d-R_1 + \frac{R_0^2 + R_1^2 - d^2}{2 R_1} \mbox{ and}\\
h_2(d)  & = & R_0 - x^* = R_0 - \frac{R_0^2 + R_1^2 - d^2}{2 R_1}.
\end{array}
\end{equation}
Recall that the area $\mathbb{A}(R, h)$ of a lens of height $h$ contained in a disk of radius
$R$ (see Figure \ref{fig0}) is 
$\mathbb{A}(R, h)=R^2 \frac{\theta}{2}-R^2 \sin(\frac{\theta}{2})\cos (\frac{\theta}{2})$ with 
$\cos(\frac{\theta}{2})=\frac{R-h}{R}$, i.e., 
\begin{equation} \label{formulalens}
\mathbb{A}(R, h)=R^2 \mbox{Arccos}\left( \frac{R-h}{R}\right)-(R-h) \sqrt{R^2 - (R-h)^2}.
\end{equation}
In the sequel, we will denote by $\mathcal{A}(\mathbb{S})$ the area of a surface
$\mathbb{S}$. With this notation, it follows that
\begin{equation} \label{areaintersection}
\mathcal{A}( \mathcal{D}(S_0, R_0) \cap \mathcal{D}(P, d) ) = \mathbb{A}(d, h_1(d)) + \mathbb{A}(R_0, h_2(d))
\end{equation}
where
\begin{equation}\label{firstlens}
\mathbb{A}(d, h_1(d))=d^2 \mbox{Arccos}\left(\frac{d^2 + R_1^2 -R_0^2 }{2 R_1 d} \right)- \frac{d^2 + R_1^2 - R_0^2}{2 R_1} \sqrt{d^2 -\left(\frac{d^2 + R_1^2 - R_0^2}{2 R_1}\right)^2}
\end{equation}
and
\begin{equation} \label{secondlens}
\mathbb{A}(R_0, h_2(d))=R_0^2 \mbox{Arccos}\left(\frac{R_0^2 + R_1^2 -d^2 }{2 R_0 R_1} \right)- \frac{R_0^2 + R_1^2 - d^2}{2 R_1} \sqrt{R_0^2 -\left(\frac{R_0^2 + R_1^2 - d^2}{2 R_1}\right)^2}.
\end{equation}
For $R_1-R_0 \leq d \leq R_1+R_0$, we obtain $F_D(d)=\frac{\mathbb{A}(d, h_1(d))+\mathbb{A}(R_0, h_2(d))}{\pi R_0^2}$ where
$\mathbb{A}(d, h_1(d))$ and $\mathbb{A}(R_0, h_2(d))$ are given by 
\eqref{firstlens} and \eqref{secondlens}. The density is
\begin{equation} \label{formuladensD}
f_D(d)=\frac{1}{\pi R_0^2}\Big[ h_2'(d) \frac{\partial \mathbb{A}(R_0, h_2(d))}{\partial h}  + \frac{\partial \mathbb{A}(d, h_1(d))}{\partial R}  + h_1'(d) \frac{\partial \mathbb{A}(d, h_1(d))}{\partial h}\Big]
\end{equation}
where $h_1'(d)=1-\frac{d}{R_1}$, $h_2'(d)=\frac{d}{R_1}$, and
\begin{equation} \label{derivativearea}
\begin{array}{lll}
\displaystyle \frac{\partial \mathbb{A}(R,h)}{\partial R} &  = & 2R \mbox{Arccos}\left(1-\frac{h}{R} \right) - 2 \sqrt{h(2R-h)}, \vspace*{0.1cm}\\
\displaystyle \frac{\partial \mathbb{A}(R,h)}{\partial h} &  = &  2 \sqrt{h(2R-h)}.
\end{array}
\end{equation}
We now consider the case where $R_1>0$ and $R_1<R_0$ (see Figure \ref{fig2}). 
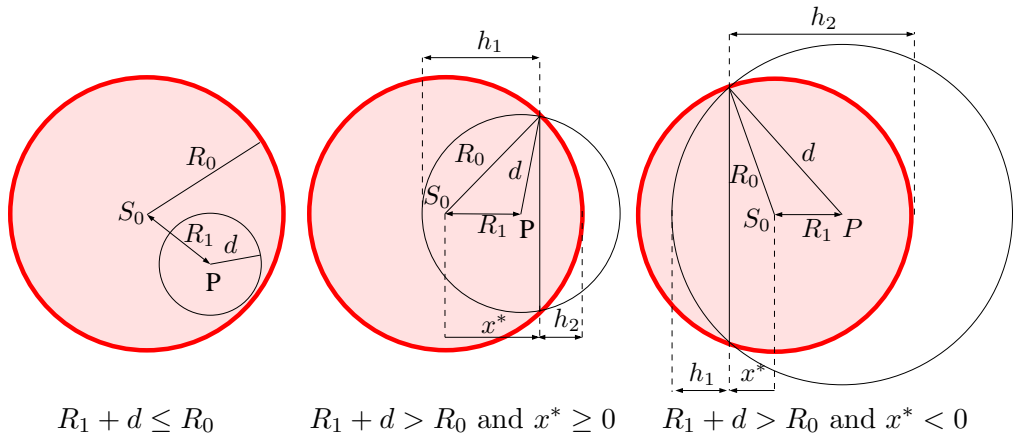
\begin{figure}
\begin{tabular}{l}
\input{Circle_Circle_1.pstex_t}
\end{tabular}
\caption{Random variable $X$ uniformly distributed in a ball of radius $R_0$ and center $S_0$. Case where $0 < R_1 < R_0$.}
\label{fig2}
\end{figure}
If $0 \leq d \leq R_0 - R_1$, we obtain $F_D(d)=\frac{\pi d^2}{\pi R_0^2}=\frac{d^2}{R_0^2}$,
if $d<0$ we have $F_D(d)=f_D(d)=0$ while
if $d>R_0+R_1$ we have $F_D(d)=1$ and $f_D(d)=0$ (see Figure \ref{fig2}). If $R_1 + R_0 \geq  d>R_0-R_1$, both in the case where 
the abscissa $x^*$ of the intersection
points between the boundaries of  $\mathcal{D}(S_0, R_0)$ and $\mathcal{D}(P, d)$  is positive and negative,
we check (see Figure \ref{fig2}) that the area of $\mathcal{D}(S_0, R_0) \cap \mathcal{D}(P, d)$ is still given by 
\eqref{areaintersection} with $\mathbb{A}(d, h_1(d))$ and $\mathbb{A}(R_0, h_2(d))$ 
given respectively by \eqref{firstlens} and \eqref{secondlens}.
Summarizing, if $0<R_1<R_0$ then if $R_1 + R_0 \geq  d>R_0-R_1$, the density of 
$D$ at $d$ is given by \eqref{formuladensD} and if $0 \leq d \leq R_0 - R_1$, we have
$f_D(d)=\frac{2 d}{R_0^2}$.
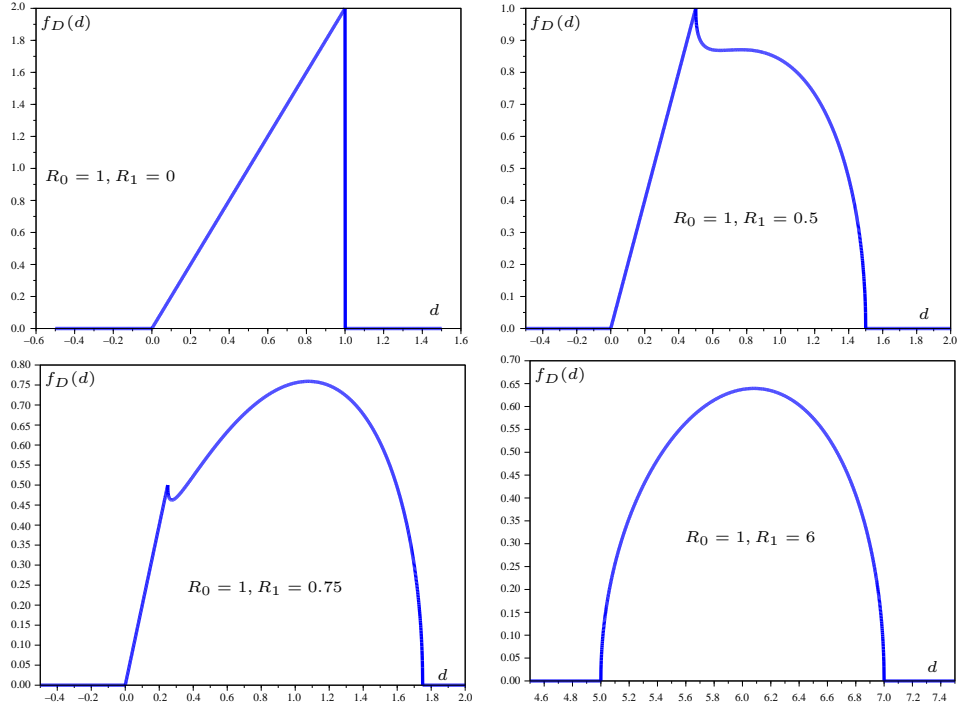
\begin{figure}
\begin{tabular}{ll}
\input{Disk1.pstex_t} &  \input{Disk2.pstex_t} \\
\input{Disk3.pstex_t} &  \input{Disk4.pstex_t} 
\end{tabular}
\caption{Density of $D$  where $X$ is uniformly distributed in a disk of radius $R_0=1$: some examples.
Top left: $R_1=0$, top right: $R_1=0.5$, bottom left: $R_1=0.75$, bottom right: $R_1=6$.} 
\label{figuredenscircle}
\end{figure} 
The density of $D$  when $X$ is uniformly distributed in a disk is given 
for some examples in Figure \ref{figuredenscircle}.

Finally, we consider the case where $\mathcal{S}$ is a disk $\mathcal{D}$ and
$P \in \mathbb{R}^3$ is not contained in the plane $\mathcal{P}$ containing 
this disk. Let $S_0$ be the center of $\mathcal{S}$ and let $S_1, S_2$ be two points 
of the boundary of the disk such that $\overrightarrow{S_0 S_1}$ and $\overrightarrow{S_0 S_2}$
are linearly independent. We introduce the projection
$P_0=\pi_{\mathcal{P}}[P]=\argmin_{Q \in \mathcal{P}} \|\overrightarrow{PQ}\|_2$ of $P$ onto $\mathcal{P}$.
Since vectors $\overrightarrow{S_0 S_1}$ and $\overrightarrow{S_0 S_2}$
are linearly independent, if $A$ is the  $(3,2)$ matrix $[\overrightarrow{S_0 S_1}, \overrightarrow{S_0 S_2}]$ whose first
column is $\overrightarrow{S_0 S_1}$ and whose second column is $\overrightarrow{S_0 S_2}$, then the matrix
$A^\transp A$ is invertible. It follows that the projection $P_0=\pi_{\mathcal{P}}[P]$ of $P$ onto $\mathcal{P}$ can be expressed
as $\overrightarrow{S_0 P_0}=A (A^\transp A)^{-1} A^\transp \overrightarrow{S_0 P}$.
With this notation, the intersection of $\mathcal{P}$ and the ball $\mathcal{B}(P, d)$ of center $P$
and radius $d$ is either empty or it is a disk of center $P_0$ and radius
\begin{equation}\label{formulard}
 R(d)=\sqrt{d^2-\|\overrightarrow{P P_0}\|_2^2} 
\end{equation}
(see Figure \ref{figdiskR3}).
\begin{figure}
\begin{tabular}{l}
\input{Circle_R3.pstex_t}
\end{tabular}
\caption{Euclidean distance to a point uniformly distributed in a disk.}
\label{figdiskR3} 
\end{figure}
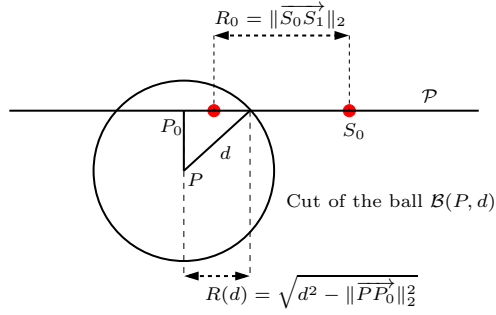
In the latter case, denoting this disk by $\mathcal{D}(P_0, R(d))$ and using the fact that
$\mathcal{D} = \mathcal{D} \cap \mathcal{P}$ (recall that $\mathcal{D} \subset \mathcal{P}$), we obtain
$$
\mathcal{D} \cap \mathcal{B}(P, d)= \mathcal{D} \cap \mathcal{P} \cap \mathcal{B}(P, d)=  \mathcal{D} \cap  \mathcal{D}(P_0, R(d)).
$$
Since $\mathcal{D}$ and  $\mathcal{D}(P_0, R(d))$ are disks contained in the plane $\mathcal{P}$, 
setting $R_0=\|\overrightarrow{S_0 S_1}\|_2$ and $R_1=\|\overrightarrow{S_0 P_0}\|_2$,
the previous results provide
the area of their intersection and the following CDFs and densities for $D$:\\
\par {\textbf{Case where $P_0=S_0$:}} The CDF and density of $D$ are given by
$$
\left\{
\begin{array}{ll}
F_D(d)=f_D(d)=0 & \mbox{if }d<\|\overrightarrow{S_0 P}\|_2,\\
\left\{
\begin{array}{l}F_D(d)=\frac{R(d)^2}{\|\overrightarrow{S_0 S_1}\|_2^2} = \frac{d^2 - \|\overrightarrow{S_0 P}\|_2^2}{\|\overrightarrow{S_0 S_1}\|_2^2}\\
f_D(d) = \frac{2 d}{\|\overrightarrow{S_0 S_1}\|_2^2}
\end{array}\right\}
  & \mbox{if }\|\overrightarrow{S_0 P}\|_2 \leq d \leq \sqrt{\|\overrightarrow{S_0 P}\|_2^2 + \|\overrightarrow{S_0 S_1}\|_2^2},\\
F_D(d)=1 \mbox{ and }f_D(d)=0 & \mbox{if }d>\sqrt{\|\overrightarrow{S_0 P}\|_2^2 + \|\overrightarrow{S_0 S_1}\|_2^2}.
\end{array}
\right.
$$
\par {\textbf{Case where $0<\|\overrightarrow{S_0 P_0}\|_2 < \|\overrightarrow{S_0 S_1}\|_2$:}} Setting
\begin{equation}\label{fordminmax}
\begin{array}{l}
d_{\min}=\sqrt{\|\overrightarrow{P P_0}\|_2^2 + (\|\overrightarrow{S_0 S_1}\|_2 - \|\overrightarrow{S_0 P_0}\|_2)^2}\mbox{ and }\\
d_{\max}=\sqrt{\|\overrightarrow{P P_0}\|_2^2 + (\|\overrightarrow{S_0 S_1}\|_2 + \|\overrightarrow{S_0 P_0}\|_2 )^2},
\end{array}
\end{equation}
the CDF of $D$ is given by 
\begin{equation} \label{cdfexpressionD}
\left\{
\begin{array}{lll}
(a) & F_D(d)=0 & \mbox{if }d<\|\overrightarrow{P P_0}\|_2,\\
(b) & F_D(d)=\frac{R(d)^2}{\|\overrightarrow{S_0 S_1}\|_2^2}=\frac{d^2 - \|\overrightarrow{P P_0}\|_2^2}{\|\overrightarrow{S_0 S_1}\|_2^2} & \mbox{if }\|\overrightarrow{P P_0}\|_2 \leq d \leq d_{\min},\\
(c) & F_D(d)=\frac{\mathbb{A}(R(d), h_1(R(d))) +  \mathbb{A}(\|\overrightarrow{S_0 S_1}\|_2, h_2(R(d)))}{\pi \|\overrightarrow{S_0 S_1}\|_2^2} & \mbox{if } d_{\min} \leq d \leq d_{\max},\\
(d) & F_D(d)=1 & \mbox{if }d>d_{\max},
\end{array}
\right.
\end{equation}
where the expression of $\mathbb{A}$ is given by 
\eqref{formulalens} and, where, using the expressions of $h_1$ and $h_2$ and recalling that $R_0=\|\overrightarrow{S_0 S_1}\|_2$ and $R_1=\|\overrightarrow{S_0 P_0}\|_2$,
\begin{equation}\label{newh1h2}
\begin{array}{lll}
h_1(R(d)) & = & \sqrt{d^2-\|\overrightarrow{P P_0}\|_2^2} - \|\overrightarrow{S_0 P_0}\|_2+ \frac{\|\overrightarrow{S_0 S_1}\|_2^2 +  \|\overrightarrow{S_0 P_0}\|_2^2 + \|\overrightarrow{P P_0}\|_2^2 - d^2}{2 \|\overrightarrow{S_0 P_0}\|_2},\\
h_2(R(d)) & = & \|\overrightarrow{S_0 S_1}\|_2 - \frac{\|\overrightarrow{S_0 S_1}\|_2^2 +  \|\overrightarrow{S_0 P_0}\|_2^2 + \|\overrightarrow{P P_0}\|_2^2 - d^2}{2 \|\overrightarrow{S_0 P_0}\|_2}.
\end{array}
\end{equation}
It follows that $f_D(d)=0$ if $d<\|\overrightarrow{P P_0}\|_2$ or $d>d_{\max}$ while
$f_D(d)=\frac{2 d}{\|\overrightarrow{S_0 S_1}\|_2^2}$ if $\|\overrightarrow{P P_0}\|_2 \leq d \leq d_{\min}$.
Finally, if  $d_{\min} \leq d \leq d_{\max}$, we have
\begin{equation} \label{formuladensD1}
\begin{array}{lll}
f_D(d)&  =  &\frac{1}{\pi \|\overrightarrow{S_0 S_1}\|_2^2}\left[\frac{d}{\|\overrightarrow{S_0 P_0}\|_2}
\frac{\partial \mathbb{A}(\|\overrightarrow{S_0 S_1}\|_2, h_2(R(d)))}{\partial h}  + 
\frac{d}{\sqrt{d^2-\|\overrightarrow{P P_0}\|_2^2}}  \frac{\partial \mathbb{A}(R(d), h_1(R(d)))}{\partial R} \right]\\
&&+ \frac{d}{\pi \|\overrightarrow{S_0 S_1}\|_2^2} \left(\frac{1}{\sqrt{d^2-\|\overrightarrow{P P_0}\|_2^2}} -   \frac{1}{\|\overrightarrow{S_0 P_0}\|_2} \right) \frac{\partial \mathbb{A}(R(d), h_1(R(d)))}{\partial h}
\end{array}
\end{equation}
where the expressions of $\frac{\partial \mathbb{A}(R, h)}{\partial R}$ and
$\frac{\partial \mathbb{A}(R, h)}{\partial h}$ are given by \eqref{derivativearea}.\\

\par {\textbf{Case where $\|\overrightarrow{S_0 P_0}\|_2 \geq \|\overrightarrow{S_0 S_1}\|_2$:}} 
With the definitions \eqref{fordminmax} of $d_{\min}$ and $d_{\max}$,
if $d<d_{\min}$ then $F_D(d)=f_D(d)=0$, if $d>d_{\max}$ then $F_D(d)=1$ and $f_D(d)=0$,
while if $d_{\min} \leq d \leq d_{\max}$, $f_D(d)$ is given by \eqref{formuladensD1}
and $F_D(d)$ is given by 
\eqref{cdfexpressionD}-(c) with $h_1(R(d))$ and $h_2(R(d))$
given by \eqref{newh1h2}.

\section{Distance to a random variable uniformly distributed in a ball}\label{secball}

Let $\mathcal{S}=\mathcal{B}(S_0, R_0)$ be
a ball of radius $R_0>0$  and center $S_0$ in $\mathbb{R}^3$ and let $P$ be at
Euclidean distance $R_1$ of $S_0$. 
The computations are identical to those of the previous section
replacing two dimensional lenses and disks by three dimensional caps and
balls. If $R_1=0$ then if $d>R_0$, we have
$f_D(d)=0$ and $F_D(d)=1$, if $d<0$, we have
$f_D(d)=F_D(d)=0$ while
if $0 \leq d \leq R_0$, we obtain
$F_D(d)=\frac{(4/3)\pi d^3}{(4/3)\pi R_0^3}$, i.e.,
$f_D(d)=3\frac{d^2}{R_0^3}$ (see Figure \ref{fig2}). 
If $0<R_1<R_0$, then if $d>R_0+R_1$, we have
$F_D(d)=1$ and $f_D(d)=0$, if $d<0$, we have
$F_D(d)=f_D(d)=0$ while
if $0 \leq d \leq R_0-R_1$, we have
$F_D(d)=\frac{(4/3)\pi d^3}{(4/3)\pi R_0^3}$, i.e.,
$f_D(d)=3\frac{d^2}{R_0^3}$ (see Figure \ref{fig2}). 
If $R_1\geq R_0$ then if $d> R_0 + R_1$, we have
$f_D(d)=0$ and $F_D(d)=1$ and if $d<R_1-R_0$, we have
$f_D(d)=F_D(d)=0$. If $0<R_1<R_0$ and 
$R_0-R_1<d\leq R_0+R_1$ or 
if $R_1 \geq R_0$ and $R_1-R_0 \leq d \leq R_1+R_0$,
then 
$\mathcal{B}(S_0, R_0) \cap \mathcal{B}(P, d)$  
is the union of a spherical cap of height $h_1(d)$ contained in a ball of radius $d$ (the ball $\mathcal{B}(P, d)$)
and of a spherical cap of height $h_2(d)$ contained in a ball of radius $R_0$ (the ball $\mathcal{B}(S_0, R_0)$)
where the expressions \eqref{h1h2def} for $h_1(d)$ and $h_2(d)$ are still valid.
Now recall that the volume of a spherical cap (see Figure \ref{fig0} for a cut of this cap) of height $h$
contained in a ball of radius $R$ in $\mathbb{R}^3$ is
\begin{equation}\label{sphcapvol}
\mathbb{V}(R, h)=\int_{x=R-h}^R \pi r^2(x)dx=\int_{x=R-h}^R \pi [R^2 - x^2]dx=\frac{\pi h^2}{3}(3R-h).
\end{equation}
It follows that if $0<R_1<R_0$ and 
$R_0-R_1<d\leq R_0+R_1$ or 
if $R_1 \geq R_0$ and $R_1-R_0 \leq d \leq R_1+R_0$, we have
$$
\begin{array}{lll}
F_D(d)&=&\frac{3}{4 \pi R_0^3}\left[\mathbb{V}(d, h_1(d))  + \mathbb{V}(R_0, h_2(d)) \right]\vspace*{0.2cm}\\
&=& \frac{1}{4 R_0^3}\Big[ h_1^2(d)(3d-h_1(d)) +  h_2^2(d)(3R_0-h_2(d))     \Big]
\end{array}
$$
where we recall that $h_1(d)$ and $h_2(d)$ are given by \eqref{h1h2def} and the density is
\begin{eqnarray*}
f_D(d) & = & 
\frac{3}{4 \pi R_0^3}\left[ \frac{\partial \mathbb{V}}{\partial R}(d, h_1(d)) + h_1'(d) \frac{\partial \mathbb{V}}{\partial h}(d, h_1(d)) + h_2'(d) \frac{\partial \mathbb{V}}{\partial h}(R_0, h_2(d)) \right]
\end{eqnarray*}
where 
$$
\displaystyle \frac{\partial \mathbb{V}(R,h)}{\partial R}   =  \pi h^2 \mbox{ and }
\displaystyle \frac{\partial \mathbb{V}(R,h)}{\partial h}   =  \pi h (2R - h).
$$
The density of $D$  when $X$ is uniformly distributed in a ball is given 
for some examples in Figure \ref{figuredensball}.
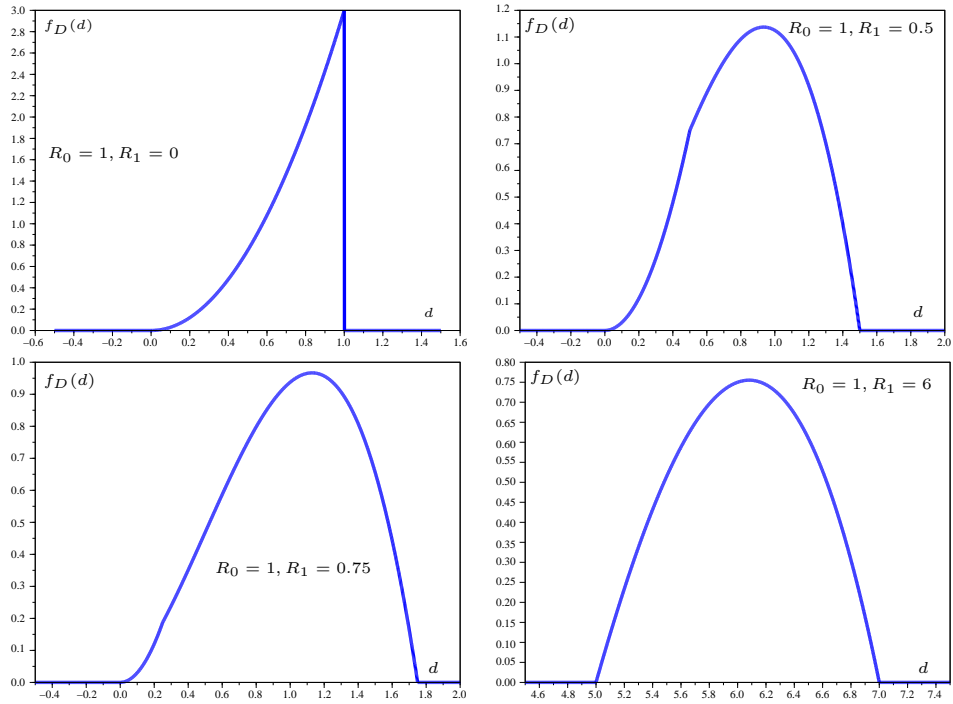
\begin{figure}
\begin{tabular}{ll}
\input{Ball1.pstex_t} &  \input{Ball2.pstex_t} \\
\input{Ball3.pstex_t} &  \input{Ball4.pstex_t} 
\end{tabular}
\caption{Density of $D$ when $X$ is uniformly distributed in a ball of radius $R_0=1$: some examples.
Top left: $R_1=0$, top right: $R_1=0.5$, bottom left: $R_1=0.75$, bottom right: $R_1=6$.} 
\label{figuredensball}
\end{figure} 

\section{Distance to a random variable uniformly distributed in a polygone}\label{secpoly}

\subsection{Distance to a random variable uniformly distributed on a line segment}\label{secls}

Let $\mathcal{S}=\overline{AB}$ be a line segment in $\mathbb{R}^{3}$ 
with $A \neq B$ and let $P \in \mathbb{R}^3$.
We introduce the projection $P_0$ of $P$ onto line $(AB)$:
$$
P_0 = A + \frac{\langle \overrightarrow{AB}, \overrightarrow{AP} \rangle }{\|\overrightarrow{AB}\|_2^2} \overrightarrow{AB}.
$$
This projection $P_0$ belongs to line segment $\overline{AB}$ if and only if 
$\langle \overrightarrow{P_0 A}, \overrightarrow{P_0 B}\rangle  \leq 0$ (see Figure \ref{lineexamples1}).
\begin{figure}
\begin{center} 
\begin{tabular}{l}
\input{Line_Cases.pstex_t}  
\end{tabular}
\caption{Distance to a random variable uniformly distributed on a line segment.} 
\label{lineexamples1}
\end{center} 
\end{figure}
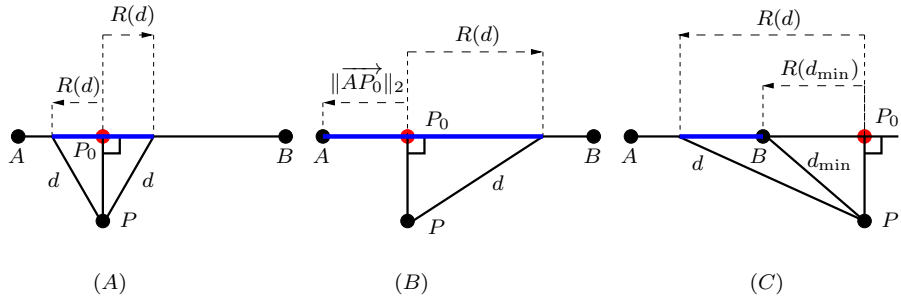 
In this case, 
setting $d_{\min}=\min(\|\overrightarrow{P A}\|_2, \|\overrightarrow{P B}\|_2)$
and $d_{\max}=\max(\|\overrightarrow{P A}\|_2, \|\overrightarrow{P B}\|_2)$,
we obtain the following CDF for $D$ (see Figure \ref{lineexamples1}):
\begin{equation} \label{cdfexpressionDls}
\left\{
\begin{array}{ll}
F_D(d)=0 & \mbox{if }d<\|\overrightarrow{P P_0}\|_2,\\
F_D(d)=\frac{2 R(d)}{\|\overrightarrow{A B}\|_2}=\frac{2 \sqrt{d^2 - \|\overrightarrow{P P_0}\|_2^2}}{\|\overrightarrow{A B}\|_2} & \mbox{if }\|\overrightarrow{P P_0}\|_2 \leq d \leq d_{\min},\\
\begin{array}{lll}
F_D(d)&=&\frac{ \min(\|\overrightarrow{P_0 A}\|_2, \|\overrightarrow{P_0 B}\|_2) +R(d)    }{\|\overrightarrow{A B}\|_2}\\
&=& \frac{\min(\|\overrightarrow{P_0 A}\|_2, \|\overrightarrow{P_0 B}\|_2)   + \sqrt{d^2 - \|\overrightarrow{P P_0}\|_2^2}    }{\|\overrightarrow{A B}\|_2}
\end{array} & \mbox{if } d_{\min} \leq d \leq d_{\max},\\
F_D(d)=1 & \mbox{if }d>d_{\max}.
\end{array}
\right.
\end{equation}
If $P_0$ does not belong to $\overline{AB}$, i.e., if 
$\langle \overrightarrow{P_0 A}, \overrightarrow{P_0 B}\rangle > 0$,
we obtain the following CDF for $D$ (see Figure \ref{lineexamples1}):
\begin{equation} \label{cdfexpressionDls2}
\left\{
\begin{array}{ll}
F_D(d)=0 & \mbox{if }d<d_{\min},\\
F_D(d)=\frac{R(d)-R(d_{\min})}{\|\overrightarrow{A B}\|_2}= \frac{\sqrt{d^2 - \|\overrightarrow{P P_0}\|_2^2}-\sqrt{d_{\min}^2 - \|\overrightarrow{P P_0}\|_2^2}}{\|\overrightarrow{A B}\|_2} & \mbox{if } d_{\min} \leq d \leq d_{\max},\\
F_D(d)=1 & \mbox{if }d>d_{\max}.
\end{array}
\right.
\end{equation}
An analytic expression of the density can be obtained deriving the above CDF.
The density of $D$  when $X$ is uniformly distributed in a line segment  is given 
for two examples in Figure \ref{figuredenssegment}.
\begin{figure}
\begin{tabular}{ll}
\input{Line1.pstex_t} &  \input{Line2.pstex_t} 
\end{tabular}
\caption{Density of $D$ when $X$ is uniformly distributed on a line segment 
$\overline{AB}$: some examples.} 
\label{figuredenssegment}
\end{figure}
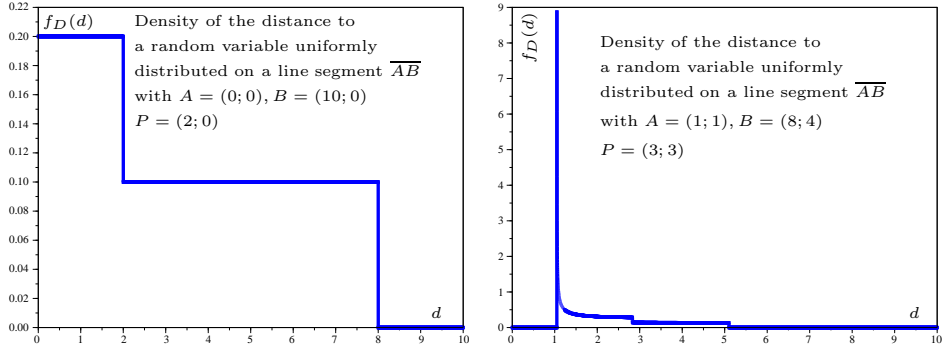 

\subsection{Simple polygone}\label{secpoly1}

Let $\mathcal{S}$ be a simple polygone contained in a plane
given by its extremal points $\{S_1, S_2, \ldots, S_n\}$
where the boundary of $\mathcal{S}$ is $\cup_{i=1}^{n} \overline{S_i S_{i+1}}$
with the convention that $S_{n+1}=S_1$ 
and where $S_i \neq S_j$ for $i \neq j$ with $1 \leq i,j \leq n$.
We assume that when travelling on the boundary of $\mathcal{S}$ from $S_1$ to $S_2$, 
then from $S_2$ to $S_3$ and so on until the last line segment $\overline{S_n S_1}$, one always has 
the relative interior of $\mathcal{S}$ to the left (see Figure \ref{fig4}).
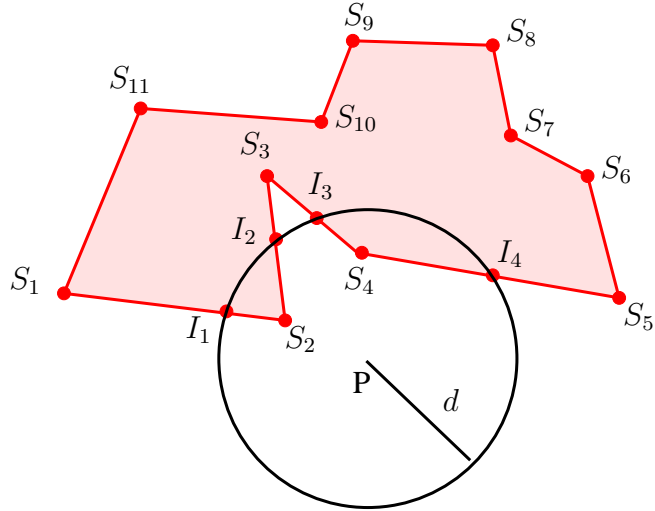
\begin{figure}
\begin{tabular}{l}
\input{Polyhedron_Source.pstex_t}
\end{tabular}
\caption{Random variable uniformly distributed in a polyhedron in the plane.}
\label{fig4}
\end{figure}
Let $P$ be a point in the plane $\mathcal{P}$ containing $\mathcal{S}$.
$F_D(d)$ is the area of the intersection of $\mathcal{S}$ and the disk  $\mathcal{D}(P, d)$ of center
$P$ and radius $d$ divided by the area of $\mathcal{S}$. These areas will be computed making use
of a special case of Green's theorem: if $\mathbb{D}$ is a closed and bounded region in the plane then
the area $\mathcal{A}(\mathbb{D})$ of $\mathbb{D}$ can be expressed as a line integral over the boundary
$\partial \mathbb{D}$ of $D$:
\begin{equation} \label{green}
\mathcal{A}(\mathbb{D})=\frac{1}{2} \oint_{\partial \mathbb{D}} [x dy - y dx].
\end{equation}
Since the boundary of $\mathcal{S}$ is a union of line segments and the boundary of $\mathcal{S} \cap \mathcal{D}(P, d)$
is made of line segments and arcs, we need to compute $\int_{C} [x dy - y dx]$
with $C$ a line segment or an arc. If $C=\overline{AB}$ is a line segment, denoting respectively
the coordinates of $A$ and $B$ by $(x_A, y_A)$ and $(x_B, y_B)$, we obtain
\begin{equation} \label{greensegment}
\mathcal{I}_{\overline{AB}}:=\int_{\overline{AB}} [x dy - y dx]= y_B x_A - y_A x_B.
\end{equation}
Now let $C=\displaystyle \stackrel{\frown}{AB}_{R_0, P} $ be an arc starting at $A=(x_A, y_A)$ and 
ending at $B=(x_B, y_B)$ with $A$ and $B$
belonging to the circle of center $P=(x_P, y_P)$ and radius $R_0>0$.
We assume that when travelling along the arc from $A$ to $B$, the relative interior of the disk is to the left.
If $\theta(A, B)$ is the angle $\angle APB$, using \eqref{green} we obtain
$$
\frac{R_0^2 \theta(A,B) }{2} = \frac{1}{2}\Big(  \mathcal{I}_{\stackrel{\frown}{AB}_{R_0, P}} + \mathcal{I}_{\overline{PA}} + \mathcal{I}_{\overline{BP}} \Big)
$$
where $\mathcal{I}_{\stackrel{\frown}{AB}_{R_0, P}}:= \displaystyle \int_{\stackrel{\frown}{AB}_{R_0, P}} [x dy - y dx]$.
Using \eqref{greensegment}, the above relation can be written
\begin{equation} \label{greenarc1}
\mathcal{I}_{\stackrel{\frown}{AB}_{R_0, P}} = R_0^2 \theta(A, B) + x_P (y_B-y_A) -y_P(x_B-x_A). 
\end{equation}
We introduce the function ${\tt{Angle}}$ defined on the boundary of $\mathcal{D}(P, d)$ 
taking values in $[0, 2\pi[$ and given by
\begin{equation} \label{polarangle}
\begin{array}{lll}
{\tt{Angle}}(x,y) & =  & \mbox{Arccos}\left(\frac{x-x_P}{R_0}\right) \mbox{ if }y \geq y_P \mbox{ and}\\
{\tt{Angle}}(x,y) & =  & 2\pi - \mbox{Arccos}\left(\frac{x-x_P}{R_0}\right) \mbox{ if }y<y_P.
\end{array}
\end{equation}
This function associates to a point of the boundary of $\mathcal{D}(P, d)$ its angle.
With this notation, for two points $A=(x_A, y_A)$ and $B=(x_B, y_B)$ of the boundary of $\mathcal{D}(P,d)$, we have
$$
\begin{array}{lll}
\theta(A,B) & =  & {\tt{Angle}}(x_B, y_B)-{\tt{Angle}}(x_A, y_A) \mbox{ if }{\tt{Angle}}(x_A, y_A) \leq {\tt{Angle}}(x_B, y_B)\\
\theta(A,B) & =  & 2\pi+ {\tt{Angle}}(x_B, y_B)-{\tt{Angle}}(x_A, y_A) \mbox{ otherwise }
\end{array}
$$
and formula \eqref{greenarc1} can be written
\begin{equation} \label{greenarc}
\left\{
\begin{array}{l}
\mathcal{I}_{\stackrel{\frown}{AB}_{R_0, P}} =  R_0^2 ({\tt{Angle}}(x_B, y_B)-{\tt{Angle}}(x_A, y_A)) + x_P (y_B-y_A) -y_P(x_B-x_A)\\
\mbox{ if }{\tt{Angle}}(x_A, y_A) \leq {\tt{Angle}}(x_B, y_B) \mbox{ and }\\
\mathcal{I}_{\stackrel{\frown}{AB}_{R_0, P}} =  R_0^2 (2 \pi + {\tt{Angle}}(x_B, y_B)-{\tt{Angle}}(x_A, y_A)) + x_P (y_B-y_A)\\
\hspace*{1.8cm}-y_P(x_B-x_A)\mbox{ if }{\tt{Angle}}(x_A, y_A) > {\tt{Angle}}(x_B, y_B).
\end{array}
\right.
\end{equation}
To compute the area of the intersection $\mathcal{S} \cap \mathcal{D}(P, d)$, we need
to determine the intersections between the boundary of $\mathcal{S}$ and the circle $\mathcal{C}(P, d)$
of center $P$ and radius $d$.
This will be done using Algorithm 1 which computes the intersection between a given line segment 
$\overline{AB}$ with $A \neq B$ and the
sphere of center $P$ and radius $d$ in $\mathbb{R}^3$.
When this intersection is nonempty, let $I_1(d)$ and $I_2(d)$ be the intersection points (eventually $I_1(d)=I_2(d)$). 
Writing $I_i(d)$ as 
\begin{equation}\label{formulaIi}
I_i(d)=A+t_i \overrightarrow{AB},
\end{equation}
$t_i$ solves
$\|\overrightarrow{PA}+t_i \overrightarrow{AB}\|_2^2 = d^2$.
Introducing 
\begin{equation} \label{formulatidelta1}
\Delta=\langle \overrightarrow{PA}, \overrightarrow{AB} \rangle^2 - \|\overrightarrow{AB}\|_2^2 (\|\overrightarrow{PA}\|_2^2 - d^2),
\end{equation}
if $\Delta<0$ then the boundary of $\mathcal{S}$  and $\mathcal{C}(P, d)$ have an empty intersection while if
$\Delta \geq 0$ the intersections $I_1(d)$ and $I_2(d)$
are given by \eqref{formulaIi} where
\begin{equation} \label{formulatidelta}
t_i = \frac{-\langle \overrightarrow{PA}, \overrightarrow{AB} \rangle \pm \sqrt{ \Delta  }}{\|\overrightarrow{AB}\|_2^2}.
\end{equation}
We are now in a position to write Algorithm 1, observing that
$I_i(d) \in (AB)$ belongs to line segment $\overline{AB}$ if and only if
$\langle \overrightarrow{I_i(d) A} ,  \overrightarrow{I_i(d) B}\rangle \leq 0$.\\
\rule{\linewidth}{1pt}
\par {\textbf{Algorithm 1: Computation of the intersection points between line segment $\overline{AB}$ with $A \neq B$ and the
sphere of center $P$ and radius $d$ in $\mathbb{R}^3$.}}\\
\rule{\linewidth}{1pt}
\par {\textbf{Inputs:}} $A, B, P, d$.\\
\par {\textbf{Initialization:}} {\tt{N}}=0; //Will store the number of intersections (0, 1, or 2).
\par {\tt{List\_Intersections=Null}}; //Will store the intersection points.\\
\par //Check if line $(AB)$ and the sphere have an empty intersection or not
\par Compute $\Delta=\langle \overrightarrow{PA}, \overrightarrow{AB} \rangle^2 - \|\overrightarrow{AB}\|_2^2 (\|\overrightarrow{PA}\|_2^2 - d^2)$.
\par {\textbf{If}} $\Delta \geq 0$ {\textbf{then}} //if $\Delta<0$ the intersection is empty.
\par \hspace*{0.4cm}{\textbf{If}} $\Delta=0$ {\textbf{then}} //the intersection of $(AB)$ and the sphere is a singleton $\{I\}$
\par \hspace*{0.8cm}Compute $I=A+t \overrightarrow{AB}$ where 
$
t = \frac{-\langle \overrightarrow{PA}, \overrightarrow{AB} \rangle}{\|\overrightarrow{AB}\|_2^2}
$
(see \eqref{formulaIi}, \eqref{formulatidelta}) and
\par \hspace*{0.8cm}check if $I$ belongs to $\overline{AB}$:
\par \hspace*{0.8cm}{\textbf{If}} $\langle \overrightarrow{IA}, \overrightarrow{IB} \rangle \leq 0$, {\textbf{then}} //$I$ belongs to $\overline{AB}$
\par \hspace*{1.2cm}{\tt{List\_Intersections}}=\{$I$\}, {\tt{N=1}}.
\par \hspace*{0.8cm}{\textbf{End If}}
\par \hspace*{0.4cm}{\textbf{Else}}
\par \hspace*{0.8cm}Compute the intersections $I_1(d)$ and $I_2(d)$ of $(AB)$ and the sphere
\par \hspace*{0.8cm}given by \eqref{formulaIi}, \eqref{formulatidelta}.
\par \hspace*{0.8cm}{\textbf{If}} $\langle \overrightarrow{I_1(d) A} ,  \overrightarrow{I_1(d) B}\rangle \leq 0$ {\textbf{then}} //$I_1(d)$ belongs to $\overline{AB}$
\par \hspace*{1.2cm}{\textbf{If}} $\langle \overrightarrow{I_2(d) A} ,  \overrightarrow{I_2(d) B}\rangle \leq 0$ {\textbf{then}}  //$I_2(d)$ belongs to $\overline{AB}$
\par \hspace*{1.6cm}//$I_1(d)$ and $I_2(d)$ belong to $\overline{AB}$
\par \hspace*{1.6cm}{\tt{List\_Intersections}}=\{$I_1(d), I_2(d)$\}, {\tt{N=2}}.
\par \hspace*{1.2cm}{\textbf{Else}} //Only $I_1(d)$ belongs to the intersection
\par \hspace*{1.6cm}{\tt{List\_Intersections}}=\{$I_1(d)$\}, {\tt{N=1}}.
\par \hspace*{1.2cm}{\textbf{End If}}
\par \hspace*{0.8cm}{\textbf{Else}}
\par \hspace*{1.2cm}{\textbf{If}} $\langle \overrightarrow{I_2(d) A} ,  \overrightarrow{I_2(d) B}\rangle \leq 0$ {\textbf{then}} //$I_2(d)$ belongs to $\overline{AB}$
\par \hspace*{1.6cm}{\tt{List\_Intersections}}=\{$I_2(d)$\}, {\tt{N=1}}.
\par \hspace*{1.2cm}{\textbf{End If}}
\par \hspace*{0.8cm}{\textbf{End If}}
\par \hspace*{0.4cm}{\textbf{End If}}
\par {\textbf{End If}}\\
\par {\textbf{Outputs:}} {\tt{N}}, {\tt{List\_Intersections}}.\\
\rule{\linewidth}{1pt}\\

Algorithm 4 which computes the CDF of $D$ will also make use of Algorithm 2 
that (i) computes the minimal distance $d_{\min}$ and maximal distance $d_{\max}$
between $P$ and the boundary of $\mathcal{S}$, (ii) computes the area of $\mathcal{S}$,
and (iii) determines if $P$ belongs to the relative interior of $\mathcal{S}$ or not.
The computation of the area of $\mathcal{S}$ will be done using formula \eqref{green}.
To know if $P$ belongs to the relative interior of $\mathcal{S}$ or not, we
compute the {\em{crossing number}} (stored in variable {\tt{Crossing\_Number}}
of Algorithm 2) for point $P$ and polyhedron $\mathcal{S}$. Let $R$ be the ray starting at $P$ and parallel to the positive $x$-axis.
The crossing number counts the number of times ray $R$ crosses the boundary of $\mathcal{S}$
going either from the inside to the outside of $\mathcal{S}$ or from the outside to the inside of $\mathcal{S}$.
If the crossing number is odd then $P$ belongs to the relative interior of $\mathcal{S}$.
Otherwise, the crossing number is even and $P$ is on the boundary of  $\mathcal{S}$ or outside $\mathcal{S}$.

Though the computation of the crosssing number (the value of variable {\tt{Crossing}}
{\tt{\_Number}} in
the end of Algorithm 2) is known (see for instance \cite{rourkebook}), we recall it here for the sake
of self-completeness. For each edge $\overline{S_i S_{i+1}}$ of the polygone, we consider its
intersection with $R$. Each time a single intersection point is found that belongs to the relative interior of an edge,
{\tt{Crossing\_Number}} increases by one. If the intersection between the edge and the ray is 
nonempty but is not a single point from the relative interior of the edge, then either this intersection
is an extremal point or it is the whole edge. There are 8 possibles cases, denoted by A-H in Figure \ref{cncases}.
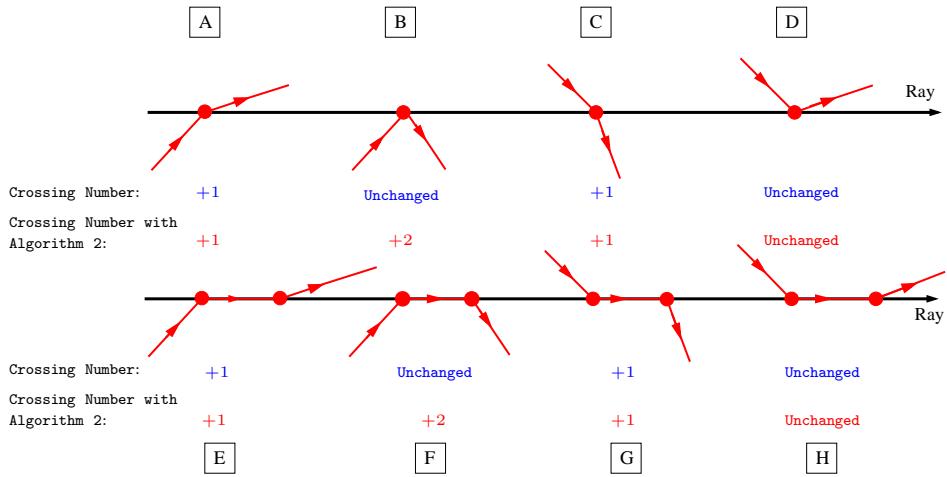
\begin{figure}
\begin{center} 
\begin{tabular}{l}
\input{CN.pstex_t}
\end{tabular}
\caption{Increase in the crossing number when the ray passes through an extremal point of the polygone or when an edge of the polygone is
contained in the ray.} 
\label{cncases}
\end{center} 
\end{figure} 
This figure also provides the increase in the crossing number in each case. To deal with these cases, the following
(known) rules are used in Algorithm 2: (a) horizontal edges (edges $\overline{S_i S_{i+1}}$  
with $y_{S_i}=y_{S_{i+1}}$) are not considered,
(b) for upward edges (edges $\overline{S_i S_{i+1}}$  with $y_{S_i}<y_{S_{i+1}}$), only the final vertex is counted as
an intersection, and 
(c) for downward edges (edges $\overline{S_i S_{i+1}}$  with $y_{S_i}>y_{S_{i+1}}$), only the starting vertex is counted as
an intersection.\footnote{Alternatively, we can of course count only the starting vertices of upward edges and the final vertices of downard edges.}
The increase in the crossing number using these rules is reported for cases A-H in Figure \ref{cncases}.
Comparing with the expected increase in the crossing number in each case, we see that
variable {\tt{Crossing\_Number}} that is updated using these rules in Algorithm 2,
will be even if and only if $P$ is on the boundary of  the polygone or outside the polygone, as expected.
\rule{\linewidth}{1pt}
\par {\textbf{Algorithm 2: Given a polygone $\mathcal{S}$ contained in a plane and a point $P$ in that plane, the 
algorithm computes the area of $\mathcal{S}$, the crossing number, and the minimal and maximal distances from $P$ to the boundary of $\mathcal{S}$.}}\\
\rule{\linewidth}{1pt}
\par {\textbf{Inputs:}} $P$ and the vertices $S_1, S_2, \ldots, S_n$ of a polygone contained in a plane.\\
\par {\textbf{Initialization:}} $\mathcal{L}=0$. //Will store line integral \eqref{green} taking $\mathbb{D}=\mathcal{S}$, i.e., 
\par \hspace*{3.6cm}//will store $\mathcal{A}(\mathcal{S})$.
\par {\tt{Crossing\_Number}}=0. //Will store the crossing number.
\par $d_{\min}=+\infty$. //Will store the minimal distance from $P$ to the boundary of $\mathcal{S}$.
\par $d_{\max}=0$. //Will store the maximal distance from $P$ to the boundary of $\mathcal{S}$.\\
\par {\textbf{For}} $i=1,\ldots,n$,
\par {\hspace*{0.7cm}}$\mathcal{L}=\mathcal{L} +\frac{1}{2} \mathcal{I}_{\overline{S_i S_{i+1}}}$ where for a line segment $\overline{AB}$, $\mathcal{I}_{\overline{AB}}$ is given by \eqref{greensegment}.
\par {\hspace*{0.7cm}}//Computation of the crossing number
\par {\hspace*{0.7cm}}{\textbf{If} $y_{S_i} < y_P \leq y_{S_{i+1}}$ or $y_{S_{i+1}} < y_P \leq y_{S_{i}}$ {\textbf{then}
\par {\hspace*{1.1cm}}//Compute the abscissa $x_I$ of the intersection $I$ of the line $y=y_P$
\par {\hspace*{1.1cm}}//and line segment $\overline{S_{i} S_{i+1}}$:
$$
x_I = x_{S_i} + \frac{x_{S_{i+1}}-x_{S_i}}{y_{S_{i+1}}-y_{S_i}}(y_P - y_{S_i}).
$$
\par {\hspace*{1.1cm}}{\textbf{If} $x_I > x_P$ {\textbf{then} 
$$
{\tt{Crossing\_Number}}={\tt{Crossing\_Number}}+1
$$
\par {\hspace*{1.1cm}}{\textbf{End If} 
\par {\hspace*{0.7cm}}{\textbf{End If} 
\par {\hspace*{0.7cm}}//Computation of the maximal distance from $P$ to the boundary of $\mathcal{S}$
\par {\hspace*{0.7cm}}$d_{\max}=\max(d_{\max}, \|\overrightarrow{P S_i}\|_2)$
\par {\hspace*{0.7cm}}//Computation of the minimal distance from $P$ to the boundary of $\mathcal{S}$
\par {\hspace*{0.7cm}}Compute the projection $P_0$ of $P$ onto line $(S_i S_{i+1})$:
$$
P_0 = S_i  + \frac{\langle \overrightarrow{S_i P}, \overrightarrow{S_i S_{i+1}} \rangle }{\|\overrightarrow{S_i S_{i+1}}\|_2^2} \overrightarrow{S_i S_{i+1}}.
$$
\par {\hspace*{0.7cm}}{\textbf{If}} $\langle \overrightarrow{P_0 S_i}, \overrightarrow{P_0 S_{i+1}} \rangle \leq 0$   {\textbf{then}}
\par \hspace*{1.1cm}//$P_0$ belongs to $\overline{AB}$ 
\par \hspace*{1.1cm}$d_{\min}=\min( d_{\min}, \|\overrightarrow{P P_0}\|_2$).
\par {\hspace*{0.7cm}}{\textbf{Else}}
\par \hspace*{1.1cm}$d_{\min}=\min(d_{\min}, \|\overrightarrow{P {S_{i}}}\|_2, \|\overrightarrow{P {S_{i+1}}}\|_2)$.
\par {\hspace*{0.7cm}}{\textbf{End If} 
\par {\textbf{End For}}\\
\par {\textbf{Outputs:}} {\tt{Crossing\_Number}}, $\mathcal{L}, d_{\min}, d_{\max}$.\\
\rule{\linewidth}{1pt}
The outputs of Algorithm 2 allow us to know if $P$ belongs to $\mathcal{S}$ or not.
 Indeed, $P$ belongs to $\mathcal{S}$ if and only if $P$ belongs to the relative interior of $\mathcal{S}$, which
 occurs if and only if the crossing number is odd, or if $P$ is on the boundary of $\mathcal{S}$, which occurs
 if and only if $d_{\min}=0$. As a result, $P$ belongs to $\mathcal{S}$ if and only if
 {\tt{Crossing\_Number}} is odd or $d_{\min}=0$.
 \begin{rem} The crossing number computed replacing the condition $x_I > x_P$ by
 $x_I \geq x_P$ in Algorithm 2 will not necessarily be odd if $P$ belongs to the boundary of $\mathcal{S}$.
 For instance, if $\mathcal{S}$ is the rectangle
 $\mathcal{S}=\{(x,y) \;:\; x_1 \leq x \leq x_2,\;y_1 \leq y \leq y_2\}$ then if the condition
  $x_I > x_P$ is replaced by
 $x_I \geq x_P$ in Algorithm 2, if we take $P=((x_1+x_2)/2, y_1)$ then
 variable {\tt{Crossing\_Number}} will be even while if we take $P=(x_2, (y_1+y_2)/2)$
 this variable will be odd. However, both points belong to the boundary of $\mathcal{S}$.
 \end{rem}

Let us now comment on Algorithm 4 that
computes the cumulative distribution function of $D$ using
Algorithms 1 and 2. 

We first explain the different steps of Algorithm 4 when there is at least an edge of
$\mathcal{S}$ that has a nonempty intersection with both the relative interior of 
$\mathcal{D}(P, d)$ and the complement of $\mathcal{D}(P, d)$. In other words, we exclude for the moment the cases
$\mathcal{D}(P, d) \subset \mathcal{S}$,  $\mathcal{S} \subset \mathcal{D}(P, d)$, and
$\mathcal{D}(P, d) \cap \mathcal{S} = \emptyset$.

In this case, at the end of Algorithm 4, $\ell$ stores line integral \eqref{green} with 
$\mathbb{D}=\mathcal{S} \cap \mathcal{D}(P, d)$, i.e.,
the area of $\mathcal{S} \cap \mathcal{D}(P, d)$.

In the first {\textbf{For}} loop of Algorithm 4, starting from $\ell=0$, we update
$\ell$ travelling along the edges of $\mathcal{S}$
always leaving  
the relative interior of $\mathcal{S}$ to the left. 
In the end of this loop, $\ell$ is the sum of
 line integrals \eqref{greensegment} computed for all the 
line segments belonging to the boundary of $\mathcal{S} \cap \mathcal{D}(P,d)$.
More precisely, at iteration $i$ of this loop, we consider edge $\overline{S_i S_{i+1}}$.
\begin{figure}
\begin{tabular}{l}
\input{Inter0.pstex_t}
\end{tabular}
\caption{Cases where $S_{i+1}$ is not on the boundary of $\mathcal{D}(P,d)$.}
\label{fignotboundary}
\end{figure}
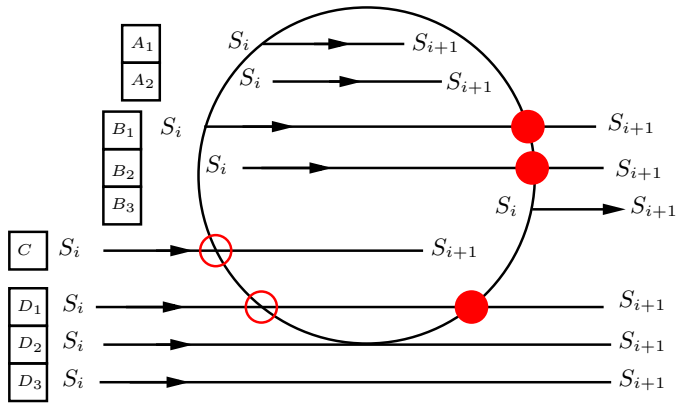
For this edge, 6 cases can happen:
\begin{itemize}
\item[(i)] $S_i$ belongs to $\mathcal{D}(P, d)$ and $S_{i+1}$ belongs to the relative interior of $\mathcal{D}(P, d)$. In this case,
the whole segment $\overline{S_i S_{i+1}}$ belongs to the boundary of $\mathcal{S} \cap \mathcal{D}(P,d)$
and $\ell \leftarrow \ell +\frac{1}{2}\mathcal{I}_{\overline{S_i S_{i+1}}}$. This corresponds to subcases $A_1$ 
(where $S_i$ is on the boundary of $\mathcal{D}(P,d)$)
and $A_2$ (where $S_i$ belongs to the relative interior of $\mathcal{D}(P,d)$) in Figure \ref{fignotboundary}.
\item[(ii)] $S_i$ belongs to $\mathcal{D}(P, d)$ and $S_{i+1}$ does not belong to $\mathcal{D}(P, d)$.
In this situation, either $S_i$ belongs to the boundary of $\mathcal{D}(P, d)$ 
(subcases $B_1$  and $B_3$ in Figure \ref{fignotboundary}) or $S_i$ belongs
to the relative interior of $\mathcal{D}(P, d)$ (subcase $B_2$ in Figure \ref{fignotboundary}).
If $\overline{S_i S_{i+1}}$ and $\mathcal{C}(P,d)$ have an
 intersection point $I_i$ that is different from $S_i$ then
$\overline{S_{i} I_i}$ belongs to the boundary of $\mathcal{S} \cap \mathcal{D}(P,d)$
and $\ell \leftarrow \ell +\frac{1}{2}\mathcal{I}_{\overline{S_i I_{i}}}$.
\item[(iii)] $S_i$ belongs to $\mathcal{D}(P, d)$ and $S_{i+1}$ is on the boundary of $\mathcal{D}(P, d)$.
As in (i), the whole segment $\overline{S_i S_{i+1}}$ belongs to the boundary of $\mathcal{S} \cap \mathcal{D}(P,d)$
and $\ell \leftarrow \ell +\frac{1}{2}\mathcal{I}_{\overline{S_i S_{i+1}}}$.
\item[(iv)] $S_i$ does not belong to $\mathcal{D}(P, d)$ and $S_{i+1}$ belongs to the relative interior of $\mathcal{D}(P, d)$
(case $C$ in Figure \ref{fignotboundary}). In this case,
$\overline{S_i S_{i+1}}$ and $\mathcal{C}(P,d)$ have a single intersection point $I_i$,
$\overline{I_i S_{i+1}}$ belongs to the boundary of $\mathcal{S} \cap \mathcal{D}(P,d)$,
and $\ell \leftarrow \ell +\frac{1}{2}\mathcal{I}_{\overline{I_{i} S_{i+1}}}$.
\item[(v)] Both $S_i$ and $S_{i+1}$ are outside $\mathcal{D}(P,d)$.
There are three subcases: $\overline{S_i S_{i+1}}$ and $\mathcal{C}(P,d)$ have two intersection points 
$I_{i 1}$ and $I_{i 2}$ (case $D_1$ in Figure \ref{fignotboundary}); 
$\overline{S_i S_{i+1}}$ and $\mathcal{C}(P,d)$ have a single intersection point
(case $D_1$ in Figure \ref{fignotboundary}); or 
$\overline{S_i S_{i+1}}$ and $\mathcal{C}(P,d)$ have an empty intersection
(case $D_3$ in Figure \ref{fignotboundary}). In case $D_1$, 
$\overline{I_{i 1} I_{i 2}}$ belongs to the boundary of $\mathcal{S} \cap \mathcal{D}(P,d)$
and $\ell \leftarrow \ell +\frac{1}{2}\mathcal{I}_{\overline{I_{i 1} I_{i 2}}}$.
\item[(vi)] $S_i$ does not belong to $\mathcal{D}(P, d)$ and $S_{i+1}$ is on the boundary of $\mathcal{D}(P, d)$.
If $\overline{S_i S_{i+1}}$ and $\mathcal{C}(P,d)$ have two intersection points $I_i$ and $S_{i+1}$ then
$\overline{I_i S_{i+1}}$ belongs to the boundary of $\mathcal{S} \cap \mathcal{D}(P,d)$
and $\ell \leftarrow \ell +\frac{1}{2}\mathcal{I}_{\overline{I_{i} S_{i+1}}}$.
\end{itemize}
We also have to determine the arcs that belong to the boundary of $\mathcal{S} \cap \mathcal{D}(P,d)$.
A simple way to do this would be as follows: 
\begin{itemize}
\item[(a)] store all the intersections between the edges of the polygone and the boundary
of  $\mathcal{D}(P,d)$.
\item[(b)] Sort these intersection points $(x_i, y_i)$ 
in ascending order of their angles ${\tt{Angle}}(x_i,y_i)$.
\item[(c)] To know if a given arc belongs to $\mathcal{S} \cap \mathcal{D}(P,d)$,
take the middle $M$ of this arc and compute the crossing number and $d_{\min}$ for $\mathcal{S}$ and $M$ using Algorithm 2. 
The corresponding arc belongs to $\mathcal{S} \cap \mathcal{D}(P,d)$ if and only if
the crossing  number is odd or $d_{\min}=0$
\end{itemize}
The complexity of this algorithm is $O(n^2)$ where $n$ is the number of edges.
Algorithm 4 which has complexity $O(n \ln n)$ selects the appropriate arcs in a more efficient manner.
In this algorithm, the extremities of these arcs are stored, without repetitions, in 
the list {\tt{Intersections}} which is updated along the iterations of the first 
{\textbf{For}} loop of Algorithm 4: {\tt{Intersections}}$(i)$ will be the $i$-th "relevant" (see below) intersection point found.
To know the arcs that belong to $\mathcal{S} \cap \mathcal{D}(P,d)$, a second list {\tt{Arcs}} is used:
the $i$-th element of list {\tt{Arcs}} is 1 if and only if
the arc from the boundary of $\mathcal{D}(P,d)$ obtained starting at {\tt{Intersections}}$(i)$ and ending at 
the next element from list {\tt{Intersections}} found travelling counter clockwise on the boundary of $\mathcal{D}(P, d)$ belongs to 
$\mathcal{S} \cap \mathcal{D}(P,d)$. To produce this information, when an intersection between $\mathcal{S}$ and
$\mathcal{C}(P, d)$ is found we need to know the type of this intersection, knowing that there are three
types of intersections: 
\begin{itemize}
\item[$T_1$:] the intersection is not "relevant", i.e., there is no arc from $\mathcal{S} \cap \mathcal{C}(P,d)$
starting or ending at this point;
\item[$T_2$:] there is an arc from $\mathcal{S} \cap \mathcal{C}(P,d)$ starting at this point (in this case the corresponding
entry of {\tt{Arcs}} is one);
\item[$T_3$:] there is an arc from $\mathcal{S} \cap \mathcal{C}(P,d)$ ending at this point (in this case the corresponding
entry of {\tt{Arcs}} is zero).
\end{itemize}
Now let us go back to the 6 cases (i)-(vi) discussed above and considered in the first {\textbf{For}} loop of
Algorithm 4. It remains to explain how to determine in each of these cases the intersection type when an intersection is found.

First, since vertices belonging to the boundary of $\mathcal{D}(P,d)$ are starting vertices of an edge and ending vertices
of another edge, to avoid counting them twice, we do not consider the intersection points that are starting vertices
of an edge. With this convention, in case (i), i.e., subcases $A_1$ and $A_2$ in Figure \ref{fignotboundary}, we do not need to store intersection points,
even if $S_i$ belongs to $\mathcal{D}(P,d)$.

In case (ii), corresponding to subcases $B_1, B_2$, and $B_3$ in Figure \ref{fignotboundary}, if 
$\overline{S_i S_{i+1}}$ and $\mathcal{C}(P,d)$ have an intersection point that is different from $S_i$ 
then this intersection point is stored in list {\tt{Intersections}} and it is of type $T_2$: the corresponding
entry in {\tt{Arcs}} is one (these type $T_2$ intersections are represented by red balls in Figure \ref{fignotboundary}).

In case (iv), corresponding to case $C$ in Figure \ref{fignotboundary}, there is a single intersection point
between $\overline{S_i S_{i+1}}$ and $\mathcal{C}(P,d)$ and it is of type $T_3$: the corresponding
entry in {\tt{Arcs}} is zero (these type $T_3$ intersections are represented by red circles in Figure \ref{fignotboundary}).

Case (v) corresponds to cases $D_1, D_2$, and $D_3$ in Figure \ref{fignotboundary}. In subcase $D_1$, i.e.,
when $\overline{S_i S_{i+1}}$ and $\mathcal{C}(P,d)$ have two intersections, the first one encountered when travelling from
$S_i$ to $S_{i+1}$ is of type $T_3$ while the second one is of type $T_2$. In subcase $D_2$, $\overline{S_i S_{i+1}}$ and $\mathcal{C}(P,d)$ 
have a single intersection which is of type $T_1$.

Let us now consider cases (iii) and (vi), the cases where $S_{i+1}$ is on the boundary of $\mathcal{D}(P,d)$.
We want to determine the intersection type for $S_{i+1}$. 
This is done using an auxiliary algorithm, Algorithm 3, that takes
as entries $P$ and $d$ (the center and radius
of $\mathcal{C}(P,d)$) and three successive vertices $S_i, S_{i+1}$, and $S_{i+2}$ of $\mathcal{S}$, knowing 
that $S_{i+1}$ is on the boundary of $\mathcal{D}(P,d)$.
The output variable {\tt{Arc}} of this algorithm is one (resp. zero) if and only 
if $S_{i+1}$ is of type $T_2$ or $T_3$ (resp. type $T_1$).
What matters to determine the intersection type for $S_{i+1}$ is whether
$\overline{S_i S_{i+1}}$ is contained in 
some half-space (to be specified below) 
that does not contain $P $ or not.
An additional input variable of Algorithm 3 described below, variable {\tt{In}}, takes the value zero in the former case and the value one in the latter case.
To explain this algorithm, it is convenient to introduce two half spaces $\mathcal{H}_{{\tt{Right}}}$
and $\mathcal{H}_{P}$ and a line $L_1$. These half spaces and lines depend on the
entries of Algorithm 3, i.e., $P$ and $d$ (the center and radius
of $\mathcal{C}(P,d)$) and three successive vertices $S_i, S_{i+1}$, and $S_{i+2}$ of $\mathcal{S}$.
Line $L_1$ is the line that contains line segment $\overline{S_i S_{i+1}}$.
The open half space $\mathcal{H}_{{\tt{Right}}}$ is the set of points that are to the right of line $L_1$
when travelling on this line in the direction $S_i \rightarrow S_{i+1}$.
Denoting by $L_2$ the line that is tangent to the circle $\mathcal{C}(P,d)$ at $S_{i+1}$ (recall that 
$S_{i+1}$ belongs to $\mathcal{C}(P,d)$), the closed half space  $\mathcal{H}_{P}$ is the set 
of points that are on the side of line
$L_2$ that does not contain $P$, including $L_2$. The definitions of these sets follow.

For $L_1$ and $\mathcal{H}_{{\tt{Right}}}$, we obtain:
\begin{equation} 
\left\{
\begin{array}{l}
\mbox{if }x_{S_{i+1}}=x_{S_i} \mbox{ and }y_{S_{i+1}}>y_{S_i} \mbox{ then }\\
\left\{
\begin{array}{l}
L_1  =  \{ (x,y)\;:\; x=x_{S_i}\},\\
\mathcal{H}_{{\tt{Right}}}  =   \{ (x, y)\;:\; x>x_{S_i}\}.
\end{array}
\right.\\
\mbox{If }x_{S_{i+1}}=x_{S_i} \mbox{ and }y_{S_{i+1}}<y_{S_i} \mbox{ then }\\
\left\{
\begin{array}{l}
L_1  =  \{ (x, y)\;:\; x=x_{S_i}\},\\
\mathcal{H}_{{\tt{Right}}}  =   \{ (x, y)\;:\; x<x_{S_i}\}.
\end{array}\right.\\
\mbox{If }x_{S_{i+1}}>x_{S_i}\mbox{ then }\\
\left\{
\begin{array}{l}
L_1  =  \{ (x, y)\;:\; y=y_{S_i} + \frac{y_{S_{i+1}} - y_{S_i}}{x_{S_{i+1}} - x_{S_i}} (x-x_{S_i})\},\\
\mathcal{H}_{{\tt{Right}}}  =   \{ (x, y)\;:\; y<y_{S_i} + \frac{y_{S_{i+1}} - y_{S_i}}{x_{S_{i+1}} - x_{S_i}} (x-x_{S_i}) \}.
\end{array}
\right.\\
\mbox{If }x_{S_{i+1}}<x_{S_i}\mbox{ then }\\
\left\{
\begin{array}{l}
L_1  =  \{ (x, y)\;:\; y=y_{S_i} + \frac{y_{S_{i+1}} - y_{S_i}}{x_{S_{i+1}} - x_{S_i}}  x_{S_i}(x-x_{S_i})\},\\
\mathcal{H}_{{\tt{Right}}}  =   \{ (x, y)\;:\; y>y_{S_i} + \frac{y_{S_{i+1}} - y_{S_i}}{x_{S_{i+1}} - x_{S_i}} (x-x_{S_i}) \}.
\end{array}
\right.
\end{array}
\right.
\end{equation}
Next observe that $M=(x,y) \in \mathcal{H}_P$ if and only if 
$\langle \overrightarrow{S_{i+1}M}, \overrightarrow{S_{i+1}P} \rangle \leq 0$
and therefore
\begin{equation} 
\mathcal{H}_{P}  =   \{ (x, y)\;:\;(x-x_{S_{i+1}})(x_P - x_{S_{i+1}}) + (y-y_{S_{i+1}})(y_P - y_{S_{i+1}}) \leq 0 \}.
\end{equation}
Let us first consider the case when
input variable
{\tt{In}} of Algorithm 3 is one, i.e., the case when 
$S_i$ does not belong to $\mathcal{H}_P$.
In this case, the edge $\overline{S_{i+1} S_{i+2}}$ can belong to three different regions,
denoted by $\mathcal{R}_1$, $\mathcal{R}_2$, and $\mathcal{R}_3$ in Figure \ref{figfourcases}
and respectively represented in pink at the top left, in green at the top right, and in yellow in the middle left figures of Figure \ref{figfourcases}.
In this Figure \ref{figfourcases}, type $T_2$ intersections are represented by red balls while
type $T_3$ intersections are represented by red circles.
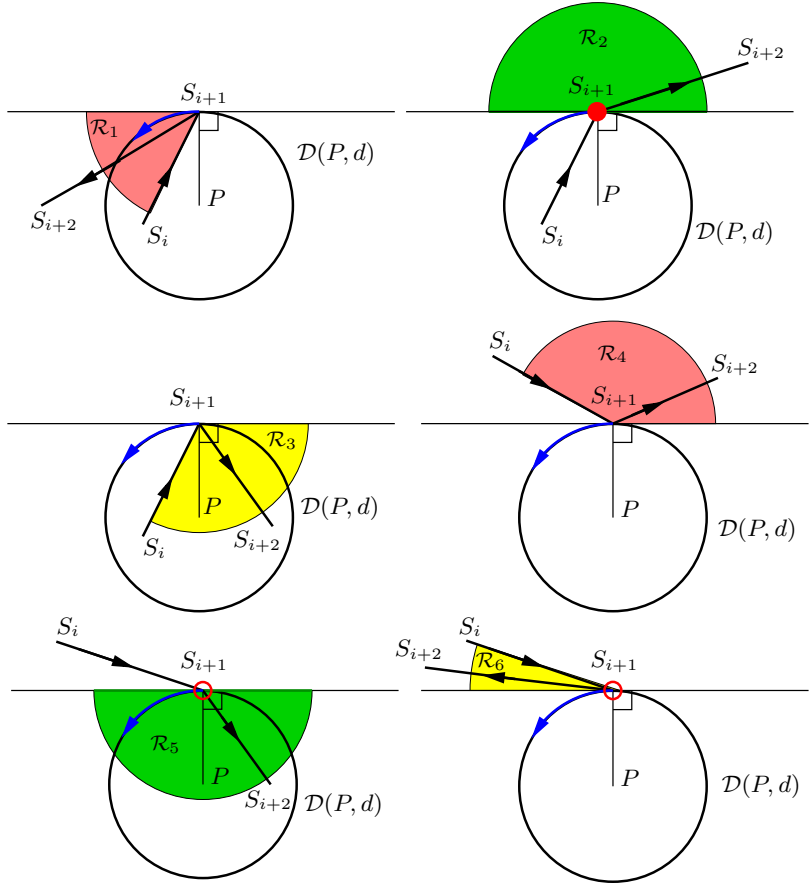
\begin{figure}
\begin{tabular}{l}
\input{Inter1.pstex_t}
\end{tabular}
\caption{The six cases where an endpoint $S_{i+1}$ of an edge is on the boundary of $\mathcal{D}(P,d)$.}
\label{figfourcases}
\end{figure}
Regions $\mathcal{R}_1, \mathcal{R}_2$, and $\mathcal{R}_3$ are given by (see Figure \ref{figfourcases}):
$$
\begin{array}{l}
\mathcal{R}_1   =  \overline{\mathcal{H}_{P}} \cap \overline{\mathcal{H}_{\tt{Right}} \cup L_1},\;\; \mathcal{R}_2   =  \mathcal{H}_{P},
\mbox{ and }\mathcal{R}_3   =  \overline{\mathcal{H}_{P}} \cap \mathcal{H}_{\tt{Right}}.
\end{array}
$$
If $S_{i+2}$ belongs to $\mathcal{R}_1$ or $\mathcal{R}_3$,
then $S_{i+1}$ is a type $T_1$ intersection while if 
$S_{i+2}$ belongs to $\mathcal{R}_2$
$S_{i+1}$ is a type $T_2$ intersection.

We now consider the case where input variable
{\tt{In}} of Algorithm 3 is zero, i.e., the case
where 
$S_{i+1}$ belongs to $\mathcal{H}_P$.
In this case, $S_{i+2}$ can belong to three different regions,
denoted by $\mathcal{R}_4$, $\mathcal{R}_5$, and $\mathcal{R}_6$ in Figure \ref{figfourcases}
and respectively represented in pink in the middle right, in green in the bottom left, and in yellow in the bottom right figures of Figure \ref{figfourcases}.

Regions $\mathcal{R}_4, \mathcal{R}_5$, and $\mathcal{R}_6$ are given by (see Figure \ref{figfourcases}):
$$
\begin{array}{l}
\mathcal{R}_4   =  \mathcal{H}_{P} \cap \overline{\mathcal{H}_{\tt{Right}} \cup L_1},\;\; \mathcal{R}_5   =  \overline{\mathcal{H}_{P}},
\mbox{ and }\mathcal{R}_6   =  \mathcal{H}_{P} \cap \mathcal{H}_{\tt{Right}}.
\end{array}
$$

If $S_{i+2}$ belongs to $\mathcal{R}_4$ or $\mathcal{R}_6$
then $S_{i+1}$ is a type $T_1$ intersection while 
if $S_{i+2}$ belongs to $\mathcal{R}_5$
$S_{i+2}$ is a type $T_3$ intersection.

Summarizing our observations, if $S_{i+1}$ belongs to the boundary of $\mathcal{D}(P,d)$, this intersection
is stored as a "relevant" intersection (it is not a type $T_1$ intersection) if and only if
{\tt{In}}=1 and $S_{i+2} \in \mathcal{R}_2$ (in this case, it is a type $T_2$ intersection)   
or {\tt{In}}=0 and $S_{i+2} \in \mathcal{R}_5$ (in this case, it is a type $T_3$ intersection).\\
\rule{\linewidth}{1pt}
\par {\textbf{Algorithm 3: Given three successive vertices $S_i, S_{i+1}$, and $S_{i+2}$ of a 
simple polygone
$\mathcal{S}$ and a circle of center $P$ and radius $d>0$ with $S_{i+1}$ belonging to this circle, the algorithm determines if $S_{i+1}$
is or is not a starting or ending point of an arc from
the boundary of $\mathcal{D}(P,d) \cap \mathcal{S}$.}}\\
\rule{\linewidth}{1pt}
\par {\textbf{Inputs:}} $P, d, S_i, S_{i+1}, S_{i+2}, {\tt{In}}$.\\
\par {\textbf{Initialization:}} {\tt{Arc}}=0.\\
\par \textbf{If} {\tt{In}} and $(x_{S_{i+2}}-x_{S_{i+1}})(x_P - x_{S_{i+1}}) + (y_{S_{i+2}}-y_{S_{i+1}})(y_P - y_{S_{i+1}}) \leq 0$ \textbf{then} 
${\tt{Arc}}=1.$
\par \textbf{Else if} $\overline{{\tt{In}}}$ and 
$(x_{S_{i+2}}-x_{S_{i+1}})(x_P - x_{S_{i+1}}) + (y_{S_{i+2}}-y_{S_{i+1}})(y_P - y_{S_{i+1}}) > 0$ \textbf{then} 
${\tt{Arc}}=1.$
\par \textbf{End if}\\
\par {\textbf{Output:}} {\tt{Arc}}.\\
\rule{\linewidth}{1pt}
In the end of the first {\textbf{For}} loop of Algorithm 4,
the "relevant"  intersections  points  $(x_i, y_i)$ 
of $\mathcal{S}$ and $\mathcal{C}(P,d)$ are stored in list {\tt{Intersections}}.
We then sort these intersections
in ascending order of their angles
${\tt{Angle}}(x_i,y_i)$ where we recall that
{\tt{Angle}} is defined in \eqref{polarangle}. The values in list {\tt{Arcs}} are sorted correspondingly.
For {\tt{Nb\_Intersections}} intersection points, this defines {\tt{Nb\_Intersections}} arcs on the circle.
At $i$-th iteration of the
second {\textbf{For}} loop of Algorithm 4, the
$i$-th arc is considered. If this arc belongs  to $\mathcal{S} \cap \mathcal{D}(P,d)$,
i.e., if {\tt{Arcs}}$(i)=1$, the corresponding line integral \eqref{greenarc} is computed. The sum of these line integrals
makes up the last part of line integral \eqref{green} for $\mathbb{D}=\mathcal{D}(P, d) \cap \mathcal{S}$.

It remains to check that the algorithm correctly computes $F_D(d)$ when  variable {\tt{Nb\_Intersections}}
in the end of Algorithm 4 is null.
This can occur in three different manners reported in Figure \ref{emptyinterpoly}:
(i) $\mathcal{D}(P,d) \cap \mathcal{S} = \emptyset$, (ii) the polygone $\mathcal{S}$ is contained
in $\mathcal{D}(P,d)$, and (iii) the disk $\mathcal{D}(P,d)$ is contained
in $\mathcal{S}$. 
\begin{figure}
\begin{center} 
\begin{tabular}{l}
\input{Poly_Inter_Empty.pstex_t}
\end{tabular}
\caption{Cases where {\tt{Nb\_Intersections}}=0.} 
\label{emptyinterpoly}
\end{center} 
\end{figure}
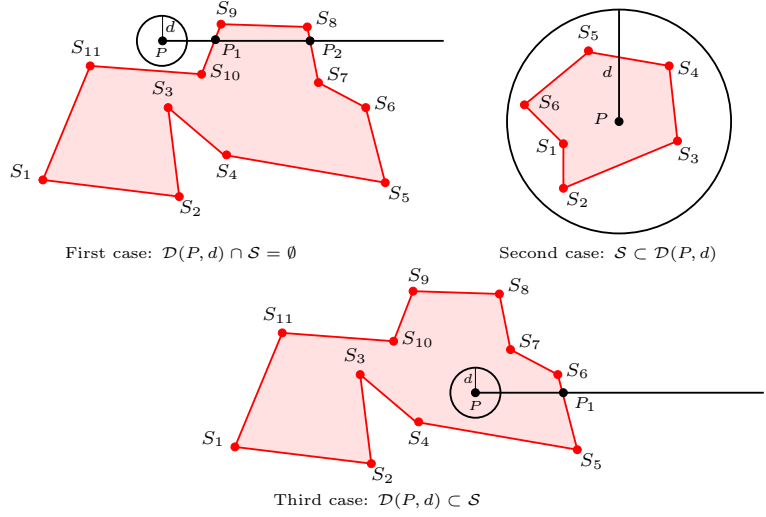 
Case (ii) corresponds to $\ell=\mathcal{L}$ and in this case
$F_D(d)=1$. If $\ell \neq \mathcal{L}$, case (i) occurs when $P$ is outside $\mathcal{S}$
and case (iii) when $P$ belongs to the relative interior of $\mathcal{S}$.
To know if case (i) or case (iii) occurs, we use
the {\em{crossing number}} computed by Algorithm 2.
If the crossing number is odd then $P$ is inside $\mathcal{S}$ and $F_D(d)=\pi d^2 / \mathcal{L}$.
Otherwise, the crossing number is even, $P$ is outside $\mathcal{S}$ (case (i)) and $F_D(d)=0$.
\rule{\linewidth}{1pt}
\par {\textbf{Algorithm 4: Computation of the value $F_D(d)$ of the cumulative distribution function 
of $D$ at $d$ when $X$ is uniformly distributed in a polygone contained in a plane with $P$ in that plane.}}\\
\rule{\linewidth}{1pt}
\par {\textbf{Inputs:}} $P$, the vertices $S_1, \ldots, S_n$, of polygone $\mathcal{S}$, {\tt{Crossing\_Number}}, $\mathcal{L}$, $d$.\\
\par {\textbf{Initialization:}} $\ell=0$ //Will store line integral \eqref{green} taking $\mathbb{D}=\mathcal{S} \cap \mathcal{D}(P, d)$,
\par \hspace*{3.3cm}//i.e., will compute the area of $\mathbb{D}=\mathcal{S} \cap \mathcal{D}(P, d)$.
\par {\tt{Intersections}}={\tt{Null}}. //List of the intersections found for $\mathcal{S}$ and $\mathcal{C}(P, d)$.
\par {\tt{Nb\_Intersections}}=0. //Number of intersections found for $\mathcal{S}$ and $\mathcal{C}(P, d)$.
\par {\tt{Arcs}}={\tt{Null}}. //Stores the arcs that are on the boundary of $\mathcal{S} \cap \mathcal{D}(P, d)$.\\
\par {\textbf{For} $i=1,\ldots,n$, 
\par {\hspace*{0.7cm}}//Check if $S_i$ belongs to $\mathcal{D}(P, d)$ or not:
\par {\hspace*{0.7cm}}{\textbf{If} $\|\overrightarrow{S_i P}\|_2 \leq d$ {\textbf{then} 
\par {\hspace*{1.1cm}}{\textbf{If} $\|\overrightarrow{S_{i+1} P}\|_2 < d$ {\textbf{then} //Cases $A_1$ and $A_2$ in Figure \ref{fignotboundary}
\par {\hspace*{1.5cm}}$\ell \leftarrow \ell+\frac{1}{2}\mathcal{I}_{\overline{S_i S_{i+1}}}$ where for a line segment $\overline{AB}$, $\mathcal{I}_{\overline{AB}}$ is given by \eqref{greensegment}.
\par {\hspace*{1.1cm}}{\textbf{Else If}$\|\overrightarrow{S_{i+1} P}\|_2 > d$ //Cases $B_1, B_2$, and $B_3$ in Figure \ref{fignotboundary}
\par {\hspace*{1.5cm}}Call Algorithm 1 to compute the intersections between the circle
\par {\hspace*{1.5cm}}of center $P$ and radius $d$ with the line segment $\overline{S_i S_{i+1}}$.
\par {\hspace*{1.5cm}}\textbf{If} there is an intersection point different from $S_i$ \textbf{then}
\par {\hspace*{1.9cm}}Let $I_i$ be this intersection point.
\par {\hspace*{1.9cm}}$\ell \leftarrow \ell+ \frac{1}{2}\mathcal{I}_{\overline{S_i I_i}}$ where for a line segment $\overline{AB}$, $\mathcal{I}_{\overline{AB}}$ is given by \eqref{greensegment}.
\par {\hspace*{1.9cm}}${\tt{Nb\_Intersections}}\leftarrow {\tt{Nb\_Intersections}} +1$.
\par {\hspace*{1.9cm}}{\tt{Intersections}}[${\tt{Nb\_Intersections}}]= I_i$.
\par {\hspace*{1.9cm}}{\tt{Arcs}}[${\tt{Nb\_Intersections}}]= 1$.
\par {\hspace*{1.5cm}}\textbf{End If}
\par {\hspace*{1.1cm}}{\textbf{Else}
\par {\hspace*{1.5cm}}$\ell \leftarrow \ell+\frac{1}{2}\mathcal{I}_{\overline{S_i S_{i+1}}}$ where for a line segment $\overline{AB}$, $\mathcal{I}_{\overline{AB}}$ is given by \eqref{greensegment}.
\par {\hspace*{1.5cm}}Call Algorithm 3 with input variables $P, d, S_i, S_{i+1}, S_{i+2}$ and with
\par {\hspace*{1.5cm}}variable {\tt{In}} set to 1.
\par {\hspace*{1.5cm}}{\textbf{If}} the variable {\tt{Arc}} returned by this algorithm is 1 \textbf{then}
\par {\hspace*{1.9cm}}${\tt{Nb\_Intersections}}\leftarrow {\tt{Nb\_Intersections}} +1$.
\par {\hspace*{1.9cm}}{\tt{Intersections}}[${\tt{Nb\_Intersections}}]= S_{i+1}$.
\par {\hspace*{1.9cm}}{\tt{Arcs}}[${\tt{Nb\_Intersections}}]= 1$.
\par {\hspace*{1.5cm}}{\textbf{End If}}
\par {\hspace*{1.1cm}}{\textbf{End If}
\par {\hspace*{0.7cm}}{\textbf{Else}}
\par {\hspace*{1.1cm}}{\textbf{If} $\|\overrightarrow{S_{i+1} P}\|_2 < d$ {\textbf{then} //Case $C$ in Figure \ref{fignotboundary}
\par {\hspace*{1.5cm}}Call Algorithm 1 to compute the intersection $I_i$ between the circle
\par {\hspace*{1.5cm}}of center $P$ and radius $d$ with the line segment $\overline{S_i S_{i+1}}$ (note that
\par {\hspace*{1.5cm}}the intersection is a single point).
\par {\hspace*{1.5cm}}$\ell \leftarrow \ell+ \frac{1}{2} \mathcal{I}_{\overline{I_i S_{i+1}}}$ where for a line segment $\overline{AB}$, $\mathcal{I}_{\overline{AB}}$ is given by \eqref{greensegment}.
\par {\hspace*{1.5cm}}${\tt{Nb\_Intersections}}\leftarrow {\tt{Nb\_Intersections}} +1$.
\par {\hspace*{1.5cm}}{\tt{Intersections}}[${\tt{Nb\_Intersections}}]= I_i$.
\par {\hspace*{1.5cm}}{\tt{Arcs}}[${\tt{Nb\_Intersections}}]= 0$.
\par {\hspace*{1.1cm}}{\textbf{Else If} $\|\overrightarrow{S_{i+1} P}\|_2 > d$ {\textbf{then} //Cases $D_1, D_2$, and $D_3$ in Figure \ref{fignotboundary}
\par {\hspace*{1.5cm}}Call Algorithm 1 to compute the intersections between the circle
\par {\hspace*{1.5cm}}of center $P$ and radius $d$ with the line segment $\overline{S_i S_{i+1}}$. 
\par {\hspace*{1.5cm}}{\textbf{If} there are two intersection points {\textbf{then}}
\par {\hspace*{1.9cm}}Let $I_{i 1}$ and $I_{i 2}$ be these intersection points where $I_{i 1}$ and $I_{i 2}$
satisfy 
\par {\hspace*{1.9cm}} $\frac{x_{I_{i 1}}-x_{S_{i}}}{x_{S_{i+1}}-x_{S_i}} \leq \frac{x_{I_{i 2}}-x_{S_{i}}}{x_{S_{i+1}}-x_{S_i}}$
if $x_{S_{i+1}} \neq x_{S_i} $ and 
$
\frac{y_{I_{i 1}}-y_{S_{i}}}{y_{S_{i+1}}-y_{S_i}} \leq \frac{y_{I_{i 2}}-y_{S_{i}}}{y_{S_{i+1}}-y_{S_i}}
$ 
\par {\hspace*{1.9cm}} if $x_{S_{i+1}}= x_{S_i} $.
\par {\hspace*{1.9cm}}$\ell \leftarrow \ell+ \frac{1}{2} \mathcal{I}_{\overline{I_{i 1} I_{i 2}}}$ where for a line segment $\overline{AB}$, $\mathcal{I}_{\overline{AB}}$ is given by \eqref{greensegment}.
\par {\hspace*{1.9cm}}${\tt{Nb\_Intersections}}\leftarrow {\tt{Nb\_Intersections}} +2$.
\par {\hspace*{1.9cm}}{\tt{Intersections}}[${\tt{Nb\_Intersections-1}}]= I_{i 1}$.
\par {\hspace*{1.9cm}}{\tt{Intersections}}[${\tt{Nb\_Intersections}}]= I_{i 2}$.
\par {\hspace*{1.9cm}}{\tt{Arcs}}[${\tt{Nb\_Intersections-1}}]= 0$.
\par {\hspace*{1.9cm}}{\tt{Arcs}}[${\tt{Nb\_Intersections}}]= 1$.
\par {\hspace*{1.5cm}}{\textbf{End If}
\par {\hspace*{1.1cm}}{\textbf{Else}
\par {\hspace*{1.5cm}}Call Algorithm 1 to compute the intersections between the circle
\par {\hspace*{1.5cm}}of center $P$ and radius $d$ with the line segment $\overline{S_i S_{i+1}}$. 
\par {\hspace*{1.5cm}}{\textbf{If}} there is one intersection {\textbf{then}}
\par {\hspace*{1.9cm}}Call  Algorithm 3 with input variables  $P, d, S_i, S_{i+1}, S_{i+2}$ 
\par {\hspace*{1.9cm}}and with variable {\tt{In}} set to 0.
\par {\hspace*{1.9cm}}{\textbf{If}} the variable {\tt{Arc}} returned by this algorithm is 1 \textbf{then}
\par {\hspace*{2.3cm}}${\tt{Nb\_Intersections}}\leftarrow {\tt{Nb\_Intersections}} +1$.
\par {\hspace*{2.3cm}}{\tt{Intersections}}[${\tt{Nb\_Intersections}}]= S_{i+1}$.
\par {\hspace*{2.3cm}}{\tt{Arcs}}[${\tt{Nb\_Intersections}}]= 0$.
\par {\hspace*{1.9cm}}{\textbf{End If}}
\par {\hspace*{1.5cm}}{\textbf{Else If}} there are two intersections $I_i$  and $S_{i+1}$ \textbf{then}
\par {\hspace*{1.9cm}}$\ell \leftarrow \ell+ \frac{1}{2} \mathcal{I}_{\overline{I_i S_{i+1}}}$ where for a line segment $\overline{AB}$, 
\par {\hspace*{1.9cm}}$\mathcal{I}_{\overline{AB}}$ is given by \eqref{greensegment}.
\par {\hspace*{1.9cm}}${\tt{Nb\_Intersections}}\leftarrow {\tt{Nb\_Intersections}} +1$.
\par {\hspace*{1.9cm}}{\tt{Intersections}}[${\tt{Nb\_Intersections}}]= I_i$.
\par {\hspace*{1.9cm}}{\tt{Arcs}}[${\tt{Nb\_Intersections}}]= 0$.
\par {\hspace*{1.9cm}}Call  Algorithm 3 with input variables  $P, d, S_i, S_{i+1}, S_{i+2}$ 
\par {\hspace*{1.9cm}}and with variable {\tt{In}} set to 1.
\par {\hspace*{1.9cm}}{\textbf{If}} the variable {\tt{Arc}} returned by this algorithm is 1 \textbf{then}
\par {\hspace*{2.3cm}}${\tt{Nb\_Intersections}}\leftarrow {\tt{Nb\_Intersections}} +1$.
\par {\hspace*{2.3cm}}{\tt{Intersections}}[${\tt{Nb\_Intersections}}]= S_{i+1}$.
\par {\hspace*{2.3cm}}{\tt{Arcs}}[${\tt{Nb\_Intersections}}]= 1$.
\par {\hspace*{1.9cm}}{\textbf{End If}}
\par {\hspace*{1.5cm}}{\textbf{End If}}
\par {\hspace*{1.1cm}}{\textbf{End If}
\par {\hspace*{0.7cm}}{\textbf{End If}
\par {\textbf{End For}\\
\par {\textbf{If} {\tt{Nb\_Intersections}}=0 {\textbf{then}
\par {\hspace*{0.7cm}}{\textbf{If} $\ell=\mathcal{L}$ {\textbf{then}
\par {\hspace*{1.1cm}}$F_D(d)=1$
\par {\hspace*{0.7cm}}{\textbf{Else if} variable {\tt{Crossing\_Number}} is odd {\textbf{then}
\par {\hspace*{1.1cm}}//$\mathcal{D}(P,d)$ is inside the polygone
\par {\hspace*{1.1cm}}$F_D(d)=\pi d^2 / \mathcal{L}$
\par {\hspace*{0.7cm}}{\textbf{Else}
\par {\hspace*{1.1cm}}//$\mathcal{D}(P,d)$ has no intersection with the polygone
\par {\hspace*{1.1cm}}$F_D(d)=0$
\par {\hspace*{0.7cm}}{\textbf{End If}
\par {\textbf{Else}
\par {\hspace*{0.7cm}}{\textbf{Sort} the elements $(x_i,y_i)$ of list {\tt{Intersections}}
\par {\hspace*{0.7cm}}by ascending order of their angles ${\tt{Angle}}(x_i,y_i)$ and sort the elements of list 
\par {\hspace*{0.7cm}}{\tt{Arcs}} correspondingly. 
\par {\hspace*{0.7cm}}Let again {\tt{Intersections}} and  {\tt{Arcs}} be the corresponding sorted lists.
\par {\hspace*{0.7cm}}{\textbf{For} $i=1,\ldots,{\tt{Nb\_Intersections}}$
\par {\hspace*{1.1cm}}{\textbf{If}} {\tt{Arcs}}[i]=1 {\textbf{then}
\par {\hspace*{1.5cm}}$\ell \leftarrow \ell+\frac{1}{2} \mathcal{I}_{\stackrel{\frown}{A B}_{d, P}}$ where $A={\tt{Intersections}}[i]$, 
$$
B=
\left\{
\begin{array}{ll}
{\tt{Intersections}}[i+1] & \mbox{ if }i<{\tt{Nb\_Intersections}},\\
{\tt{Intersections}}[1] & \mbox{ if }i={\tt{Nb\_Intersections}},
\end{array}
\right.
$$
\par {\hspace*{1.5cm}}and where $\mathcal{I}_{\stackrel{\frown}{A B}_{d, P}}$ is obtained  substituting $R_0$ by $d$ in \eqref{greenarc}.
\par {\hspace*{1.1cm}}{\textbf{End If}}
\par {\hspace*{0.7cm}}{\textbf{End For}
\par {\hspace*{0.7cm}}$F_D(d)=\ell / \mathcal{L}$.
\par {\textbf{End If}\\
\par {\textbf{Output:}} $F_D(d)$.\\
\rule{\linewidth}{1pt}\\
After calling Algorithm 2, if the crossing number 
is odd, we know that $P$ belongs to the relative interior of $\mathcal{S}$ and for
$0 \leq d \leq d_{\min}$, we have $f_D(d)=\frac{2 \pi d}{\mathcal{L}}$.
For $d \geq d_{\max}$ or $d \leq 0$, the density is null.
If the crossing number is even, $f_D(d)$ is null for $0 \leq d \leq d_{\min}$. For $d_{\min} \leq d \leq d_{\max}$, 
Algorithm 5 provides approximations of the density at points
$d_i, i=1,\ldots,N-1$. \\
\rule{\linewidth}{1pt}
\par {\textbf{Algorithm 5: Computation of the approximate density of $D$ (distance from $P$ to a random variable
uniformly distributed in a polygone) in the range $[d_{\min}, d_{\max}]$.}}\\
\rule{\linewidth}{1pt}
\par {\textbf{Inputs:}} The vertices $S_1, \ldots, S_n$ of a polygone contained in a plane,
\par a point $P$ in this plane, and the number $N$ of discretization points.\\
\par {\textbf{Initialization:}} Call Algorithm 2 to compute $d_{\min}, d_{\max},$ the crossing number,
\par and the area $\mathcal{L}$ of $\mathcal{S}$.
\par {\tt{F\_Old}}$=0$. 
\par {\textbf{For}} $i=1,\ldots,N-1$,
\par \hspace*{0.7cm}Compute $d_i=d_{\min} + \frac{(d_{\max}-d_{\min}) i}{N}$.
\par \hspace*{0.7cm}Call Algorithm 4 with input variables the crossing number, $\mathcal{L}$, $d_{\min}$, $d_{\max}$,
\par \hspace*{0.7cm}and $d=d_i$ to compute $F_D( d_i)$.
\par \hspace*{0.7cm}Compute
${\tilde f}_D(d_i) = \frac{N \Big[ F_D( d_{i} )-{\tt{F\_Old}} \Big] }{d_{\max}-d_{\min}}$ and set ${\tt{F\_Old}}=F_D( d_i)$.
\par {\textbf{End For}}\\
\par {\textbf{Outputs:}} ${\tilde f}_D(d_i), i=1,\ldots,N-1$.\\
\rule{\linewidth}{1pt}\\
Finally, we consider the case where the polygone is contained in a plane $\mathcal{P}$
and $P$ is not contained in that plane. In this situation, referring to arguments
from Section \ref{diskdist}, we can use the previous results reparametrizing the problem
and replacing $P$ and $d$ respectively by $P_0$, the projection of $P$ onto $\mathcal{P}$, and
$R(d)=\sqrt{d^2 - \|\overrightarrow{P P_0}\|_2^2}$. 
Indeed, since $\mathcal{S} \subset \mathcal{P}$, we have
$$
\mathcal{S} \cap \mathcal{B}(P, d)= \mathcal{S} \cap \mathcal{P} \cap \mathcal{B}(P, d)=  \mathcal{S} \cap  \mathcal{D}(P_0, R(d))
$$
where $\mathcal{D}(P_0, R(d))$ is the disk of center $P_0$ and radius $R(d)$ contained in the plane
$\mathcal{P}$ (see Figure \ref{figdiskR3}). 
Since $S_1, S_2$, and $S_3$ are consecutive extremal points of $\mathcal{S}$,
the vectors $\overrightarrow{S_2 S_1}$ and $\overrightarrow{S_2 S_3}$ are linearly independent.
Using Gram-Schmidt orthonormalization process, we obtain two points $S'_1$ and $S'_3$ 
of the plane $\mathcal{P}$ such that the vectors  $\overrightarrow{S_2 S'_1}$
and $\overrightarrow{S_2 S'_3}$ are orthonormal and for any
point $Q$ in plane $\mathcal{P}$, the vector $\overrightarrow{S_2 Q}$ can be 
uniquely written as a linear combination of these vectors.
Vectors $\overrightarrow{S_2 S'_1}$ and $\overrightarrow{S_2 S'_3}$ are given by
$$
\overrightarrow{S_2 S'_1}=\frac{\overrightarrow{S_2 S_1}}{\|\overrightarrow{S_2 S_1}\|_2}
\mbox{ and }\overrightarrow{S_2 S'_3}=\frac{\overrightarrow{S_2 S_3}-\langle \overrightarrow{S_2 S_3},\overrightarrow{S_2 S'_1}\rangle \overrightarrow{S_2 S'_1}}{\|\overrightarrow{S_2 S_3}-\langle \overrightarrow{S_2 S_3},\overrightarrow{S_2 S'_1}\rangle \overrightarrow{S_2 S'_1}\|_2}.
$$
It follows that if $A$ is the  $(3,2)$ matrix $[\overrightarrow{S_2 S'_1}, \overrightarrow{S_2 S'_3}]$ whose first
column is $\overrightarrow{S_2 S'_1}$ and whose second column is $\overrightarrow{S_2 S'_3}$, then the matrix
$A^\transp A$ is invertible and the projection $P_0=\pi_{\mathcal{P}}[P]$ of $P$ on $\mathcal{P}$ can be expressed
as 
\begin{equation}\label{projformPo}
\overrightarrow{S_2 P_0}=A (A^\transp A)^{-1} A^\transp \overrightarrow{S_2 P}.
\end{equation}
Before calling Algorithms 2 and 4, we need to reparametrize the problem: we write
$\overrightarrow{S_2 P_0}=x_{P_0} \overrightarrow{S_2 S'_1} + y_{P_0} \overrightarrow{S_2 S'_3}=
A (x_{P_0}; y_{P_0})$
and $\overrightarrow{S_2 S_i}=A (x_i; y_i)$ for $i=1,\ldots,n$.
In particular, we have $(x_1, y_1)=(\|\overrightarrow{S_2 S_1}\|_2, 0)$ and $(x_2, y_2)=(0, 0)$.
Since $A$ has rank 2, eventually after re-ordering the lines of $A$, 
we can assume that $A$ is of the form $A=[A_0;a_0]$ where $A_0$ is 
a $(2,2)$ invertible matrix with $A_0(1,1) \neq 0$. 
Using Gaussian elimination, the system 
$\overrightarrow{S_2 P_0}=A (x_{P_0}; y_{P_0})$ can be written
$\left[\begin{array}{l}\;U_0\\0\;\; 0\end{array}  \right] \left[ \begin{array}{l}x_{P_0}\\y_{P_0}\end{array}  \right]=
\left[\begin{array}{l}b\\ 0 \end{array}\right]$
for some two-dimensional vector $b$ and an invertible upper triangular matrix $U_0=\left[\begin{array}{ll}U_{1 1} & U_{1 2}\\0 &  U_{2 2} \end{array}\right]$.
Another by-product of Gaussian elimination is the lower triangular matrix $L_0=\left[\begin{array}{ll}1 & 0\\L_{2 1} &  1 \end{array}\right]$ such that $A=L_0 U_0$ is the
$LU$ decomposition of $A_0$. We obtain 
\begin{equation} \label{coordP0}
x_{P_0}  =   \frac{\overrightarrow{S_2 P_0}(1)}{U_{1 1}}\Big[1+ \frac{U_{1 2} L_{2 1}}{U_{2 2}} \Big]-\frac{U_{1 2}}{U_{1 1}} \overrightarrow{S_2 P_0}(2),\;
y_{P_0}  =   \frac{\overrightarrow{S_2 P_0}(2)- L_{2 1} \overrightarrow{S_2 P_0}(1)}{U_{22}},
\end{equation}
and for $i \geq 3$,
\begin{equation}\label{coordxi}
x_i  =   \frac{\overrightarrow{S_2 S_i}(1)}{U_{1 1}}\Big[1+ \frac{U_{1 2} L_{2 1}}{U_{2 2}} \Big]-\frac{U_{1 2}}{U_{1 1}} \overrightarrow{S_2 S_i}(2),\;
y_i  =   \frac{\overrightarrow{S_2 S_i}(2)- L_{2 1} \overrightarrow{S_2 S_i}(1)}{U_{22}}.
\end{equation}
Algorithms 2, 3, and 4 can now be used with $P$  replaced by $(x_{P_0}, y_{P_0})$ and where
the coordinates of the extremal points of the polygone are $(x_i, y_i), i=1,\ldots,n$.
First, Algorithm 2 is called to compute the area $\mathcal{L}$ of $\mathcal{S}$, the crossing number for
$P_0$ and $\mathcal{S}$,
and the minimal and maximal distances from $P_0$ to the boundary of $\mathcal{S}$, respectively denoted
by $d_{\min}$ and $d_{\max}$.
Recalling the definition \eqref{projformPo} of $P_0$, we introduce
\begin{equation} \label{formdmdM}
\begin{array}{l}
d_m = \sqrt{d_{\min}^2 + \|\overrightarrow{P P_0}\|_2^2} \mbox{ and }d_M = \sqrt{d_{\max}^2 + \|\overrightarrow{P P_0}\|_2^2}.
\end{array}
\end{equation}
With this notation, for $d \geq d_{M}$ or $d \leq 0$, the density is null and
if the crossing number is odd, i.e., if $P_0$ belongs to the relative interior of $\mathcal{S}$, then
for $0 \leq d \leq d_{m}$, we have $f_D(d)=\frac{2 \pi d}{\mathcal{L}}$.
Otherwise, if the crossing number is even, $f_D(d)$ is null for $0 \leq d \leq d_{m}$.\\ 

\par For $d_{m} \leq d \leq d_{M}$, 
Algorithm 6 provides approximations $\tilde f_D(d_i)$ of the value of the density at points
$d_i, i=1,\ldots,N-1$.\vspace*{0.4cm}\\
\rule{\linewidth}{1pt}
\par {\textbf{Algorithm 6: Computation of the approximate density of $D$ (distance from $P$ to a random variable
uniformly distributed in a polyhedron) in the range $[d_{m}, d_{M}]$.}}\\
\rule{\linewidth}{1pt}
\par {\textbf{Inputs:}} The vertices $S_1, \ldots, S_n$ of a polyhedron contained in a plane,
\par the point $P$, and the number $N$ of discretization points.\\
\par {\textbf{Initialization:}} Call Algorithm 2 with $P$  replaced by $(x_{P_0}, y_{P_0})$ (see equation 
\par \eqref{coordP0}) and where the coordinates of the extremal points of the polyhedron are 
\par $(x_i, y_i), i=1,\ldots,n$, given by \eqref{coordxi}. This will compute the area $\mathcal{L}$ of $\mathcal{S}$, the 
\par crossing number for $P_0$ and $\mathcal{S}$, and the minimal and maximal distances from $P_0$ 
\par to the boundary of $\mathcal{S}$, respectively denoted by $d_{\min}$ and $d_{\max}$.
\par {\tt{F\_Old}}$=0$. 
\par Compute $d_m$ and $d_M$ given by \eqref{formdmdM}.
\par {\textbf{For}} $i=1,\ldots,N-1$,
\par \hspace*{0.7cm}Compute $d_i=d_{m} + \frac{(d_{M}-d_{m}) i}{N}$.
\par \hspace*{0.7cm}Call Algorithm 4 with input variables the crossing number, $\mathcal{L}$, $d_{\min}$, $d_{\max}$,
\par \hspace*{0.7cm} and $d=\sqrt{d_i^2 - \|\overrightarrow{P P_0}\|_2^2}$ to compute $F_D( d_i )$.
\par \hspace*{0.7cm}Compute
${\tilde f}_D(d_i) = \frac{N \Big[ F_D( d_{i} )-{\tt{F\_Old}} \Big] }{d_{\max}-d_{\min}}$ and set ${\tt{F\_Old}}=F_D( d_i)$.
\par {\textbf{End For}}\\
\par {\textbf{Oututs:}} ${\tilde f}_D(d_i), i=1,\ldots,N-1$.\\
\rule{\linewidth}{1pt}

\section{Numerical experiments}\label{sec:numexp}

We use Algorithm 6 (refereed to as {\tt{Green}} in the sequel since it is based
on Green's formula) to obtain approximations of the density 
of $D$ when $X$ is uniformly distributed in some polyhedra $\mathcal{S}$.\footnote{The Matlab code implementing the computations of the densities
discussed in this paper as well as the Matlab code of the nummerical experiments
of this section are available at \url{https://github.com/vguigues/Areas_Library}.}
We compare the performance of this algorithm with another algorithm
discussed in \cite{guiguesmoacyr2019} which computes the area of the intersection
of a disk and a polygone using a triangulation of the
polygone (we refer to this algorithm as
{\tt{Triangulation}} in what follows). The area of the intersection is then obtained computing the sum of the
areas of intersection of the disk with the triangles of the triangulation.

\par We start considering for $D$ the distance from the center of a rectangle
with side lengths 1 and 0.8 to a random variable with uniform
distribution in this rectangle. The corresponding density is given 
in Figure \ref{figuredenspolyhedronrec}.
\begin{figure}
\centering
\includegraphics[scale=0.7]{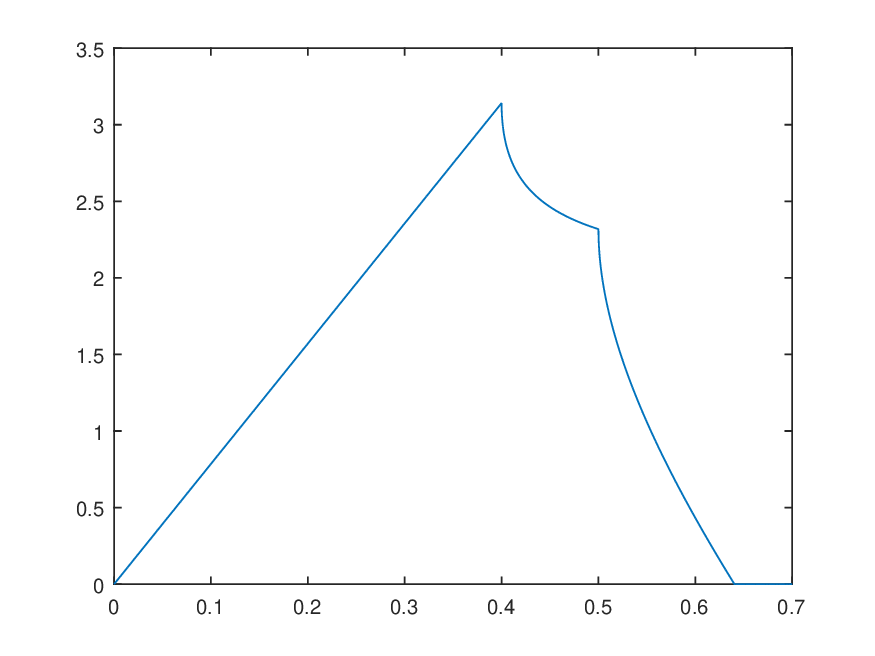}
\caption{Density of the distance from the center of a rectangle
with side lenghts 1 and 0.8 to a random variable with uniform
distribution in this rectangle.} 
\label{figuredenspolyhedronrec}
\end{figure} 
In this simple case, an analytic expression
of the density was given in \cite{stewartzhang2012} and we compare
the value of the density 
obtained using this analytic formula with the
approximations provided by our {\tt{Green}} and 
{\tt{Triangulation}} algorithms. 
The value of the density is computed at $N$ equally spaced  discretization points
$x_i,i=1,\ldots,N$,
from a set containing the support of $D$.
Varying $N$ in the set $\{10\,000, 20\,000, 50\,000, 100\,000\}$,
we obtain the maximal errors given in Table \ref{tablenum1}
where the maximal error is given by 
$\max_{i=1,\ldots,N} |f_D( x_i )-f_G( x_i )|$ and 
$\max_{i=1,\ldots,N} |f_D( x_i )-f_T( x_i )|$ for respectively
{\tt{Green}} and {\tt{Triangulation}} algorithms where 
$f_D$ stands for the density of $D$ given in \cite{stewartzhang2012}
and $f_G( x_ i)$ (resp. $f_T( x_i )$) is the approximation of the density computed by 
Algorithm {\tt{Green}} (resp. {\tt{Triangulation}}) at $x_i$.
In all cases the maximal error is very small which shows that 
{\tt{Green}} and {\tt{Triangulation}} algorithms  correctly compute 
the $N$ areas of  intersection of the disks and polygone of this 
example.\footnote{To approximate the density at 
$N$ points, we need to compute the cumulative distribution function 
at $N$ points and therefore when $N=100\,000$, Algorithms  {\tt{Green}} and {\tt{Triangulation}}
are called $100\,000$ times each to compute $100\,000$ areas.}
We also observe that the approximations are slightly better with
our algorithm {\tt{Green}} and, as expected, the maximal error decreases with 
$N$ for {\tt{Green}}. This is not the case for {\tt{Triangulation}}, probably
due to roundoff errors.\\

\begin{table}
\begin{tabular}{|c|c|c|}
\hline
Number $N$ of discretization points   & Maximal error - {\tt{Green}} & 
Maximal error - {\tt{Triangulation}}\\
\hline
$10\,000$ &  0.017  &  0.023  \\ 
\hline 
$20\,000$ &  0.010  &  0.020  \\
\hline
$50\,000$ &  0.007  & 0.04   \\
\hline   
$100\,000$ & 0.004   & 0.03   \\
\hline   
\end{tabular}
\caption{Maximal error obtained with {\tt{Green}} and 
{\tt{Triangulation}} algorithms computing the density of
$D$ ($D$ being the distance from the center of a rectangle
with side lengths 1 and 0.8 to a random variable with uniform
distribution in this rectangle) at $N$ discretization points.}\label{tablenum1}
\end{table}

\par We now compare algorithms {\tt{Green}} and {\tt{Triangulation}}
on 6 other examples. More precisely, we consider three polyhedra
(a triangle, a rectangle, and an arbitrary polygone) 
and in each case a point $P$
inside the polygone and a point $P$ outside, see
the left plots of Figures \ref{figuredenspolyhedron} and \ref{figuredenspolyhedron1}. 
\begin{figure}
\begin{tabular}{ll}
\input{Triangle.pstex_t} &  \input{Triangle_5_0B.pstex_t} \\
\input{TriangleIn.pstex_t} &  \input{Triangle_4_2B.pstex_t}\\
 \input{Rectangle.pstex_t} &  \input{Rectangle_1_1B.pstex_t}\\
 \input{RectangleIn.pstex_t} &  \input{Rectangle_6_5B.pstex_t}
\end{tabular}
\caption{Density of $D$ when $X$ is uniformly distributed in a polygone: some examples.} 
\label{figuredenspolyhedron}
\end{figure} 
\begin{figure}
\begin{tabular}{ll}
 \input{Poly_Ext.pstex_t} &  \input{Poly_40B.pstex_t}\\
 \input{Poly_Int.pstex_t} &  \input{Poly_43B.pstex_t}
\end{tabular}
\caption{Density of $D$ when $X$ is uniformly distributed in a polygone: some examples.} 
\label{figuredenspolyhedron1}
\end{figure} 
The values of the corresponding densities of $D$  
at a set of $N=10\,000$ equally spaced points $x_i,i=1,\ldots,N$, 
contained in the support of $D$,
were computed
using {\tt{Green}}
and {\tt{Triangulation}} algorithms and are represented in the right plots
of Figures \ref{figuredenspolyhedron} and \ref{figuredenspolyhedron1}.
The maximal errors $\max_{i=1,\ldots,N} |f_G(x_i) - f_T(x_i)|$
were 
$5.7\small{\times}10^{-10}$, $4.1\small{\times}10^{-8}$,
$8.6\small{\times}10^{-10}$,
$4.5\small{\times}10^{-10}$,
$2.3\small{\times}10^{-5}$, and
$1.7\small{\times}10^{-9}$
for the six examples (from top to bottom on Figures \ref{figuredenspolyhedron} and \ref{figuredenspolyhedron1}),
where $f_G(x_i)$ and $f_T(x_i)$ have the same meaning as before.
The fact that these errors are very small is an indication that
{\tt{Green}} and {\tt{Triangulation}} algorithms were correctly implemented.\\

\par Finally, we perform a last set of tests computing, 
using {\tt{Green}} and {\tt{Triangulation}} algorithms,
the 
areas of intersection of 350 disks and polyhedra as well 
as the mean and maximal time required to compute these areas.
The polyhedra and disks are generated as follows.
The coordinates of the centers of the disks (resp. the radii) 
are obtained sampling
independently from the uniform distribution on the interval
$[-100,100]$ (resp. $[50,250]$). To generate a polygone with
$4n$ vertices we sample $4 n$ points 
taking $n$ points in each orthant with polar angles generated
randomly and independently in this orthant and
radial coordinates generated randomly and independently in 
the interval [0,1000].
We then sort in ascending order the polar angles of these points.
This list defines the successive vertices of a 
star-shaped (simple) polygone.
An example of such a star-shaped polygone with $n=3$ and
$4n=12$ vertices is given in Figure \ref{figuredenspolyhedron11}, together
with a triangulation of this polygone.
\begin{figure}
\centering
\includegraphics[scale=0.7]{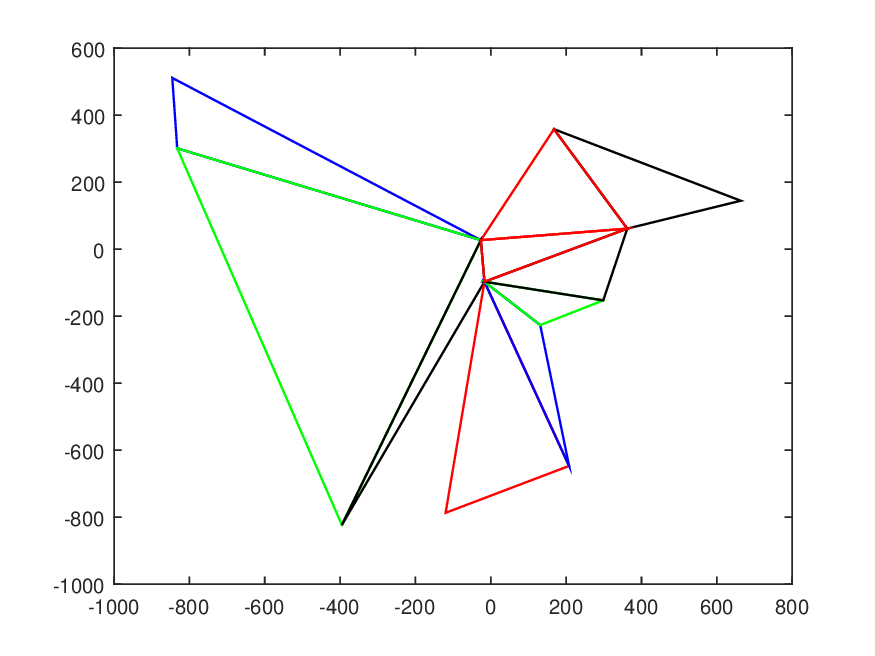}
\caption{Star-shaped polygone and a triangulation of this polygone.} 
\label{figuredenspolyhedron11}
\end{figure} 
For each value of $n$ in the set $\{10, 25, 50, 80, 100, 150, 200\}$
we generate 50 star-shaped polyhedra and disks as explained above and for each polygone
and disk, we compute the area of their intersection using {\tt{Green}} and {\tt{Triangulation}} algorithms.
For each value of $n$, the mean and maximal time (over the 50 instances)
required to compute these areas are reported in Table \ref{tablerandompoly}.
We also report in this table the mean and maximal errors defined respectively
by $\frac{1}{50}\sum_{i=1}^{50} |\mathcal{A}_G(i) - \mathcal{A}_T(i)|$  and  
$\max_{i=1,\ldots,50} |\mathcal{A}_G(i)-\mathcal{A}_T(i)|$ where 
$\mathcal{A}_G(i)$ and  $\mathcal{A}_T(i)$ are respectively the
areas  of the intersection for  instance $i$ computed with {\tt{Green}} and {\tt{Triangulation}} algorithms. 
We observe that these errors are negligible which shows that both algorithms compute
the same areas. Moreover, on all instances
{\tt{Green}} algorithm computes all areas
extremely quickly and much quicker than
{\tt{Triangulation}} algorithm.
For this latter algorithm,  both the mean
and maximal time required to compute the intersection areas significantly increase
with the number of vertices of the polygone.

\begin{table}
\begin{tabular}{|c|c|c|c|c|c|c|}
\hline
$4n$  &  Mean time-{\tt{Tr}} & Mean Time-{\tt{Gr}}& Max time-{\tt{Tr}}&  Max time-{\tt{Gr}}& Mean error &Max error\\
\hline
40  &   0.30   & 0.004  & 0.34  & 0.008 &  9.5$\small{\times}10^{-10}$  & $10^{-8}$    \\
\hline
100 &   1.99   & 0.007  & 2.54 & 0.014 &   2.5$\small{\times}10^{-9}$       & 4.1$\small{\times}10^{-8}$\\
\hline
200 &   8.09   & 0.012  & 8.96 & 0.018 & 3.8$\small{\times}10^{-9}$        & 3.6$\small{\times}10^{-8}$ \\
\hline
320 &   22.57  & 0.020  & 34.26 & 0.036 &   6.8$\small{\times}10^{-9}$     &  3.4$\small{\times}10^{-8}$ \\
\hline
400 &   45.21  & 0.021  & 669.76 & 0.039 &   1.1$\small{\times}10^{-8}$   & 1.7$\small{\times}10^{-7}$ \\ 
\hline
600 &   128.52 & 0.033  & 2 772.5 & 0.074 &   1.4$\small{\times}10^{-8}$  &1.2$\small{\times}10^{-7}$ \\
\hline
800 &  369.80  & 0.043  & 9 661.8 & 0.076 &  1.7$\small{\times}10^{-8}$   & 9.7$\small{\times}10^{-8}$\\
\hline
\end{tabular}
\caption{Mean and maximal time (in seconds) required to compute the areas of intersection of 50 polyhedra
with $4n$ vertices with disks
using {\tt{Green}} ({\tt{Gr}} for short in the table) and 
{\tt{Triangulation}} ({\tt{Tr}} for short in the table) algorithms.
The last two columns report respectively the mean and maximal errors.}\label{tablerandompoly}
\end{table}    

\section{Application to PSHA and extensions} \label{conclusion}

The results of Sections \ref{diskdist}, \ref{secball}, and \ref{secpoly} can be used 
to determine for the application presented in Section \ref{secapp} the distribution of the distance between the epicenter in $\mathcal{S}$  and
an arbitrary point $P$ when $\mathcal{S}$ is a union of disks, a union of balls, or the boundary of a polyhedron
in $\mathbb{R}^3$. For this application, the coordinates of $P$, of the centers of the disks and of two points on the
boundaries of these disks,
of the centers of the balls, and of the vertices 
$S_1, \ldots, S_n$ of the polyhedron are given providing for each point its latitude, its longitude, and its depth measured from the surface
of the earth. To apply the computations of the previous sections, we need to choose 
a Cartesian coordinate system and use the corresponding Cartesian coordinates of these points.
These coordinates are given as follows. We take for the positive $x$-axis the ray $OA$ where
$O$ is the center of the earth and $A$ is the point on the surface of the earth with longitude 0
and latitude 0. We take for the positive $z$-axis the ray $OB$ where $O$ is the center of the earth
and $B$ is the north pole. The positive $y$-axis is chosen correspondingly and corresponds to
ray $OC$ where $C$ is the point on the surface of the earth with latitude 0 and longitude 90$^o$ East.
Let $P$ be a point at depth $d$ from the surface of the earth with latitude $\varphi \in [0,90^{o}]$ (North or South) and longitude
$\lambda \in [0,180^{o}]$ (East or West). If the latitude is $\varphi$ North (resp. $\varphi$ South), we use
the notation $\varphi N$ (resp. $\varphi S$) while if the longitude is $\lambda$ East (resp. $\lambda$ West), we use
the notation $\lambda E$ (resp. $\lambda W$).
Denoting by $R$ the earth radius, the Cartesian coordinates of $P$ in the chosen Cartesian coordinate system are
$$
\begin{array}{l}
\Big((R-d)\cos \varphi \cos \lambda,(R-d)\cos \varphi \sin \lambda, (R-d)\sin \varphi \Big) \mbox{ if }P=(R-d,\lambda E, \varphi N),\\
\Big((R-d)\cos \varphi \cos \lambda,(R-d)\cos \varphi \sin \lambda, -(R-d)\sin \varphi \Big)\mbox{if }P=(R-d,\lambda E, \varphi S),\\
\Big((R-d)\cos \varphi \cos \lambda,-(R-d)\cos \varphi \sin \lambda, (R-d)\sin \varphi \Big)\mbox{if }P=(R-d,\lambda W, \varphi N),\\
\Big((R-d)\cos \varphi \cos \lambda,-(R-d)\cos \varphi \sin \lambda, -(R-d)\sin \varphi \Big)\mbox{if }P=(R-d,\lambda W, \varphi S).
\end{array}
$$

In the case where the $\ell_2$-norm is replaced by either the 
$\ell_1$-norm or the $\ell_{\infty}$-norm and when $\mathcal{S}$ is a union of disks contained in a plane
with $P$ in that plane, we can use the results of Section \ref{secpoly}. Indeed,
since the level curves of the $\ell_1$-norm and the $\ell_{\infty}$-norm in the plane are squares, to compute
the CDF of $D$ at a given point in these cases we need to determine the area of the intersection of a square (a particular polygone) with
disks. It also possible to extend Algorithm 5 to the case where
the $\ell_2$-norm is replaced by either the 
$\ell_1$-norm or the $\ell_{\infty}$-norm
and $\mathcal{S}$ is a union of simple polygones.

Another extension of interest is the case where $\mathcal{S}$ is an arbitrary polyhedron in $\mathbb{R}^3$.
In this case, the CDF and density of the corresponding random variable $D$ given by
$D(\omega)=\|\overrightarrow{P X(\omega)}\|_2$ for any $\omega \in \Omega$ can be approximated using
Monte Carlo methods. This is possible if we have at hand a black box able to decide if a given point in $\mathbb{R}^3$
belongs to polyhedron $\mathcal{S}$ or not.\\

\par {\textbf{Acknowledgments}} The author would like to thank Marlon Pirchiner who pointed out
useful references for PSHA. The author's research was 
partially supported by an FGV grant, CNPq grant 307287/2013-0, 
FAPERJ grants E-26/110.313/2014 and E-26/201.599/2014.

\addcontentsline{toc}{section}{References}
\bibliographystyle{plain}
\bibliography{CDF}

\end{document}

%% file: Intro.pstex_t
\begin{picture}(0,0)%
\includegraphics{Intro.pstex}%
\end{picture}%
\setlength{\unitlength}{2279sp}%
\begingroup\makeatletter\ifx\SetFigFont\undefined%
\gdef\SetFigFont#1#2#3#4#5{%
  \reset@font\fontsize{#1}{#2pt}%
  \fontfamily{#3}\fontseries{#4}\fontshape{#5}%
  \selectfont}%
\fi\endgroup%
\begin{picture}(7491,6136)(175,-6605)
\put(2926,-2491){\makebox(0,0)[lb]{\smash{{\SetFigFont{7}{8.4}{\rmdefault}{\mddefault}{\updefault}{\color[rgb]{0,0,0}$P$}%
}}}}
\put(7651,-2446){\makebox(0,0)[lb]{\smash{{\SetFigFont{7}{8.4}{\rmdefault}{\mddefault}{\updefault}{\color[rgb]{0,0,0}$P$}%
}}}}
\put(2791,-5551){\makebox(0,0)[lb]{\smash{{\SetFigFont{7}{8.4}{\rmdefault}{\mddefault}{\updefault}{\color[rgb]{0,0,0}$P$}%
}}}}
\put(2836,-1006){\makebox(0,0)[lb]{\smash{{\SetFigFont{9}{10.8}{\rmdefault}{\mddefault}{\updefault}{\color[rgb]{0,0,0}$\mathcal{S}$}%
}}}}
\put(7516,-1501){\makebox(0,0)[lb]{\smash{{\SetFigFont{9}{10.8}{\rmdefault}{\mddefault}{\updefault}{\color[rgb]{0,0,0}$\mathcal{S}$}%
}}}}
\put(1936,-4305){\makebox(0,0)[lb]{\smash{{\SetFigFont{9}{10.8}{\rmdefault}{\mddefault}{\updefault}{\color[rgb]{0,0,0}$\mathcal{S}$}%
}}}}
\put(6121,-3796){\makebox(0,0)[lb]{\smash{{\SetFigFont{9}{10.8}{\rmdefault}{\mddefault}{\updefault}{\color[rgb]{0,0,0}$\mathcal{S}$}%
}}}}
\put(7426,-5641){\makebox(0,0)[lb]{\smash{{\SetFigFont{7}{8.4}{\rmdefault}{\mddefault}{\updefault}{\color[rgb]{0,0,0}$P$}%
}}}}
\end{picture}%

%% file: Zones.pstex_t
\begin{picture}(0,0)%
\includegraphics{Zones.pstex}%
\end{picture}%
\setlength{\unitlength}{2072sp}%
\begingroup\makeatletter\ifx\SetFigFont\undefined%
\gdef\SetFigFont#1#2#3#4#5{%
  \reset@font\fontsize{#1}{#2pt}%
  \fontfamily{#3}\fontseries{#4}\fontshape{#5}%
  \selectfont}%
\fi\endgroup%
\begin{picture}(6582,5187)(256,-5224)
\put(2521,-2131){\makebox(0,0)[lb]{\smash{{\SetFigFont{9}{10.8}{\rmdefault}{\mddefault}{\updefault}{\color[rgb]{0,0,0}Zone 1}%
}}}}
\put(5626,-1411){\makebox(0,0)[lb]{\smash{{\SetFigFont{9}{10.8}{\rmdefault}{\mddefault}{\updefault}{\color[rgb]{0,0,0}$f_{M_3}, f_{D_3}$}%
}}}}
\put(1396,-4111){\makebox(0,0)[lb]{\smash{{\SetFigFont{9}{10.8}{\rmdefault}{\mddefault}{\updefault}{\color[rgb]{0,0,0}Zone 4}%
}}}}
\put(4141,-3346){\makebox(0,0)[lb]{\smash{{\SetFigFont{8}{9.6}{\rmdefault}{\mddefault}{\updefault}{\color[rgb]{0,0,0}$P$}%
}}}}
\put(1216,-196){\makebox(0,0)[lb]{\smash{{\SetFigFont{9}{10.8}{\rmdefault}{\mddefault}{\updefault}{\color[rgb]{0,0,0}Zone 2 (line segment)}%
}}}}
\put(5536,-601){\makebox(0,0)[lb]{\smash{{\SetFigFont{9}{10.8}{\rmdefault}{\mddefault}{\updefault}{\color[rgb]{0,0,0}Zone 3 (disk)}%
}}}}
\put(2476,-2491){\makebox(0,0)[lb]{\smash{{\SetFigFont{9}{10.8}{\rmdefault}{\mddefault}{\updefault}{\color[rgb]{0,0,0}(polyhedron)}%
}}}}
\put(1396,-4471){\makebox(0,0)[lb]{\smash{{\SetFigFont{9}{10.8}{\rmdefault}{\mddefault}{\updefault}{\color[rgb]{0,0,0}(polyhedron)}%
}}}}
\put(1711,-4876){\makebox(0,0)[lb]{\smash{{\SetFigFont{9}{10.8}{\rmdefault}{\mddefault}{\updefault}{\color[rgb]{0,0,0}$f_{M_4}, f_{D_4}, \lambda_4$}%
}}}}
\put(2521,-2896){\makebox(0,0)[lb]{\smash{{\SetFigFont{9}{10.8}{\rmdefault}{\mddefault}{\updefault}{\color[rgb]{0,0,0}$f_{M_1}, f_{D_1}, \lambda_1$}%
}}}}
\put(6031,-1771){\makebox(0,0)[lb]{\smash{{\SetFigFont{9}{10.8}{\rmdefault}{\mddefault}{\updefault}{\color[rgb]{0,0,0}$\lambda_3$}%
}}}}
\put(271,-871){\makebox(0,0)[lb]{\smash{{\SetFigFont{9}{10.8}{\rmdefault}{\mddefault}{\updefault}{\color[rgb]{0,0,0}$f_{M_2}, f_{D_2}, \lambda_2$}%
}}}}
\end{picture}%

%% file: Poisson.pstex_t
\begin{picture}(0,0)%
\includegraphics{Poisson.pstex}%
\end{picture}%
\setlength{\unitlength}{2279sp}%
\begingroup\makeatletter\ifx\SetFigFont\undefined%
\gdef\SetFigFont#1#2#3#4#5{%
  \reset@font\fontsize{#1}{#2pt}%
  \fontfamily{#3}\fontseries{#4}\fontshape{#5}%
  \selectfont}%
\fi\endgroup%
\begin{picture}(10035,3553)(-388,-5525)
\put(7786,-3346){\makebox(0,0)[lb]{\smash{{\SetFigFont{9}{10.8}{\rmdefault}{\mddefault}{\updefault}{\color[rgb]{0,0,0}Poisson process for }%
}}}}
\put(7786,-3706){\makebox(0,0)[lb]{\smash{{\SetFigFont{9}{10.8}{\rmdefault}{\mddefault}{\updefault}{\color[rgb]{0,0,0}the earthquakes of}%
}}}}
\put(7786,-4066){\makebox(0,0)[lb]{\smash{{\SetFigFont{9}{10.8}{\rmdefault}{\mddefault}{\updefault}{\color[rgb]{0,0,0}zone $i$, rate $\lambda_i$}%
}}}}
\put(-89,-5461){\makebox(0,0)[lb]{\smash{{\SetFigFont{9}{10.8}{\rmdefault}{\mddefault}{\updefault}{\color[rgb]{0,0,0}Poisson process for the earthquakes of zone $i$ causing $PGA>A^*$ at $P$, rate $\lambda_i p_i$}%
}}}}
\put(-373,-2131){\makebox(0,0)[lb]{\smash{{\SetFigFont{9}{10.8}{\rmdefault}{\mddefault}{\updefault}{\color[rgb]{0,0,0}Poisson process for the earthquakes of zone $i$ causing $PGA\leq A^*$ at $P$, rate $\lambda_i (1-p_i)$}%
}}}}
\end{picture}%

%% file: Lens.pstex_t
\begin{picture}(0,0)%
\includegraphics{Lens.pstex}%
\end{picture}%
\setlength{\unitlength}{2486sp}%
\begingroup\makeatletter\ifx\SetFigFont\undefined%
\gdef\SetFigFont#1#2#3#4#5{%
  \reset@font\fontsize{#1}{#2pt}%
  \fontfamily{#3}\fontseries{#4}\fontshape{#5}%
  \selectfont}%
\fi\endgroup%
\begin{picture}(6159,2688)(1447,-8380)
\put(2431,-6586){\makebox(0,0)[lb]{\smash{{\SetFigFont{10}{12.0}{\rmdefault}{\mddefault}{\updefault}{\color[rgb]{0,0,0}$R$}%
}}}}
\put(6526,-6496){\makebox(0,0)[lb]{\smash{{\SetFigFont{10}{12.0}{\rmdefault}{\mddefault}{\updefault}{\color[rgb]{0,0,0}$R$}%
}}}}
\put(2836,-6901){\makebox(0,0)[lb]{\smash{{\SetFigFont{9}{10.8}{\rmdefault}{\mddefault}{\updefault}{\color[rgb]{0,0,0}$\theta$}%
}}}}
\put(6436,-7099){\makebox(0,0)[lb]{\smash{{\SetFigFont{9}{10.8}{\rmdefault}{\mddefault}{\updefault}{\color[rgb]{0,0,0}$\theta$}%
}}}}
\end{picture}%

%% file: Circle_Circle.pstex_t
\begin{picture}(0,0)%
\includegraphics{Circle_Circle.pstex}%
\end{picture}%
\setlength{\unitlength}{2901sp}%
\begingroup\makeatletter\ifx\SetFigFont\undefined%
\gdef\SetFigFont#1#2#3#4#5{%
  \reset@font\fontsize{#1}{#2pt}%
  \fontfamily{#3}\fontseries{#4}\fontshape{#5}%
  \selectfont}%
\fi\endgroup%
\begin{picture}(6949,6136)(1091,-7179)
\put(4681,-3616){\makebox(0,0)[lb]{\smash{{\SetFigFont{11}{13.2}{\rmdefault}{\mddefault}{\updefault}{\color[rgb]{0,0,0}$d$}%
}}}}
\put(3511,-4021){\makebox(0,0)[lb]{\smash{{\SetFigFont{7}{8.4}{\rmdefault}{\mddefault}{\updefault}{\color[rgb]{0,0,0}$h_2$}%
}}}}
\put(2755,-4021){\makebox(0,0)[lb]{\smash{{\SetFigFont{7}{8.4}{\rmdefault}{\mddefault}{\updefault}{\color[rgb]{0,0,0}$\theta/2$}%
}}}}
\put(3781,-7081){\makebox(0,0)[lb]{\smash{{\SetFigFont{14}{16.8}{\rmdefault}{\mddefault}{\updefault}{\color[rgb]{0,0,0}$R_1$}%
}}}}
\put(3259,-4021){\makebox(0,0)[lb]{\smash{{\SetFigFont{7}{8.4}{\rmdefault}{\mddefault}{\updefault}{\color[rgb]{0,0,0}$h_1$}%
}}}}
\put(2026,-5866){\makebox(0,0)[lb]{\smash{{\SetFigFont{11}{13.2}{\rmdefault}{\mddefault}{\updefault}{\color[rgb]{0,0,0}$x^*$}%
}}}}
\put(1554,-1929){\makebox(0,0)[lb]{\smash{{\SetFigFont{7}{8.4}{\rmdefault}{\mddefault}{\updefault}{\color[rgb]{0,0,0}$h_1$}%
}}}}
\put(2701,-1929){\makebox(0,0)[lb]{\smash{{\SetFigFont{7}{8.4}{\rmdefault}{\mddefault}{\updefault}{\color[rgb]{0,0,0}$h_2$}%
}}}}
\put(2872,-6001){\makebox(0,0)[lb]{\smash{{\SetFigFont{11}{13.2}{\rmdefault}{\mddefault}{\updefault}{\color[rgb]{0,0,0}$x^*$}%
}}}}
\put(2161,-4201){\makebox(0,0)[lb]{\smash{{\SetFigFont{12}{14.4}{\rmdefault}{\mddefault}{\updefault}{\color[rgb]{0,0,0}$S_0$}%
}}}}
\put(2476,-3661){\makebox(0,0)[lb]{\smash{{\SetFigFont{14}{16.8}{\rmdefault}{\mddefault}{\updefault}{\color[rgb]{0,0,0}$R_0$}%
}}}}
\end{picture}%

%% file: Circle_Circle_1.pstex_t
\begin{picture}(0,0)%
\includegraphics{Circle_Circle_1.pstex}%
\end{picture}%
\setlength{\unitlength}{2486sp}%
\begingroup\makeatletter\ifx\SetFigFont\undefined%
\gdef\SetFigFont#1#2#3#4#5{%
  \reset@font\fontsize{#1}{#2pt}%
  \fontfamily{#3}\fontseries{#4}\fontshape{#5}%
  \selectfont}%
\fi\endgroup%
\begin{picture}(10030,4213)(1532,-5669)
\put(3691,-3864){\makebox(0,0)[lb]{\smash{{\SetFigFont{10}{12.0}{\rmdefault}{\mddefault}{\updefault}{\color[rgb]{0,0,0}$d$}%
}}}}
\put(3286,-3751){\makebox(0,0)[lb]{\smash{{\SetFigFont{10}{12.0}{\rmdefault}{\mddefault}{\updefault}{\color[rgb]{0,0,0}$R_1$}%
}}}}
\put(4546,-5605){\makebox(0,0)[lb]{\smash{{\SetFigFont{11}{13.2}{\rmdefault}{\mddefault}{\updefault}{\color[rgb]{0,0,0}$R_1+d > R_0$ and $x^* \geq 0$}%
}}}}
\put(6211,-1825){\makebox(0,0)[lb]{\smash{{\SetFigFont{10}{12.0}{\rmdefault}{\mddefault}{\updefault}{\color[rgb]{0,0,0}$h_1$}%
}}}}
\put(6256,-4660){\makebox(0,0)[lb]{\smash{{\SetFigFont{10}{12.0}{\rmdefault}{\mddefault}{\updefault}{\color[rgb]{0,0,0}$x^*$}%
}}}}
\put(6211,-3670){\makebox(0,0)[lb]{\smash{{\SetFigFont{10}{12.0}{\rmdefault}{\mddefault}{\updefault}{\color[rgb]{0,0,0}$R_1$}%
}}}}
\put(5986,-2950){\makebox(0,0)[lb]{\smash{{\SetFigFont{10}{12.0}{\rmdefault}{\mddefault}{\updefault}{\color[rgb]{0,0,0}$R_0$}%
}}}}
\put(6526,-3130){\makebox(0,0)[lb]{\smash{{\SetFigFont{10}{12.0}{\rmdefault}{\mddefault}{\updefault}{\color[rgb]{0,0,0}$d$}%
}}}}
\put(6976,-4606){\makebox(0,0)[lb]{\smash{{\SetFigFont{10}{12.0}{\rmdefault}{\mddefault}{\updefault}{\color[rgb]{0,0,0}$h_2$}%
}}}}
\put(2026,-5605){\makebox(0,0)[lb]{\smash{{\SetFigFont{11}{13.2}{\rmdefault}{\mddefault}{\updefault}{\color[rgb]{0,0,0}$R_1+d \leq R_0$}%
}}}}
\put(9856,-3706){\makebox(0,0)[lb]{\smash{{\SetFigFont{10}{12.0}{\rmdefault}{\mddefault}{\updefault}{\color[rgb]{0,0,0}$P$}%
}}}}
\put(9451,-3706){\makebox(0,0)[lb]{\smash{{\SetFigFont{10}{12.0}{\rmdefault}{\mddefault}{\updefault}{\color[rgb]{0,0,0}$R_1$}%
}}}}
\put(9451,-2941){\makebox(0,0)[lb]{\smash{{\SetFigFont{10}{12.0}{\rmdefault}{\mddefault}{\updefault}{\color[rgb]{0,0,0}$d$}%
}}}}
\put(8844,-5146){\makebox(0,0)[lb]{\smash{{\SetFigFont{10}{12.0}{\rmdefault}{\mddefault}{\updefault}{\color[rgb]{0,0,0}$x^*$}%
}}}}
\put(8326,-5146){\makebox(0,0)[lb]{\smash{{\SetFigFont{10}{12.0}{\rmdefault}{\mddefault}{\updefault}{\color[rgb]{0,0,0}$h_1$}%
}}}}
\put(8056,-5605){\makebox(0,0)[lb]{\smash{{\SetFigFont{11}{13.2}{\rmdefault}{\mddefault}{\updefault}{\color[rgb]{0,0,0}$R_1+d > R_0$ and $x^* < 0$}%
}}}}
\put(9541,-1591){\makebox(0,0)[lb]{\smash{{\SetFigFont{10}{12.0}{\rmdefault}{\mddefault}{\updefault}{\color[rgb]{0,0,0}$h_2$}%
}}}}
\put(8722,-3166){\makebox(0,0)[lb]{\smash{{\SetFigFont{10}{12.0}{\rmdefault}{\mddefault}{\updefault}{\color[rgb]{0,0,0}$R_0$}%
}}}}
\put(2629,-3535){\makebox(0,0)[lb]{\smash{{\SetFigFont{10}{12.0}{\rmdefault}{\mddefault}{\updefault}{\color[rgb]{0,0,0}$S_0$}%
}}}}
\put(5680,-3373){\makebox(0,0)[lb]{\smash{{\SetFigFont{10}{12.0}{\rmdefault}{\mddefault}{\updefault}{\color[rgb]{0,0,0}$S_0$}%
}}}}
\put(8866,-3603){\makebox(0,0)[lb]{\smash{{\SetFigFont{10}{12.0}{\rmdefault}{\mddefault}{\updefault}{\color[rgb]{0,0,0}$S_0$}%
}}}}
\put(3309,-2986){\makebox(0,0)[lb]{\smash{{\SetFigFont{10}{12.0}{\rmdefault}{\mddefault}{\updefault}{\color[rgb]{0,0,0}$R_0$}%
}}}}
\end{picture}%

%% file: Disk1.pstex_t
\begin{picture}(0,0)%
\includegraphics{Disk1.pstex}%
\end{picture}%
\setlength{\unitlength}{1381sp}%
\begingroup\makeatletter\ifx\SetFigFont\undefined%
\gdef\SetFigFont#1#2#3#4#5{%
  \reset@font\fontsize{#1}{#2pt}%
  \fontfamily{#3}\fontseries{#4}\fontshape{#5}%
  \selectfont}%
\fi\endgroup%
\begin{picture}(8238,6165)(811,-11701)
\put(8326,-11161){\makebox(0,0)[lb]{\smash{{\SetFigFont{6}{7.2}{\rmdefault}{\mddefault}{\updefault}{\color[rgb]{0,0,0}$d$}%
}}}}
\put(1426,-8761){\makebox(0,0)[lb]{\smash{{\SetFigFont{6}{7.2}{\rmdefault}{\mddefault}{\updefault}{\color[rgb]{0,0,0}$R_0=1, R_1=0$}%
}}}}
\put(1351,-5986){\makebox(0,0)[lb]{\smash{{\SetFigFont{6}{7.2}{\rmdefault}{\mddefault}{\updefault}{\color[rgb]{0,0,0}$f_D(d)$}%
}}}}
\end{picture}%

%% file: Disk2.pstex_t
\begin{picture}(0,0)%
\includegraphics{Disk2.pstex}%
\end{picture}%
\setlength{\unitlength}{1381sp}%
\begingroup\makeatletter\ifx\SetFigFont\undefined%
\gdef\SetFigFont#1#2#3#4#5{%
  \reset@font\fontsize{#1}{#2pt}%
  \fontfamily{#3}\fontseries{#4}\fontshape{#5}%
  \selectfont}%
\fi\endgroup%
\begin{picture}(8226,6126)(811,-11701)
\put(8326,-11236){\makebox(0,0)[lb]{\smash{{\SetFigFont{6}{7.2}{\rmdefault}{\mddefault}{\updefault}{\color[rgb]{0,0,0}$d$}%
}}}}
\put(1351,-5986){\makebox(0,0)[lb]{\smash{{\SetFigFont{6}{7.2}{\rmdefault}{\mddefault}{\updefault}{\color[rgb]{0,0,0}$f_D(d)$}%
}}}}
\put(3901,-9511){\makebox(0,0)[lb]{\smash{{\SetFigFont{6}{7.2}{\rmdefault}{\mddefault}{\updefault}{\color[rgb]{0,0,0}$R_0=1, R_1=0.5$}%
}}}}
\end{picture}%

%% file: Disk3.pstex_t
\begin{picture}(0,0)%
\includegraphics{Disk3.pstex}%
\end{picture}%
\setlength{\unitlength}{1381sp}%
\begingroup\makeatletter\ifx\SetFigFont\undefined%
\gdef\SetFigFont#1#2#3#4#5{%
  \reset@font\fontsize{#1}{#2pt}%
  \fontfamily{#3}\fontseries{#4}\fontshape{#5}%
  \selectfont}%
\fi\endgroup%
\begin{picture}(8316,6150)(736,-11701)
\put(8401,-11311){\makebox(0,0)[lb]{\smash{{\SetFigFont{6}{7.2}{\rmdefault}{\mddefault}{\updefault}{\color[rgb]{0,0,0}$d$}%
}}}}
\put(3901,-9736){\makebox(0,0)[lb]{\smash{{\SetFigFont{6}{7.2}{\rmdefault}{\mddefault}{\updefault}{\color[rgb]{0,0,0}$R_0=1, R_1=0.75$}%
}}}}
\put(1351,-5986){\makebox(0,0)[lb]{\smash{{\SetFigFont{6}{7.2}{\rmdefault}{\mddefault}{\updefault}{\color[rgb]{0,0,0}$f_D(d)$}%
}}}}
\end{picture}%

%% file: Disk4.pstex_t
\begin{picture}(0,0)%
\includegraphics{Disk4.pstex}%
\end{picture}%
\setlength{\unitlength}{1381sp}%
\begingroup\makeatletter\ifx\SetFigFont\undefined%
\gdef\SetFigFont#1#2#3#4#5{%
  \reset@font\fontsize{#1}{#2pt}%
  \fontfamily{#3}\fontseries{#4}\fontshape{#5}%
  \selectfont}%
\fi\endgroup%
\begin{picture}(8254,6240)(736,-11776)
\put(1351,-5986){\makebox(0,0)[lb]{\smash{{\SetFigFont{6}{7.2}{\rmdefault}{\mddefault}{\updefault}{\color[rgb]{0,0,0}$f_D(d)$}%
}}}}
\put(4051,-8911){\makebox(0,0)[lb]{\smash{{\SetFigFont{6}{7.2}{\rmdefault}{\mddefault}{\updefault}{\color[rgb]{0,0,0}$R_0=1, R_1=6$}%
}}}}
\put(8401,-11236){\makebox(0,0)[lb]{\smash{{\SetFigFont{6}{7.2}{\rmdefault}{\mddefault}{\updefault}{\color[rgb]{0,0,0}$d$}%
}}}}
\end{picture}%

%% file: Circle_R3.pstex_t
\begin{picture}(0,0)%
\includegraphics{Circle_R3.pstex}%
\end{picture}%
\setlength{\unitlength}{1657sp}%
\begingroup\makeatletter\ifx\SetFigFont\undefined%
\gdef\SetFigFont#1#2#3#4#5{%
  \reset@font\fontsize{#1}{#2pt}%
  \fontfamily{#3}\fontseries{#4}\fontshape{#5}%
  \selectfont}%
\fi\endgroup%
\begin{picture}(7086,4368)(283,-5665)
\put(3466,-3481){\makebox(0,0)[lb]{\smash{{\SetFigFont{7}{8.4}{\rmdefault}{\mddefault}{\updefault}{\color[rgb]{0,0,0}$d$}%
}}}}
\put(4456,-4201){\makebox(0,0)[lb]{\smash{{\SetFigFont{7}{8.4}{\rmdefault}{\mddefault}{\updefault}{\color[rgb]{0,0,0}Cut of the ball $\mathcal{B}(P,d)$}%
}}}}
\put(2971,-3886){\makebox(0,0)[lb]{\smash{{\SetFigFont{7}{8.4}{\rmdefault}{\mddefault}{\updefault}{\color[rgb]{0,0,0}$P$}%
}}}}
\put(6481,-2626){\makebox(0,0)[lb]{\smash{{\SetFigFont{7}{8.4}{\rmdefault}{\mddefault}{\updefault}{\color[rgb]{0,0,0}$\mathcal{P}$}%
}}}}
\put(5298,-3144){\makebox(0,0)[lb]{\smash{{\SetFigFont{7}{8.4}{\rmdefault}{\mddefault}{\updefault}{\color[rgb]{0,0,0}$S_0$}%
}}}}
\put(2566,-3076){\makebox(0,0)[lb]{\smash{{\SetFigFont{7}{8.4}{\rmdefault}{\mddefault}{\updefault}{\color[rgb]{0,0,0}$P_0$}%
}}}}
\put(3241,-5596){\makebox(0,0)[lb]{\smash{{\SetFigFont{7}{8.4}{\rmdefault}{\mddefault}{\updefault}{\color[rgb]{0,0,0}$R(d)=\sqrt{d^ 2-\|\overrightarrow{P P_0}\|_2^2}$}%
}}}}
\put(3376,-1456){\makebox(0,0)[lb]{\smash{{\SetFigFont{7}{8.4}{\rmdefault}{\mddefault}{\updefault}{\color[rgb]{0,0,0}$R_0=\|\overrightarrow{S_0 S_1}\|_2$}%
}}}}
\end{picture}%

%% file: Ball1.pstex_t
\begin{picture}(0,0)%
\includegraphics{Ball1.pstex}%
\end{picture}%
\setlength{\unitlength}{1381sp}%
\begingroup\makeatletter\ifx\SetFigFont\undefined%
\gdef\SetFigFont#1#2#3#4#5{%
  \reset@font\fontsize{#1}{#2pt}%
  \fontfamily{#3}\fontseries{#4}\fontshape{#5}%
  \selectfont}%
\fi\endgroup%
\begin{picture}(8226,6169)(826,-11701)
\put(8251,-11161){\makebox(0,0)[lb]{\smash{{\SetFigFont{5}{6.0}{\rmdefault}{\mddefault}{\updefault}{\color[rgb]{0,0,0}$d$}%
}}}}
\put(1426,-5986){\makebox(0,0)[lb]{\smash{{\SetFigFont{5}{6.0}{\rmdefault}{\mddefault}{\updefault}{\color[rgb]{0,0,0}$f_D(d)$}%
}}}}
\put(1501,-8311){\makebox(0,0)[lb]{\smash{{\SetFigFont{6}{7.2}{\rmdefault}{\mddefault}{\updefault}{\color[rgb]{0,0,0}$R_0=1, R_1=0$}%
}}}}
\end{picture}%

%% file: Ball2.pstex_t
\begin{picture}(0,0)%
\includegraphics{Ball2.pstex}%
\end{picture}%
\setlength{\unitlength}{1381sp}%
\begingroup\makeatletter\ifx\SetFigFont\undefined%
\gdef\SetFigFont#1#2#3#4#5{%
  \reset@font\fontsize{#1}{#2pt}%
  \fontfamily{#3}\fontseries{#4}\fontshape{#5}%
  \selectfont}%
\fi\endgroup%
\begin{picture}(8211,6165)(826,-11701)
\put(8326,-11161){\makebox(0,0)[lb]{\smash{{\SetFigFont{6}{7.2}{\rmdefault}{\mddefault}{\updefault}{\color[rgb]{0,0,0}$d$}%
}}}}
\put(1351,-5986){\makebox(0,0)[lb]{\smash{{\SetFigFont{6}{7.2}{\rmdefault}{\mddefault}{\updefault}{\color[rgb]{0,0,0}$f_D(d)$}%
}}}}
\put(6076,-6061){\makebox(0,0)[lb]{\smash{{\SetFigFont{6}{7.2}{\rmdefault}{\mddefault}{\updefault}{\color[rgb]{0,0,0}$R_0=1, R_1=0.5$}%
}}}}
\end{picture}%

%% file: Ball3.pstex_t
\begin{picture}(0,0)%
\includegraphics{Ball3.pstex}%
\end{picture}%
\setlength{\unitlength}{1381sp}%
\begingroup\makeatletter\ifx\SetFigFont\undefined%
\gdef\SetFigFont#1#2#3#4#5{%
  \reset@font\fontsize{#1}{#2pt}%
  \fontfamily{#3}\fontseries{#4}\fontshape{#5}%
  \selectfont}%
\fi\endgroup%
\begin{picture}(8211,6150)(826,-11701)
\put(1426,-6061){\makebox(0,0)[lb]{\smash{{\SetFigFont{6}{7.2}{\rmdefault}{\mddefault}{\updefault}{\color[rgb]{0,0,0}$f_D(d)$}%
}}}}
\put(8326,-11236){\makebox(0,0)[lb]{\smash{{\SetFigFont{6}{7.2}{\rmdefault}{\mddefault}{\updefault}{\color[rgb]{0,0,0}$d$}%
}}}}
\put(4501,-9436){\makebox(0,0)[lb]{\smash{{\SetFigFont{6}{7.2}{\rmdefault}{\mddefault}{\updefault}{\color[rgb]{0,0,0}$R_0=1, R_1=0.75$}%
}}}}
\end{picture}%

%% file: Ball4.pstex_t
\begin{picture}(0,0)%
\includegraphics{Ball4.pstex}%
\end{picture}%
\setlength{\unitlength}{1381sp}%
\begingroup\makeatletter\ifx\SetFigFont\undefined%
\gdef\SetFigFont#1#2#3#4#5{%
  \reset@font\fontsize{#1}{#2pt}%
  \fontfamily{#3}\fontseries{#4}\fontshape{#5}%
  \selectfont}%
\fi\endgroup%
\begin{picture}(8221,6126)(730,-11701)
\put(8326,-11236){\makebox(0,0)[lb]{\smash{{\SetFigFont{6}{7.2}{\rmdefault}{\mddefault}{\updefault}{\color[rgb]{0,0,0}$d$}%
}}}}
\put(1351,-5986){\makebox(0,0)[lb]{\smash{{\SetFigFont{6}{7.2}{\rmdefault}{\mddefault}{\updefault}{\color[rgb]{0,0,0}$f_D(d)$}%
}}}}
\put(6226,-6136){\makebox(0,0)[lb]{\smash{{\SetFigFont{6}{7.2}{\rmdefault}{\mddefault}{\updefault}{\color[rgb]{0,0,0}$R_0=1, R_1=6$}%
}}}}
\end{picture}%

%% file: Line_Cases.pstex_t
\begin{picture}(0,0)%
\includegraphics{Line_Cases.pstex}%
\end{picture}%
\setlength{\unitlength}{1865sp}%
\begingroup\makeatletter\ifx\SetFigFont\undefined%
\gdef\SetFigFont#1#2#3#4#5{%
  \reset@font\fontsize{#1}{#2pt}%
  \fontfamily{#3}\fontseries{#4}\fontshape{#5}%
  \selectfont}%
\fi\endgroup%
\begin{picture}(11883,3766)(301,-4175)
\put(3871,-2446){\makebox(0,0)[lb]{\smash{{\SetFigFont{8}{9.6}{\rmdefault}{\mddefault}{\updefault}{\color[rgb]{0,0,0}$B$}%
}}}}
\put(7966,-2446){\makebox(0,0)[lb]{\smash{{\SetFigFont{8}{9.6}{\rmdefault}{\mddefault}{\updefault}{\color[rgb]{0,0,0}$B$}%
}}}}
\put(8461,-2446){\makebox(0,0)[lb]{\smash{{\SetFigFont{8}{9.6}{\rmdefault}{\mddefault}{\updefault}{\color[rgb]{0,0,0}$A$}%
}}}}
\put(10171,-2446){\makebox(0,0)[lb]{\smash{{\SetFigFont{8}{9.6}{\rmdefault}{\mddefault}{\updefault}{\color[rgb]{0,0,0}$B$}%
}}}}
\put(1801,-3301){\makebox(0,0)[lb]{\smash{{\SetFigFont{8}{9.6}{\rmdefault}{\mddefault}{\updefault}{\color[rgb]{0,0,0}$P$}%
}}}}
\put(5896,-3391){\makebox(0,0)[lb]{\smash{{\SetFigFont{8}{9.6}{\rmdefault}{\mddefault}{\updefault}{\color[rgb]{0,0,0}$P$}%
}}}}
\put(2098,-2761){\makebox(0,0)[lb]{\smash{{\SetFigFont{8}{9.6}{\rmdefault}{\mddefault}{\updefault}{\color[rgb]{0,0,0}$d$}%
}}}}
\put(6751,-2806){\makebox(0,0)[lb]{\smash{{\SetFigFont{8}{9.6}{\rmdefault}{\mddefault}{\updefault}{\color[rgb]{0,0,0}$d$}%
}}}}
\put(11926,-3301){\makebox(0,0)[lb]{\smash{{\SetFigFont{8}{9.6}{\rmdefault}{\mddefault}{\updefault}{\color[rgb]{0,0,0}$P$}%
}}}}
\put(316,-2446){\makebox(0,0)[lb]{\smash{{\SetFigFont{8}{9.6}{\rmdefault}{\mddefault}{\updefault}{\color[rgb]{0,0,0}$A$}%
}}}}
\put(1171,-2356){\makebox(0,0)[lb]{\smash{{\SetFigFont{8}{9.6}{\rmdefault}{\mddefault}{\updefault}{\color[rgb]{0,0,0}$P_0$}%
}}}}
\put(1441,-4111){\makebox(0,0)[lb]{\smash{{\SetFigFont{8}{9.6}{\rmdefault}{\mddefault}{\updefault}{\color[rgb]{0,0,0}$(A)$}%
}}}}
\put(5446,-4111){\makebox(0,0)[lb]{\smash{{\SetFigFont{8}{9.6}{\rmdefault}{\mddefault}{\updefault}{\color[rgb]{0,0,0}$(B)$}%
}}}}
\put(10171,-4111){\makebox(0,0)[lb]{\smash{{\SetFigFont{8}{9.6}{\rmdefault}{\mddefault}{\updefault}{\color[rgb]{0,0,0}$(C)$}%
}}}}
\put(10576,-1276){\makebox(0,0)[lb]{\smash{{\SetFigFont{8}{9.6}{\rmdefault}{\mddefault}{\updefault}{\color[rgb]{0,0,0}$R(d_{\min})$}%
}}}}
\put(4366,-2446){\makebox(0,0)[lb]{\smash{{\SetFigFont{8}{9.6}{\rmdefault}{\mddefault}{\updefault}{\color[rgb]{0,0,0}$A$}%
}}}}
\put(10936,-2491){\makebox(0,0)[lb]{\smash{{\SetFigFont{8}{9.6}{\rmdefault}{\mddefault}{\updefault}{\color[rgb]{0,0,0}$d_{\min}$}%
}}}}
\put(1666,-556){\makebox(0,0)[lb]{\smash{{\SetFigFont{8}{9.6}{\rmdefault}{\mddefault}{\updefault}{\color[rgb]{0,0,0}$R(d)$}%
}}}}
\put(5851,-1951){\makebox(0,0)[lb]{\smash{{\SetFigFont{8}{9.6}{\rmdefault}{\mddefault}{\updefault}{\color[rgb]{0,0,0}$P_0$}%
}}}}
\put(6256,-826){\makebox(0,0)[lb]{\smash{{\SetFigFont{8}{9.6}{\rmdefault}{\mddefault}{\updefault}{\color[rgb]{0,0,0}$R(d)$}%
}}}}
\put(847,-2761){\makebox(0,0)[lb]{\smash{{\SetFigFont{8}{9.6}{\rmdefault}{\mddefault}{\updefault}{\color[rgb]{0,0,0}$d$}%
}}}}
\put(9406,-2536){\makebox(0,0)[lb]{\smash{{\SetFigFont{8}{9.6}{\rmdefault}{\mddefault}{\updefault}{\color[rgb]{0,0,0}$d$}%
}}}}
\put(11836,-1906){\makebox(0,0)[lb]{\smash{{\SetFigFont{8}{9.6}{\rmdefault}{\mddefault}{\updefault}{\color[rgb]{0,0,0}$P_0$}%
}}}}
\put(4591,-1456){\makebox(0,0)[lb]{\smash{{\SetFigFont{8}{9.6}{\rmdefault}{\mddefault}{\updefault}{\color[rgb]{0,0,0}$\|\overrightarrow{A P_0}\|_2$}%
}}}}
\put(10261,-610){\makebox(0,0)[lb]{\smash{{\SetFigFont{8}{9.6}{\rmdefault}{\mddefault}{\updefault}{\color[rgb]{0,0,0}$R(d)$}%
}}}}
\put(946,-1501){\makebox(0,0)[lb]{\smash{{\SetFigFont{8}{9.6}{\rmdefault}{\mddefault}{\updefault}{\color[rgb]{0,0,0}$R(d)$}%
}}}}
\end{picture}%

%% file: Line1.pstex_t
\begin{picture}(0,0)%
\includegraphics{Line1.pstex}%
\end{picture}%
\setlength{\unitlength}{1381sp}%
\begingroup\makeatletter\ifx\SetFigFont\undefined%
\gdef\SetFigFont#1#2#3#4#5{%
  \reset@font\fontsize{#1}{#2pt}%
  \fontfamily{#3}\fontseries{#4}\fontshape{#5}%
  \selectfont}%
\fi\endgroup%
\begin{picture}(8303,6150)(736,-11701)
\put(1351,-5986){\makebox(0,0)[lb]{\smash{{\SetFigFont{6}{7.2}{\rmdefault}{\mddefault}{\updefault}{\color[rgb]{0,0,0}$f_D(d)$}%
}}}}
\put(8326,-11236){\makebox(0,0)[lb]{\smash{{\SetFigFont{6}{7.2}{\rmdefault}{\mddefault}{\updefault}{\color[rgb]{0,0,0}$d$}%
}}}}
\put(3001,-5986){\makebox(0,0)[lb]{\smash{{\SetFigFont{6}{7.2}{\rmdefault}{\mddefault}{\updefault}{\color[rgb]{0,0,0}Density of the distance to}%
}}}}
\put(3001,-6436){\makebox(0,0)[lb]{\smash{{\SetFigFont{6}{7.2}{\rmdefault}{\mddefault}{\updefault}{\color[rgb]{0,0,0}a random variable uniformly}%
}}}}
\put(3001,-6886){\makebox(0,0)[lb]{\smash{{\SetFigFont{6}{7.2}{\rmdefault}{\mddefault}{\updefault}{\color[rgb]{0,0,0}distributed on a line segment $\overline{AB}$}%
}}}}
\put(3001,-7336){\makebox(0,0)[lb]{\smash{{\SetFigFont{6}{7.2}{\rmdefault}{\mddefault}{\updefault}{\color[rgb]{0,0,0}with $A=(0;0), B=(10;0)$}%
}}}}
\put(3001,-7786){\makebox(0,0)[lb]{\smash{{\SetFigFont{6}{7.2}{\rmdefault}{\mddefault}{\updefault}{\color[rgb]{0,0,0}$P=(2;0)$}%
}}}}
\end{picture}%

%% file: Line2.pstex_t
\begin{picture}(0,0)%
\includegraphics{Line2.pstex}%
\end{picture}%
\setlength{\unitlength}{1381sp}%
\begingroup\makeatletter\ifx\SetFigFont\undefined%
\gdef\SetFigFont#1#2#3#4#5{%
  \reset@font\fontsize{#1}{#2pt}%
  \fontfamily{#3}\fontseries{#4}\fontshape{#5}%
  \selectfont}%
\fi\endgroup%
\begin{picture}(8009,6138)(1006,-11701)
\put(8401,-11236){\makebox(0,0)[lb]{\smash{{\SetFigFont{6}{7.2}{\rmdefault}{\mddefault}{\updefault}{\color[rgb]{0,0,0}$d$}%
}}}}
\put(2851,-6361){\makebox(0,0)[lb]{\smash{{\SetFigFont{6}{7.2}{\rmdefault}{\mddefault}{\updefault}{\color[rgb]{0,0,0}Density of the distance to}%
}}}}
\put(2851,-6811){\makebox(0,0)[lb]{\smash{{\SetFigFont{6}{7.2}{\rmdefault}{\mddefault}{\updefault}{\color[rgb]{0,0,0}a random variable uniformly }%
}}}}
\put(2851,-7261){\makebox(0,0)[lb]{\smash{{\SetFigFont{6}{7.2}{\rmdefault}{\mddefault}{\updefault}{\color[rgb]{0,0,0}distributed on a line segment $\overline{AB}$}%
}}}}
\put(2851,-7786){\makebox(0,0)[lb]{\smash{{\SetFigFont{6}{7.2}{\rmdefault}{\mddefault}{\updefault}{\color[rgb]{0,0,0}with $A=(1;1), B=(8;4)$}%
}}}}
\put(2851,-8311){\makebox(0,0)[lb]{\smash{{\SetFigFont{6}{7.2}{\rmdefault}{\mddefault}{\updefault}{\color[rgb]{0,0,0}$P=(3;3)$}%
}}}}
\put(1651,-6661){\rotatebox{90.0}{\makebox(0,0)[lb]{\smash{{\SetFigFont{6}{7.2}{\rmdefault}{\mddefault}{\updefault}{\color[rgb]{0,0,0}$f_D(d)$}%
}}}}}
\end{picture}%

%% file: Polyhedron_Source.pstex_t
\begin{picture}(0,0)%
\includegraphics{Polyhedron_Source.pstex}%
\end{picture}%
\setlength{\unitlength}{2486sp}%
\begingroup\makeatletter\ifx\SetFigFont\undefined%
\gdef\SetFigFont#1#2#3#4#5{%
  \reset@font\fontsize{#1}{#2pt}%
  \fontfamily{#3}\fontseries{#4}\fontshape{#5}%
  \selectfont}%
\fi\endgroup%
\begin{picture}(6162,4983)(5971,-11760)
\put(10306,-10771){\makebox(0,0)[lb]{\smash{{\SetFigFont{12}{14.4}{\rmdefault}{\mddefault}{\updefault}{\color[rgb]{0,0,0}$d$}%
}}}}
\put(9316,-9556){\makebox(0,0)[lb]{\smash{{\SetFigFont{12}{14.4}{\rmdefault}{\mddefault}{\updefault}{\color[rgb]{0,0,0}$S_4$}%
}}}}
\put(7741,-10051){\makebox(0,0)[lb]{\smash{{\SetFigFont{12}{14.4}{\rmdefault}{\mddefault}{\updefault}{\color[rgb]{0,0,0}$I_1$}%
}}}}
\put(8731,-10141){\makebox(0,0)[lb]{\smash{{\SetFigFont{12}{14.4}{\rmdefault}{\mddefault}{\updefault}{\color[rgb]{0,0,0}$S_2$}%
}}}}
\put(12106,-9916){\makebox(0,0)[lb]{\smash{{\SetFigFont{12}{14.4}{\rmdefault}{\mddefault}{\updefault}{\color[rgb]{0,0,0}$S_5$}%
}}}}
\put(11881,-8454){\makebox(0,0)[lb]{\smash{{\SetFigFont{12}{14.4}{\rmdefault}{\mddefault}{\updefault}{\color[rgb]{0,0,0}$S_{6}$}%
}}}}
\put(11116,-7981){\makebox(0,0)[lb]{\smash{{\SetFigFont{12}{14.4}{\rmdefault}{\mddefault}{\updefault}{\color[rgb]{0,0,0}$S_{7}$}%
}}}}
\put(10936,-7081){\makebox(0,0)[lb]{\smash{{\SetFigFont{12}{14.4}{\rmdefault}{\mddefault}{\updefault}{\color[rgb]{0,0,0}$S_{8}$}%
}}}}
\put(9316,-6924){\makebox(0,0)[lb]{\smash{{\SetFigFont{12}{14.4}{\rmdefault}{\mddefault}{\updefault}{\color[rgb]{0,0,0}$S_{9}$}%
}}}}
\put(9226,-7936){\makebox(0,0)[lb]{\smash{{\SetFigFont{12}{14.4}{\rmdefault}{\mddefault}{\updefault}{\color[rgb]{0,0,0}$S_{10}$}%
}}}}
\put(5986,-9556){\makebox(0,0)[lb]{\smash{{\SetFigFont{12}{14.4}{\rmdefault}{\mddefault}{\updefault}{\color[rgb]{0,0,0}$S_{1}$}%
}}}}
\put(10846,-9286){\makebox(0,0)[lb]{\smash{{\SetFigFont{12}{14.4}{\rmdefault}{\mddefault}{\updefault}{\color[rgb]{0,0,0}$I_4$}%
}}}}
\put(6976,-7554){\makebox(0,0)[lb]{\smash{{\SetFigFont{12}{14.4}{\rmdefault}{\mddefault}{\updefault}{\color[rgb]{0,0,0}$S_{11}$}%
}}}}
\put(8956,-8611){\makebox(0,0)[lb]{\smash{{\SetFigFont{12}{14.4}{\rmdefault}{\mddefault}{\updefault}{\color[rgb]{0,0,0}$I_3$}%
}}}}
\put(8191,-9061){\makebox(0,0)[lb]{\smash{{\SetFigFont{12}{14.4}{\rmdefault}{\mddefault}{\updefault}{\color[rgb]{0,0,0}$I_2$}%
}}}}
\put(8281,-8206){\makebox(0,0)[lb]{\smash{{\SetFigFont{12}{14.4}{\rmdefault}{\mddefault}{\updefault}{\color[rgb]{0,0,0}$S_3$}%
}}}}
\end{picture}%

%% file: CN.pstex_t
\begin{picture}(0,0)%
\includegraphics{CN.pstex}%
\end{picture}%
\setlength{\unitlength}{1657sp}%
\begingroup\makeatletter\ifx\SetFigFont\undefined%
\gdef\SetFigFont#1#2#3#4#5{%
  \reset@font\fontsize{#1}{#2pt}%
  \fontfamily{#3}\fontseries{#4}\fontshape{#5}%
  \selectfont}%
\fi\endgroup%
\begin{picture}(14043,7019)(-1814,-6158)
\put(-1799,-1996){\makebox(0,0)[lb]{\smash{{\SetFigFont{6}{7.2}{\rmdefault}{\mddefault}{\updefault}{\color[rgb]{0,0,0}{\tt{Crossing Number}}:}%
}}}}
\put(991,-1996){\makebox(0,0)[lb]{\smash{{\SetFigFont{6}{7.2}{\rmdefault}{\mddefault}{\updefault}{\color[rgb]{0,0,1}$+1$}%
}}}}
\put(6886,-1996){\makebox(0,0)[lb]{\smash{{\SetFigFont{6}{7.2}{\rmdefault}{\mddefault}{\updefault}{\color[rgb]{0,0,1}$+1$}%
}}}}
\put(-1799,-2446){\makebox(0,0)[lb]{\smash{{\SetFigFont{6}{7.2}{\rmdefault}{\mddefault}{\updefault}{\color[rgb]{0,0,0}{\tt{Crossing Number with}}}%
}}}}
\put(-1799,-2761){\makebox(0,0)[lb]{\smash{{\SetFigFont{6}{7.2}{\rmdefault}{\mddefault}{\updefault}{\color[rgb]{0,0,0}{\tt{Algorithm 2}}:}%
}}}}
\put(6886,-2716){\makebox(0,0)[lb]{\smash{{\SetFigFont{6}{7.2}{\rmdefault}{\mddefault}{\updefault}{\color[rgb]{1,0,0}$+1$}%
}}}}
\put(-1799,-4651){\makebox(0,0)[lb]{\smash{{\SetFigFont{6}{7.2}{\rmdefault}{\mddefault}{\updefault}{\color[rgb]{0,0,0}{\tt{Crossing Number}}:}%
}}}}
\put(1126,-4696){\makebox(0,0)[lb]{\smash{{\SetFigFont{6}{7.2}{\rmdefault}{\mddefault}{\updefault}{\color[rgb]{0,0,1}$+1$}%
}}}}
\put(4006,-4696){\makebox(0,0)[lb]{\smash{{\SetFigFont{6}{7.2}{\rmdefault}{\mddefault}{\updefault}{\color[rgb]{0,0,1}{\tt{Unchanged}}}%
}}}}
\put(7201,-4696){\makebox(0,0)[lb]{\smash{{\SetFigFont{6}{7.2}{\rmdefault}{\mddefault}{\updefault}{\color[rgb]{0,0,1}$+1$}%
}}}}
\put(9811,-4696){\makebox(0,0)[lb]{\smash{{\SetFigFont{6}{7.2}{\rmdefault}{\mddefault}{\updefault}{\color[rgb]{0,0,1}{\tt{Unchanged}}}%
}}}}
\put(-1799,-5101){\makebox(0,0)[lb]{\smash{{\SetFigFont{6}{7.2}{\rmdefault}{\mddefault}{\updefault}{\color[rgb]{0,0,0}{\tt{Crossing Number with}}}%
}}}}
\put(-1799,-5416){\makebox(0,0)[lb]{\smash{{\SetFigFont{6}{7.2}{\rmdefault}{\mddefault}{\updefault}{\color[rgb]{0,0,0}{\tt{Algorithm 2}}:}%
}}}}
\put(1081,-5416){\makebox(0,0)[lb]{\smash{{\SetFigFont{6}{7.2}{\rmdefault}{\mddefault}{\updefault}{\color[rgb]{1,0,0}$+1$}%
}}}}
\put(7201,-5416){\makebox(0,0)[lb]{\smash{{\SetFigFont{6}{7.2}{\rmdefault}{\mddefault}{\updefault}{\color[rgb]{1,0,0}$+1$}%
}}}}
\put(991,-2716){\makebox(0,0)[lb]{\smash{{\SetFigFont{6}{7.2}{\rmdefault}{\mddefault}{\updefault}{\color[rgb]{1,0,0}$+1$}%
}}}}
\put(3511,-2041){\makebox(0,0)[lb]{\smash{{\SetFigFont{6}{7.2}{\rmdefault}{\mddefault}{\updefault}{\color[rgb]{0,0,1}{\tt{Unchanged}}}%
}}}}
\put(9496,-1996){\makebox(0,0)[lb]{\smash{{\SetFigFont{6}{7.2}{\rmdefault}{\mddefault}{\updefault}{\color[rgb]{0,0,1}{\tt{Unchanged}}}%
}}}}
\put(9496,-2716){\makebox(0,0)[lb]{\smash{{\SetFigFont{6}{7.2}{\rmdefault}{\mddefault}{\updefault}{\color[rgb]{1,0,0}{\tt{Unchanged}}}%
}}}}
\put(3871,-2716){\makebox(0,0)[lb]{\smash{{\SetFigFont{6}{7.2}{\rmdefault}{\mddefault}{\updefault}{\color[rgb]{1,0,0}+2}%
}}}}
\put(4366,-5416){\makebox(0,0)[lb]{\smash{{\SetFigFont{6}{7.2}{\rmdefault}{\mddefault}{\updefault}{\color[rgb]{1,0,0}+2}%
}}}}
\put(9811,-5416){\makebox(0,0)[lb]{\smash{{\SetFigFont{6}{7.2}{\rmdefault}{\mddefault}{\updefault}{\color[rgb]{1,0,0}{\tt{Unchanged}}}%
}}}}
\end{picture}%

%% file: Inter0.pstex_t
\begin{picture}(0,0)%
\includegraphics{Inter0.pstex}%
\end{picture}%
\setlength{\unitlength}{2072sp}%
\begingroup\makeatletter\ifx\SetFigFont\undefined%
\gdef\SetFigFont#1#2#3#4#5{%
  \reset@font\fontsize{#1}{#2pt}%
  \fontfamily{#3}\fontseries{#4}\fontshape{#5}%
  \selectfont}%
\fi\endgroup%
\begin{picture}(7473,4767)(-4082,-5044)
\put(3241,-4381){\makebox(0,0)[lb]{\smash{{\SetFigFont{9}{10.8}{\rmdefault}{\mddefault}{\updefault}{\color[rgb]{0,0,0}$S_{i+1}$}%
}}}}
\put(3241,-4831){\makebox(0,0)[lb]{\smash{{\SetFigFont{9}{10.8}{\rmdefault}{\mddefault}{\updefault}{\color[rgb]{0,0,0}$S_{i+1}$}%
}}}}
\put(-2600,-1231){\makebox(0,0)[lb]{\smash{{\SetFigFont{6}{7.2}{\rmdefault}{\mddefault}{\updefault}{\color[rgb]{0,0,0}$A_2$}%
}}}}
\put(1171,-1276){\makebox(0,0)[lb]{\smash{{\SetFigFont{9}{10.8}{\rmdefault}{\mddefault}{\updefault}{\color[rgb]{0,0,0}$S_{i+1}$}%
}}}}
\put(-1304,-1231){\makebox(0,0)[lb]{\smash{{\SetFigFont{9}{10.8}{\rmdefault}{\mddefault}{\updefault}{\color[rgb]{0,0,0}$S_i$}%
}}}}
\put(-1439,-781){\makebox(0,0)[lb]{\smash{{\SetFigFont{9}{10.8}{\rmdefault}{\mddefault}{\updefault}{\color[rgb]{0,0,0}$S_i$}%
}}}}
\put(766,-804){\makebox(0,0)[lb]{\smash{{\SetFigFont{9}{10.8}{\rmdefault}{\mddefault}{\updefault}{\color[rgb]{0,0,0}$S_{i+1}$}%
}}}}
\put(-3419,-4381){\makebox(0,0)[lb]{\smash{{\SetFigFont{9}{10.8}{\rmdefault}{\mddefault}{\updefault}{\color[rgb]{0,0,0}$S_i$}%
}}}}
\put(-3419,-4831){\makebox(0,0)[lb]{\smash{{\SetFigFont{9}{10.8}{\rmdefault}{\mddefault}{\updefault}{\color[rgb]{0,0,0}$S_i$}%
}}}}
\put(3241,-3886){\makebox(0,0)[lb]{\smash{{\SetFigFont{9}{10.8}{\rmdefault}{\mddefault}{\updefault}{\color[rgb]{0,0,0}$S_{i+1}$}%
}}}}
\put(-2600,-781){\makebox(0,0)[lb]{\smash{{\SetFigFont{6}{7.2}{\rmdefault}{\mddefault}{\updefault}{\color[rgb]{0,0,0}$A_1$}%
}}}}
\put(-3959,-3931){\makebox(0,0)[lb]{\smash{{\SetFigFont{6}{7.2}{\rmdefault}{\mddefault}{\updefault}{\color[rgb]{0,0,0}$D_1$}%
}}}}
\put(-3959,-4381){\makebox(0,0)[lb]{\smash{{\SetFigFont{6}{7.2}{\rmdefault}{\mddefault}{\updefault}{\color[rgb]{0,0,0}$D_2$}%
}}}}
\put(-3419,-3931){\makebox(0,0)[lb]{\smash{{\SetFigFont{9}{10.8}{\rmdefault}{\mddefault}{\updefault}{\color[rgb]{0,0,0}$S_i$}%
}}}}
\put(-3959,-4831){\makebox(0,0)[lb]{\smash{{\SetFigFont{6}{7.2}{\rmdefault}{\mddefault}{\updefault}{\color[rgb]{0,0,0}$D_3$}%
}}}}
\put(-3959,-3256){\makebox(0,0)[lb]{\smash{{\SetFigFont{6}{7.2}{\rmdefault}{\mddefault}{\updefault}{\color[rgb]{0,0,0}$C$}%
}}}}
\put(-3464,-3256){\makebox(0,0)[lb]{\smash{{\SetFigFont{9}{10.8}{\rmdefault}{\mddefault}{\updefault}{\color[rgb]{0,0,0}$S_i$}%
}}}}
\put(991,-3256){\makebox(0,0)[lb]{\smash{{\SetFigFont{9}{10.8}{\rmdefault}{\mddefault}{\updefault}{\color[rgb]{0,0,0}$S_{i+1}$}%
}}}}
\put(-2834,-1816){\makebox(0,0)[lb]{\smash{{\SetFigFont{6}{7.2}{\rmdefault}{\mddefault}{\updefault}{\color[rgb]{0,0,0}$B_1$}%
}}}}
\put(-2834,-2311){\makebox(0,0)[lb]{\smash{{\SetFigFont{6}{7.2}{\rmdefault}{\mddefault}{\updefault}{\color[rgb]{0,0,0}$B_2$}%
}}}}
\put(-2834,-2716){\makebox(0,0)[lb]{\smash{{\SetFigFont{6}{7.2}{\rmdefault}{\mddefault}{\updefault}{\color[rgb]{0,0,0}$B_3$}%
}}}}
\put(-2249,-1816){\makebox(0,0)[lb]{\smash{{\SetFigFont{9}{10.8}{\rmdefault}{\mddefault}{\updefault}{\color[rgb]{0,0,0}$S_i$}%
}}}}
\put(-1709,-2266){\makebox(0,0)[lb]{\smash{{\SetFigFont{9}{10.8}{\rmdefault}{\mddefault}{\updefault}{\color[rgb]{0,0,0}$S_i$}%
}}}}
\put(3376,-2761){\makebox(0,0)[lb]{\smash{{\SetFigFont{9}{10.8}{\rmdefault}{\mddefault}{\updefault}{\color[rgb]{0,0,0}$S_{i+1}$}%
}}}}
\put(3106,-1771){\makebox(0,0)[lb]{\smash{{\SetFigFont{9}{10.8}{\rmdefault}{\mddefault}{\updefault}{\color[rgb]{0,0,0}$S_{i+1}$}%
}}}}
\put(3196,-2311){\makebox(0,0)[lb]{\smash{{\SetFigFont{9}{10.8}{\rmdefault}{\mddefault}{\updefault}{\color[rgb]{0,0,0}$S_{i+1}$}%
}}}}
\put(1756,-2761){\makebox(0,0)[lb]{\smash{{\SetFigFont{9}{10.8}{\rmdefault}{\mddefault}{\updefault}{\color[rgb]{0,0,0}$S_i$}%
}}}}
\end{picture}%

%% file: Inter1.pstex_t
\begin{picture}(0,0)%
\includegraphics{Inter1.pstex}%
\end{picture}%
\setlength{\unitlength}{2072sp}%
\begingroup\makeatletter\ifx\SetFigFont\undefined%
\gdef\SetFigFont#1#2#3#4#5{%
  \reset@font\fontsize{#1}{#2pt}%
  \fontfamily{#3}\fontseries{#4}\fontshape{#5}%
  \selectfont}%
\fi\endgroup%
\begin{picture}(9629,10557)(-516,-10878)
\put(-224,-3031){\makebox(0,0)[lb]{\smash{{\SetFigFont{9}{10.8}{\rmdefault}{\mddefault}{\updefault}{\color[rgb]{0,0,0}$S_{i+2}$}%
}}}}
\put(1171,-3211){\makebox(0,0)[lb]{\smash{{\SetFigFont{9}{10.8}{\rmdefault}{\mddefault}{\updefault}{\color[rgb]{0,0,0}$S_i$}%
}}}}
\put(1576,-1501){\makebox(0,0)[lb]{\smash{{\SetFigFont{9}{10.8}{\rmdefault}{\mddefault}{\updefault}{\color[rgb]{0,0,0}$S_{i+1}$}%
}}}}
\put(2971,-2221){\makebox(0,0)[lb]{\smash{{\SetFigFont{9}{10.8}{\rmdefault}{\mddefault}{\updefault}{\color[rgb]{0,0,0}$\mathcal{D}(P, d)$}%
}}}}
\put(8236,-961){\makebox(0,0)[lb]{\smash{{\SetFigFont{9}{10.8}{\rmdefault}{\mddefault}{\updefault}{\color[rgb]{0,0,0}$S_{i+2}$}%
}}}}
\put(5896,-3211){\makebox(0,0)[lb]{\smash{{\SetFigFont{9}{10.8}{\rmdefault}{\mddefault}{\updefault}{\color[rgb]{0,0,0}$S_i$}%
}}}}
\put(4906,-7891){\makebox(0,0)[lb]{\smash{{\SetFigFont{9}{10.8}{\rmdefault}{\mddefault}{\updefault}{\color[rgb]{0,0,0}$S_{i}$}%
}}}}
\put(1126,-6946){\makebox(0,0)[lb]{\smash{{\SetFigFont{9}{10.8}{\rmdefault}{\mddefault}{\updefault}{\color[rgb]{0,0,0}$S_i$}%
}}}}
\put( 91,-7891){\makebox(0,0)[lb]{\smash{{\SetFigFont{9}{10.8}{\rmdefault}{\mddefault}{\updefault}{\color[rgb]{0,0,0}$S_{i}$}%
}}}}
\put(1576,-8296){\makebox(0,0)[lb]{\smash{{\SetFigFont{9}{10.8}{\rmdefault}{\mddefault}{\updefault}{\color[rgb]{0,0,0}$S_{i+1}$}%
}}}}
\put(4141,-8161){\makebox(0,0)[lb]{\smash{{\SetFigFont{9}{10.8}{\rmdefault}{\mddefault}{\updefault}{\color[rgb]{0,0,0}$S_{i+2}$}%
}}}}
\put(6841,-6496){\makebox(0,0)[lb]{\smash{{\SetFigFont{9}{10.8}{\rmdefault}{\mddefault}{\updefault}{\color[rgb]{0,0,0}$P$}%
}}}}
\put(5266,-4471){\makebox(0,0)[lb]{\smash{{\SetFigFont{9}{10.8}{\rmdefault}{\mddefault}{\updefault}{\color[rgb]{0,0,0}$S_i$}%
}}}}
\put(7921,-4741){\makebox(0,0)[lb]{\smash{{\SetFigFont{9}{10.8}{\rmdefault}{\mddefault}{\updefault}{\color[rgb]{0,0,0}$S_{i+2}$}%
}}}}
\put(6481,-5124){\makebox(0,0)[lb]{\smash{{\SetFigFont{9}{10.8}{\rmdefault}{\mddefault}{\updefault}{\color[rgb]{0,0,0}$S_{i+1}$}%
}}}}
\put(6481,-8291){\makebox(0,0)[lb]{\smash{{\SetFigFont{9}{10.8}{\rmdefault}{\mddefault}{\updefault}{\color[rgb]{0,0,0}$S_{i+1}$}%
}}}}
\put(8056,-9781){\makebox(0,0)[lb]{\smash{{\SetFigFont{9}{10.8}{\rmdefault}{\mddefault}{\updefault}{\color[rgb]{0,0,0}$\mathcal{D}(P, d)$}%
}}}}
\put(8011,-6721){\makebox(0,0)[lb]{\smash{{\SetFigFont{9}{10.8}{\rmdefault}{\mddefault}{\updefault}{\color[rgb]{0,0,0}$\mathcal{D}(P, d)$}%
}}}}
\put(3016,-6451){\makebox(0,0)[lb]{\smash{{\SetFigFont{9}{10.8}{\rmdefault}{\mddefault}{\updefault}{\color[rgb]{0,0,0}$\mathcal{D}(P, d)$}%
}}}}
\put(7741,-3166){\makebox(0,0)[lb]{\smash{{\SetFigFont{9}{10.8}{\rmdefault}{\mddefault}{\updefault}{\color[rgb]{0,0,0}$\mathcal{D}(P, d)$}%
}}}}
\put(2206,-6856){\makebox(0,0)[lb]{\smash{{\SetFigFont{9}{10.8}{\rmdefault}{\mddefault}{\updefault}{\color[rgb]{0,0,0}$S_{i+2}$}%
}}}}
\put(3061,-10006){\makebox(0,0)[lb]{\smash{{\SetFigFont{9}{10.8}{\rmdefault}{\mddefault}{\updefault}{\color[rgb]{0,0,0}$\mathcal{D}(P, d)$}%
}}}}
\put(2341,-9961){\makebox(0,0)[lb]{\smash{{\SetFigFont{9}{10.8}{\rmdefault}{\mddefault}{\updefault}{\color[rgb]{0,0,0}$S_{i+2}$}%
}}}}
\put(1891,-2761){\makebox(0,0)[lb]{\smash{{\SetFigFont{9}{10.8}{\rmdefault}{\mddefault}{\updefault}{\color[rgb]{0,0,0}$P$}%
}}}}
\put(6616,-2716){\makebox(0,0)[lb]{\smash{{\SetFigFont{9}{10.8}{\rmdefault}{\mddefault}{\updefault}{\color[rgb]{0,0,0}$P$}%
}}}}
\put(6211,-1411){\makebox(0,0)[lb]{\smash{{\SetFigFont{9}{10.8}{\rmdefault}{\mddefault}{\updefault}{\color[rgb]{0,0,0}$S_{i+1}$}%
}}}}
\put(6346,-826){\makebox(0,0)[lb]{\smash{{\SetFigFont{8}{9.6}{\rmdefault}{\mddefault}{\updefault}{\color[rgb]{0,0,0}$\mathcal{R}_2$}%
}}}}
\put(1846,-6451){\makebox(0,0)[lb]{\smash{{\SetFigFont{9}{10.8}{\rmdefault}{\mddefault}{\updefault}{\color[rgb]{0,0,0}$P$}%
}}}}
\put(1441,-5146){\makebox(0,0)[lb]{\smash{{\SetFigFont{9}{10.8}{\rmdefault}{\mddefault}{\updefault}{\color[rgb]{0,0,0}$S_{i+1}$}%
}}}}
\put(1936,-9691){\makebox(0,0)[lb]{\smash{{\SetFigFont{9}{10.8}{\rmdefault}{\mddefault}{\updefault}{\color[rgb]{0,0,0}$P$}%
}}}}
\put(6841,-9713){\makebox(0,0)[lb]{\smash{{\SetFigFont{9}{10.8}{\rmdefault}{\mddefault}{\updefault}{\color[rgb]{0,0,0}$P$}%
}}}}
\put(5109,-8291){\makebox(0,0)[lb]{\smash{{\SetFigFont{8}{9.6}{\rmdefault}{\mddefault}{\updefault}{\color[rgb]{0,0,0}$\mathcal{R}_6$}%
}}}}
\put(496,-1906){\makebox(0,0)[lb]{\smash{{\SetFigFont{8}{9.6}{\rmdefault}{\mddefault}{\updefault}{\color[rgb]{0,0,0}$\mathcal{R}_1$}%
}}}}
\put(6571,-4606){\makebox(0,0)[lb]{\smash{{\SetFigFont{8}{9.6}{\rmdefault}{\mddefault}{\updefault}{\color[rgb]{0,0,0}$\mathcal{R}_4$}%
}}}}
\put(1216,-9286){\makebox(0,0)[lb]{\smash{{\SetFigFont{8}{9.6}{\rmdefault}{\mddefault}{\updefault}{\color[rgb]{0,0,0}$\mathcal{R}_5$}%
}}}}
\put(2611,-5641){\makebox(0,0)[lb]{\smash{{\SetFigFont{8}{9.6}{\rmdefault}{\mddefault}{\updefault}{\color[rgb]{0,0,0}$\mathcal{R}_3$}%
}}}}
\end{picture}%

%% file: Poly_Inter_Empty.pstex_t
\begin{picture}(0,0)%
\includegraphics{Poly_Inter_Empty.pstex}%
\end{picture}%
\setlength{\unitlength}{1533sp}%
\begingroup\makeatletter\ifx\SetFigFont\undefined%
\gdef\SetFigFont#1#2#3#4#5{%
  \reset@font\fontsize{#1}{#2pt}%
  \fontfamily{#3}\fontseries{#4}\fontshape{#5}%
  \selectfont}%
\fi\endgroup%
\begin{picture}(12243,8153)(5971,-14930)
\put(14986,-10006){\makebox(0,0)[lb]{\smash{{\SetFigFont{7}{8.4}{\rmdefault}{\mddefault}{\updefault}{\color[rgb]{0,0,0}$S_2$}%
}}}}
\put(16786,-9286){\makebox(0,0)[lb]{\smash{{\SetFigFont{7}{8.4}{\rmdefault}{\mddefault}{\updefault}{\color[rgb]{0,0,0}$S_3$}%
}}}}
\put(13906,-10861){\makebox(0,0)[lb]{\smash{{\SetFigFont{6}{7.2}{\rmdefault}{\mddefault}{\updefault}{\color[rgb]{0,0,0}Second case: $\mathcal{S} \subset \mathcal{D}(P,d)$}%
}}}}
\put(12421,-13876){\makebox(0,0)[lb]{\smash{{\SetFigFont{7}{8.4}{\rmdefault}{\mddefault}{\updefault}{\color[rgb]{0,0,0}$S_4$}%
}}}}
\put(11836,-14461){\makebox(0,0)[lb]{\smash{{\SetFigFont{7}{8.4}{\rmdefault}{\mddefault}{\updefault}{\color[rgb]{0,0,0}$S_2$}%
}}}}
\put(15211,-14236){\makebox(0,0)[lb]{\smash{{\SetFigFont{7}{8.4}{\rmdefault}{\mddefault}{\updefault}{\color[rgb]{0,0,0}$S_5$}%
}}}}
\put(14986,-12774){\makebox(0,0)[lb]{\smash{{\SetFigFont{7}{8.4}{\rmdefault}{\mddefault}{\updefault}{\color[rgb]{0,0,0}$S_{6}$}%
}}}}
\put(14221,-12301){\makebox(0,0)[lb]{\smash{{\SetFigFont{7}{8.4}{\rmdefault}{\mddefault}{\updefault}{\color[rgb]{0,0,0}$S_{7}$}%
}}}}
\put(14041,-11401){\makebox(0,0)[lb]{\smash{{\SetFigFont{7}{8.4}{\rmdefault}{\mddefault}{\updefault}{\color[rgb]{0,0,0}$S_{8}$}%
}}}}
\put(12421,-11244){\makebox(0,0)[lb]{\smash{{\SetFigFont{7}{8.4}{\rmdefault}{\mddefault}{\updefault}{\color[rgb]{0,0,0}$S_{9}$}%
}}}}
\put(12331,-12256){\makebox(0,0)[lb]{\smash{{\SetFigFont{7}{8.4}{\rmdefault}{\mddefault}{\updefault}{\color[rgb]{0,0,0}$S_{10}$}%
}}}}
\put(9091,-13876){\makebox(0,0)[lb]{\smash{{\SetFigFont{7}{8.4}{\rmdefault}{\mddefault}{\updefault}{\color[rgb]{0,0,0}$S_{1}$}%
}}}}
\put(10081,-11874){\makebox(0,0)[lb]{\smash{{\SetFigFont{7}{8.4}{\rmdefault}{\mddefault}{\updefault}{\color[rgb]{0,0,0}$S_{11}$}%
}}}}
\put(6886,-10861){\makebox(0,0)[lb]{\smash{{\SetFigFont{6}{7.2}{\rmdefault}{\mddefault}{\updefault}{\color[rgb]{0,0,0}First case: $\mathcal{D}(P,d) \cap \mathcal{S} = \emptyset$}%
}}}}
\put(10261,-14866){\makebox(0,0)[lb]{\smash{{\SetFigFont{6}{7.2}{\rmdefault}{\mddefault}{\updefault}{\color[rgb]{0,0,0}Third case: $\mathcal{D}(P,d) \subset \mathcal{S}$}%
}}}}
\put(9316,-9556){\makebox(0,0)[lb]{\smash{{\SetFigFont{7}{8.4}{\rmdefault}{\mddefault}{\updefault}{\color[rgb]{0,0,0}$S_4$}%
}}}}
\put(8731,-10141){\makebox(0,0)[lb]{\smash{{\SetFigFont{7}{8.4}{\rmdefault}{\mddefault}{\updefault}{\color[rgb]{0,0,0}$S_2$}%
}}}}
\put(12106,-9916){\makebox(0,0)[lb]{\smash{{\SetFigFont{7}{8.4}{\rmdefault}{\mddefault}{\updefault}{\color[rgb]{0,0,0}$S_5$}%
}}}}
\put(11881,-8454){\makebox(0,0)[lb]{\smash{{\SetFigFont{7}{8.4}{\rmdefault}{\mddefault}{\updefault}{\color[rgb]{0,0,0}$S_{6}$}%
}}}}
\put(11116,-7981){\makebox(0,0)[lb]{\smash{{\SetFigFont{7}{8.4}{\rmdefault}{\mddefault}{\updefault}{\color[rgb]{0,0,0}$S_{7}$}%
}}}}
\put(10936,-7081){\makebox(0,0)[lb]{\smash{{\SetFigFont{7}{8.4}{\rmdefault}{\mddefault}{\updefault}{\color[rgb]{0,0,0}$S_{8}$}%
}}}}
\put(9316,-6924){\makebox(0,0)[lb]{\smash{{\SetFigFont{7}{8.4}{\rmdefault}{\mddefault}{\updefault}{\color[rgb]{0,0,0}$S_{9}$}%
}}}}
\put(9226,-7936){\makebox(0,0)[lb]{\smash{{\SetFigFont{7}{8.4}{\rmdefault}{\mddefault}{\updefault}{\color[rgb]{0,0,0}$S_{10}$}%
}}}}
\put(5986,-9556){\makebox(0,0)[lb]{\smash{{\SetFigFont{7}{8.4}{\rmdefault}{\mddefault}{\updefault}{\color[rgb]{0,0,0}$S_{1}$}%
}}}}
\put(6976,-7554){\makebox(0,0)[lb]{\smash{{\SetFigFont{7}{8.4}{\rmdefault}{\mddefault}{\updefault}{\color[rgb]{0,0,0}$S_{11}$}%
}}}}
\put(8506,-7216){\makebox(0,0)[lb]{\smash{{\SetFigFont{6}{7.2}{\rmdefault}{\mddefault}{\updefault}{\color[rgb]{0,0,0}$d$}%
}}}}
\put(15121,-13291){\makebox(0,0)[lb]{\smash{{\SetFigFont{6}{7.2}{\rmdefault}{\mddefault}{\updefault}{\color[rgb]{0,0,0}$P_{1}$}%
}}}}
\put(11026,-7576){\makebox(0,0)[lb]{\smash{{\SetFigFont{6}{7.2}{\rmdefault}{\mddefault}{\updefault}{\color[rgb]{0,0,0}$P_{2}$}%
}}}}
\put(8326,-7621){\makebox(0,0)[lb]{\smash{{\SetFigFont{5}{6.0}{\rmdefault}{\mddefault}{\updefault}{\color[rgb]{0,0,0}$P$}%
}}}}
\put(15166,-7306){\makebox(0,0)[lb]{\smash{{\SetFigFont{7}{8.4}{\rmdefault}{\mddefault}{\updefault}{\color[rgb]{0,0,0}$S_5$}%
}}}}
\put(14491,-9151){\makebox(0,0)[lb]{\smash{{\SetFigFont{7}{8.4}{\rmdefault}{\mddefault}{\updefault}{\color[rgb]{0,0,0}$S_1$}%
}}}}
\put(15436,-8701){\makebox(0,0)[lb]{\smash{{\SetFigFont{6}{7.2}{\rmdefault}{\mddefault}{\updefault}{\color[rgb]{0,0,0}$P$}%
}}}}
\put(15571,-7936){\makebox(0,0)[lb]{\smash{{\SetFigFont{6}{7.2}{\rmdefault}{\mddefault}{\updefault}{\color[rgb]{0,0,0}$d$}%
}}}}
\put(16786,-7891){\makebox(0,0)[lb]{\smash{{\SetFigFont{7}{8.4}{\rmdefault}{\mddefault}{\updefault}{\color[rgb]{0,0,0}$S_4$}%
}}}}
\put(14536,-8431){\makebox(0,0)[lb]{\smash{{\SetFigFont{7}{8.4}{\rmdefault}{\mddefault}{\updefault}{\color[rgb]{0,0,0}$S_6$}%
}}}}
\put(13321,-12886){\makebox(0,0)[lb]{\smash{{\SetFigFont{5}{6.0}{\rmdefault}{\mddefault}{\updefault}{\color[rgb]{0,0,0}$d$}%
}}}}
\put(11386,-12526){\makebox(0,0)[lb]{\smash{{\SetFigFont{7}{8.4}{\rmdefault}{\mddefault}{\updefault}{\color[rgb]{0,0,0}$S_3$}%
}}}}
\put(8281,-8206){\makebox(0,0)[lb]{\smash{{\SetFigFont{7}{8.4}{\rmdefault}{\mddefault}{\updefault}{\color[rgb]{0,0,0}$S_3$}%
}}}}
\put(9406,-7576){\makebox(0,0)[lb]{\smash{{\SetFigFont{6}{7.2}{\rmdefault}{\mddefault}{\updefault}{\color[rgb]{0,0,0}$P_{1}$}%
}}}}
\put(13411,-13336){\makebox(0,0)[lb]{\smash{{\SetFigFont{5}{6.0}{\rmdefault}{\mddefault}{\updefault}{\color[rgb]{0,0,0}$P$}%
}}}}
\end{picture}%

%% file: Triangle.pstex_t
\begin{picture}(0,0)%
\includegraphics{Triangle.pstex}%
\end{picture}%
\setlength{\unitlength}{1450sp}%
\begingroup\makeatletter\ifx\SetFigFont\undefined%
\gdef\SetFigFont#1#2#3#4#5{%
  \reset@font\fontsize{#1}{#2pt}%
  \fontfamily{#3}\fontseries{#4}\fontshape{#5}%
  \selectfont}%
\fi\endgroup%
\begin{picture}(7783,5058)(362,-4615)
\put(1869,-2095){\makebox(0,0)[lb]{\smash{{\SetFigFont{6}{7.2}{\rmdefault}{\mddefault}{\updefault}{\color[rgb]{0,0,0}$S_3$}%
}}}}
\put(946,-3481){\makebox(0,0)[lb]{\smash{{\SetFigFont{6}{7.2}{\rmdefault}{\mddefault}{\updefault}{\color[rgb]{0,0,0}$S_1$}%
}}}}
\put(3241,-4426){\makebox(0,0)[lb]{\smash{{\SetFigFont{6}{7.2}{\rmdefault}{\mddefault}{\updefault}{\color[rgb]{0,0,0}$P$}%
}}}}
\put(518,-4381){\makebox(0,0)[lb]{\smash{{\SetFigFont{6}{7.2}{\rmdefault}{\mddefault}{\updefault}{\color[rgb]{0,0,0}$O$}%
}}}}
\put(5446,-3468){\makebox(0,0)[lb]{\smash{{\SetFigFont{6}{7.2}{\rmdefault}{\mddefault}{\updefault}{\color[rgb]{0,0,0}$S_2$}%
}}}}
\end{picture}%

%% file: Triangle_5_0B.pstex_t
\begin{picture}(0,0)%
\includegraphics{Triangle_5_0B.pstex}%
\end{picture}%
\setlength{\unitlength}{1658sp}%
\begingroup\makeatletter\ifx\SetFigFont\undefined%
\gdef\SetFigFont#1#2#3#4#5{%
  \reset@font\fontsize{#1}{#2pt}%
  \fontfamily{#3}\fontseries{#4}\fontshape{#5}%
  \selectfont}%
\fi\endgroup%
\begin{picture}(5878,4584)(2217,-8033)
\put(3121,-4336){\makebox(0,0)[lb]{\smash{{\SetFigFont{7}{8.4}{\rmdefault}{\mddefault}{\updefault}{\color[rgb]{0,0,0}distributed in the triangle}%
}}}}
\put(3121,-4711){\makebox(0,0)[lb]{\smash{{\SetFigFont{7}{8.4}{\rmdefault}{\mddefault}{\updefault}{\color[rgb]{0,0,0}of the left}%
}}}}
\put(4951,-4711){\makebox(0,0)[lb]{\smash{{\SetFigFont{7}{8.4}{\rmdefault}{\mddefault}{\updefault}{\color[rgb]{0,0,0}figure, $P=(5;0)$}%
}}}}
\put(3121,-3961){\makebox(0,0)[lb]{\smash{{\SetFigFont{7}{8.4}{\rmdefault}{\mddefault}{\updefault}{\color[rgb]{0,0,0}Density of $D$ when $X$ is uniformly}%
}}}}
\put(2851,-4486){\rotatebox{90.0}{\makebox(0,0)[lb]{\smash{{\SetFigFont{7}{8.4}{\rmdefault}{\mddefault}{\updefault}{\color[rgb]{0,0,0}$f_D(d)$}%
}}}}}
\end{picture}%

%% file: TriangleIn.pstex_t
\begin{picture}(0,0)%
\includegraphics{TriangleIn.pstex}%
\end{picture}%
\setlength{\unitlength}{1450sp}%
\begingroup\makeatletter\ifx\SetFigFont\undefined%
\gdef\SetFigFont#1#2#3#4#5{%
  \reset@font\fontsize{#1}{#2pt}%
  \fontfamily{#3}\fontseries{#4}\fontshape{#5}%
  \selectfont}%
\fi\endgroup%
\begin{picture}(7783,5058)(362,-4615)
\put(1869,-2095){\makebox(0,0)[lb]{\smash{{\SetFigFont{6}{7.2}{\rmdefault}{\mddefault}{\updefault}{\color[rgb]{0,0,0}$S_3$}%
}}}}
\put(946,-3481){\makebox(0,0)[lb]{\smash{{\SetFigFont{6}{7.2}{\rmdefault}{\mddefault}{\updefault}{\color[rgb]{0,0,0}$S_1$}%
}}}}
\put(518,-4381){\makebox(0,0)[lb]{\smash{{\SetFigFont{6}{7.2}{\rmdefault}{\mddefault}{\updefault}{\color[rgb]{0,0,0}$O$}%
}}}}
\put(5446,-3468){\makebox(0,0)[lb]{\smash{{\SetFigFont{6}{7.2}{\rmdefault}{\mddefault}{\updefault}{\color[rgb]{0,0,0}$S_2$}%
}}}}
\put(2836,-3481){\makebox(0,0)[lb]{\smash{{\SetFigFont{6}{7.2}{\rmdefault}{\mddefault}{\updefault}{\color[rgb]{0,0,0}$P$}%
}}}}
\end{picture}%

%% file: Triangle_4_2B.pstex_t
\begin{picture}(0,0)%
\includegraphics{Triangle_4_2B.pstex}%
\end{picture}%
\setlength{\unitlength}{1658sp}%
\begingroup\makeatletter\ifx\SetFigFont\undefined%
\gdef\SetFigFont#1#2#3#4#5{%
  \reset@font\fontsize{#1}{#2pt}%
  \fontfamily{#3}\fontseries{#4}\fontshape{#5}%
  \selectfont}%
\fi\endgroup%
\begin{picture}(5887,4584)(2123,-8033)
\put(2851,-4486){\rotatebox{90.0}{\makebox(0,0)[lb]{\smash{{\SetFigFont{7}{8.4}{\rmdefault}{\mddefault}{\updefault}{\color[rgb]{0,0,0}$f_D(d)$}%
}}}}}
\put(3826,-3811){\makebox(0,0)[lb]{\smash{{\SetFigFont{7}{8.4}{\rmdefault}{\mddefault}{\updefault}{\color[rgb]{0,0,0}Density of $D$ when $X$ is}%
}}}}
\put(3865,-4186){\makebox(0,0)[lb]{\smash{{\SetFigFont{7}{8.4}{\rmdefault}{\mddefault}{\updefault}{\color[rgb]{0,0,0}uniformly distributed in the}%
}}}}
\put(4201,-4561){\makebox(0,0)[lb]{\smash{{\SetFigFont{7}{8.4}{\rmdefault}{\mddefault}{\updefault}{\color[rgb]{0,0,0}triangle of the left figure}%
}}}}
\put(4201,-4936){\makebox(0,0)[lb]{\smash{{\SetFigFont{7}{8.4}{\rmdefault}{\mddefault}{\updefault}{\color[rgb]{0,0,0}$P=(4;2)$}%
}}}}
\end{picture}%

%% file: Rectangle.pstex_t
\begin{picture}(0,0)%
\includegraphics{Rectangle.pstex}%
\end{picture}%
\setlength{\unitlength}{1450sp}%
\begingroup\makeatletter\ifx\SetFigFont\undefined%
\gdef\SetFigFont#1#2#3#4#5{%
  \reset@font\fontsize{#1}{#2pt}%
  \fontfamily{#3}\fontseries{#4}\fontshape{#5}%
  \selectfont}%
\fi\endgroup%
\begin{picture}(7758,5058)(397,-4615)
\put(1846,-3076){\makebox(0,0)[lb]{\smash{{\SetFigFont{6}{7.2}{\rmdefault}{\mddefault}{\updefault}{\color[rgb]{0,0,0}$S_1$}%
}}}}
\put(6346,-3076){\makebox(0,0)[lb]{\smash{{\SetFigFont{6}{7.2}{\rmdefault}{\mddefault}{\updefault}{\color[rgb]{0,0,0}$S_2$}%
}}}}
\put(6346,-781){\makebox(0,0)[lb]{\smash{{\SetFigFont{6}{7.2}{\rmdefault}{\mddefault}{\updefault}{\color[rgb]{0,0,0}$S_3$}%
}}}}
\put(1846,-781){\makebox(0,0)[lb]{\smash{{\SetFigFont{6}{7.2}{\rmdefault}{\mddefault}{\updefault}{\color[rgb]{0,0,0}$S_4$}%
}}}}
\put(1418,-3976){\makebox(0,0)[lb]{\smash{{\SetFigFont{6}{7.2}{\rmdefault}{\mddefault}{\updefault}{\color[rgb]{0,0,0}$P$}%
}}}}
\put(518,-4381){\makebox(0,0)[lb]{\smash{{\SetFigFont{6}{7.2}{\rmdefault}{\mddefault}{\updefault}{\color[rgb]{0,0,0}$O$}%
}}}}
\end{picture}%

%% file: Rectangle_1_1B.pstex_t
\begin{picture}(0,0)%
\includegraphics{Rectangle_1_1B.pstex}%
\end{picture}%
\setlength{\unitlength}{1658sp}%
\begingroup\makeatletter\ifx\SetFigFont\undefined%
\gdef\SetFigFont#1#2#3#4#5{%
  \reset@font\fontsize{#1}{#2pt}%
  \fontfamily{#3}\fontseries{#4}\fontshape{#5}%
  \selectfont}%
\fi\endgroup%
\begin{picture}(5950,4584)(2123,-8033)
\put(2626,-3886){\makebox(0,0)[lb]{\smash{{\SetFigFont{7}{8.4}{\rmdefault}{\mddefault}{\updefault}{\color[rgb]{0,0,0}$f_D(d)$}%
}}}}
\put(3451,-6661){\makebox(0,0)[lb]{\smash{{\SetFigFont{7}{8.4}{\rmdefault}{\mddefault}{\updefault}{\color[rgb]{0,0,0}Density of $D$ when $X$ is}%
}}}}
\put(6451,-3886){\makebox(0,0)[lb]{\smash{{\SetFigFont{7}{8.4}{\rmdefault}{\mddefault}{\updefault}{\color[rgb]{0,0,0}$P=(1;1)$}%
}}}}
\put(3451,-7111){\makebox(0,0)[lb]{\smash{{\SetFigFont{7}{8.4}{\rmdefault}{\mddefault}{\updefault}{\color[rgb]{0,0,0}uniformly distributed}%
}}}}
\put(3451,-7561){\makebox(0,0)[lb]{\smash{{\SetFigFont{7}{8.4}{\rmdefault}{\mddefault}{\updefault}{\color[rgb]{0,0,0}rectangle of the left figure}%
}}}}
\put(7105,-7111){\makebox(0,0)[lb]{\smash{{\SetFigFont{7}{8.4}{\rmdefault}{\mddefault}{\updefault}{\color[rgb]{0,0,0}in the }%
}}}}
\end{picture}%

%% file: RectangleIn.pstex_t
\begin{picture}(0,0)%
\includegraphics{RectangleIn.pstex}%
\end{picture}%
\setlength{\unitlength}{1450sp}%
\begingroup\makeatletter\ifx\SetFigFont\undefined%
\gdef\SetFigFont#1#2#3#4#5{%
  \reset@font\fontsize{#1}{#2pt}%
  \fontfamily{#3}\fontseries{#4}\fontshape{#5}%
  \selectfont}%
\fi\endgroup%
\begin{picture}(7758,5058)(397,-4615)
\put(1846,-3076){\makebox(0,0)[lb]{\smash{{\SetFigFont{6}{7.2}{\rmdefault}{\mddefault}{\updefault}{\color[rgb]{0,0,0}$S_1$}%
}}}}
\put(6346,-3076){\makebox(0,0)[lb]{\smash{{\SetFigFont{6}{7.2}{\rmdefault}{\mddefault}{\updefault}{\color[rgb]{0,0,0}$S_2$}%
}}}}
\put(6346,-781){\makebox(0,0)[lb]{\smash{{\SetFigFont{6}{7.2}{\rmdefault}{\mddefault}{\updefault}{\color[rgb]{0,0,0}$S_3$}%
}}}}
\put(1846,-781){\makebox(0,0)[lb]{\smash{{\SetFigFont{6}{7.2}{\rmdefault}{\mddefault}{\updefault}{\color[rgb]{0,0,0}$S_4$}%
}}}}
\put(518,-4381){\makebox(0,0)[lb]{\smash{{\SetFigFont{6}{7.2}{\rmdefault}{\mddefault}{\updefault}{\color[rgb]{0,0,0}$O$}%
}}}}
\put(3736,-2131){\makebox(0,0)[lb]{\smash{{\SetFigFont{6}{7.2}{\rmdefault}{\mddefault}{\updefault}{\color[rgb]{0,0,0}$P$}%
}}}}
\end{picture}%

%% file: Rectangle_6_5B.pstex_t
\begin{picture}(0,0)%
\includegraphics{Rectangle_6_5B.pstex}%
\end{picture}%
\setlength{\unitlength}{1658sp}%
\begingroup\makeatletter\ifx\SetFigFont\undefined%
\gdef\SetFigFont#1#2#3#4#5{%
  \reset@font\fontsize{#1}{#2pt}%
  \fontfamily{#3}\fontseries{#4}\fontshape{#5}%
  \selectfont}%
\fi\endgroup%
\begin{picture}(5887,4584)(2123,-8033)
\put(2551,-3886){\makebox(0,0)[lb]{\smash{{\SetFigFont{7}{8.4}{\rmdefault}{\mddefault}{\updefault}{\color[rgb]{0,0,0}$f_D(d)$}%
}}}}
\put(6301,-3886){\makebox(0,0)[lb]{\smash{{\SetFigFont{7}{8.4}{\rmdefault}{\mddefault}{\updefault}{\color[rgb]{0,0,0}$P=(6;5)$}%
}}}}
\put(3001,-6886){\makebox(0,0)[lb]{\smash{{\SetFigFont{7}{8.4}{\rmdefault}{\mddefault}{\updefault}{\color[rgb]{0,0,0}Density of $D$ when $X$ is}%
}}}}
\put(3001,-7261){\makebox(0,0)[lb]{\smash{{\SetFigFont{7}{8.4}{\rmdefault}{\mddefault}{\updefault}{\color[rgb]{0,0,0}uniformly distributed in the}%
}}}}
\put(3001,-7636){\makebox(0,0)[lb]{\smash{{\SetFigFont{7}{8.4}{\rmdefault}{\mddefault}{\updefault}{\color[rgb]{0,0,0}rectangle of the left figure}%
}}}}
\end{picture}%

%% file: Poly_Ext.pstex_t
\begin{picture}(0,0)%
\includegraphics{Poly_Ext.pstex}%
\end{picture}%
\setlength{\unitlength}{1450sp}%
\begingroup\makeatletter\ifx\SetFigFont\undefined%
\gdef\SetFigFont#1#2#3#4#5{%
  \reset@font\fontsize{#1}{#2pt}%
  \fontfamily{#3}\fontseries{#4}\fontshape{#5}%
  \selectfont}%
\fi\endgroup%
\begin{picture}(7758,5058)(397,-4615)
\put(518,-4381){\makebox(0,0)[lb]{\smash{{\SetFigFont{6}{7.2}{\rmdefault}{\mddefault}{\updefault}{\color[rgb]{0,0,0}$O$}%
}}}}
\put(901,-3571){\makebox(0,0)[lb]{\smash{{\SetFigFont{6}{7.2}{\rmdefault}{\mddefault}{\updefault}{\color[rgb]{0,0,0}$S_1$}%
}}}}
\put(1846,-3931){\makebox(0,0)[lb]{\smash{{\SetFigFont{6}{7.2}{\rmdefault}{\mddefault}{\updefault}{\color[rgb]{0,0,0}$S_2$}%
}}}}
\put(3106,-3571){\makebox(0,0)[lb]{\smash{{\SetFigFont{6}{7.2}{\rmdefault}{\mddefault}{\updefault}{\color[rgb]{0,0,0}$S_3$}%
}}}}
\put(4051,-3976){\makebox(0,0)[lb]{\smash{{\SetFigFont{6}{7.2}{\rmdefault}{\mddefault}{\updefault}{\color[rgb]{0,0,0}$S_4$}%
}}}}
\put(4546,-3031){\makebox(0,0)[lb]{\smash{{\SetFigFont{6}{7.2}{\rmdefault}{\mddefault}{\updefault}{\color[rgb]{0,0,0}$S_5$}%
}}}}
\put(4141,-1276){\makebox(0,0)[lb]{\smash{{\SetFigFont{6}{7.2}{\rmdefault}{\mddefault}{\updefault}{\color[rgb]{0,0,0}$S_7$}%
}}}}
\put(2296,-1681){\makebox(0,0)[lb]{\smash{{\SetFigFont{6}{7.2}{\rmdefault}{\mddefault}{\updefault}{\color[rgb]{0,0,0}$S_8$}%
}}}}
\put(901,-2626){\makebox(0,0)[lb]{\smash{{\SetFigFont{6}{7.2}{\rmdefault}{\mddefault}{\updefault}{\color[rgb]{0,0,0}$S_9$}%
}}}}
\put(1756,-3076){\makebox(0,0)[lb]{\smash{{\SetFigFont{6}{7.2}{\rmdefault}{\mddefault}{\updefault}{\color[rgb]{0,0,0}$S_{10}$}%
}}}}
\put(2791,-4426){\makebox(0,0)[lb]{\smash{{\SetFigFont{6}{7.2}{\rmdefault}{\mddefault}{\updefault}{\color[rgb]{0,0,0}$P$}%
}}}}
\put(3691,-2626){\makebox(0,0)[lb]{\smash{{\SetFigFont{6}{7.2}{\rmdefault}{\mddefault}{\updefault}{\color[rgb]{0,0,0}$S_6$}%
}}}}
\end{picture}%

%% file: Poly_40B.pstex_t
\begin{picture}(0,0)%
\includegraphics{Poly_40B.pstex}%
\end{picture}%
\setlength{\unitlength}{1658sp}%
\begingroup\makeatletter\ifx\SetFigFont\undefined%
\gdef\SetFigFont#1#2#3#4#5{%
  \reset@font\fontsize{#1}{#2pt}%
  \fontfamily{#3}\fontseries{#4}\fontshape{#5}%
  \selectfont}%
\fi\endgroup%
\begin{picture}(5887,4584)(2123,-8033)
\put(2626,-3886){\makebox(0,0)[lb]{\smash{{\SetFigFont{7}{8.4}{\rmdefault}{\mddefault}{\updefault}{\color[rgb]{0,0,0}$f_D(d)$}%
}}}}
\put(3151,-7186){\makebox(0,0)[lb]{\smash{{\SetFigFont{7}{8.4}{\rmdefault}{\mddefault}{\updefault}{\color[rgb]{0,0,0}uniformly distributed in}%
}}}}
\put(3151,-7561){\makebox(0,0)[lb]{\smash{{\SetFigFont{7}{8.4}{\rmdefault}{\mddefault}{\updefault}{\color[rgb]{0,0,0}polyhedron of the left figure}%
}}}}
\put(3265,-6811){\makebox(0,0)[lb]{\smash{{\SetFigFont{7}{8.4}{\rmdefault}{\mddefault}{\updefault}{\color[rgb]{0,0,0}Density of $D$ when $X$}%
}}}}
\put(7051,-7186){\makebox(0,0)[lb]{\smash{{\SetFigFont{7}{8.4}{\rmdefault}{\mddefault}{\updefault}{\color[rgb]{0,0,0}the}%
}}}}
\put(6751,-6811){\makebox(0,0)[lb]{\smash{{\SetFigFont{7}{8.4}{\rmdefault}{\mddefault}{\updefault}{\color[rgb]{0,0,0}is}%
}}}}
\put(6433,-3886){\makebox(0,0)[lb]{\smash{{\SetFigFont{7}{8.4}{\rmdefault}{\mddefault}{\updefault}{\color[rgb]{0,0,0}$P=(4;0)$}%
}}}}
\end{picture}%

%% file: Poly_Int.pstex_t
\begin{picture}(0,0)%
\includegraphics{Poly_Int.pstex}%
\end{picture}%
\setlength{\unitlength}{1450sp}%
\begingroup\makeatletter\ifx\SetFigFont\undefined%
\gdef\SetFigFont#1#2#3#4#5{%
  \reset@font\fontsize{#1}{#2pt}%
  \fontfamily{#3}\fontseries{#4}\fontshape{#5}%
  \selectfont}%
\fi\endgroup%
\begin{picture}(7758,5058)(397,-4615)
\put(518,-4381){\makebox(0,0)[lb]{\smash{{\SetFigFont{6}{7.2}{\rmdefault}{\mddefault}{\updefault}{\color[rgb]{0,0,0}$O$}%
}}}}
\put(901,-3571){\makebox(0,0)[lb]{\smash{{\SetFigFont{6}{7.2}{\rmdefault}{\mddefault}{\updefault}{\color[rgb]{0,0,0}$S_1$}%
}}}}
\put(1846,-3931){\makebox(0,0)[lb]{\smash{{\SetFigFont{6}{7.2}{\rmdefault}{\mddefault}{\updefault}{\color[rgb]{0,0,0}$S_2$}%
}}}}
\put(4051,-3976){\makebox(0,0)[lb]{\smash{{\SetFigFont{6}{7.2}{\rmdefault}{\mddefault}{\updefault}{\color[rgb]{0,0,0}$S_4$}%
}}}}
\put(4546,-3031){\makebox(0,0)[lb]{\smash{{\SetFigFont{6}{7.2}{\rmdefault}{\mddefault}{\updefault}{\color[rgb]{0,0,0}$S_5$}%
}}}}
\put(3691,-2626){\makebox(0,0)[lb]{\smash{{\SetFigFont{6}{7.2}{\rmdefault}{\mddefault}{\updefault}{\color[rgb]{0,0,0}$S_6$}%
}}}}
\put(4096,-1276){\makebox(0,0)[lb]{\smash{{\SetFigFont{6}{7.2}{\rmdefault}{\mddefault}{\updefault}{\color[rgb]{0,0,0}$S_7$}%
}}}}
\put(2296,-1681){\makebox(0,0)[lb]{\smash{{\SetFigFont{6}{7.2}{\rmdefault}{\mddefault}{\updefault}{\color[rgb]{0,0,0}$S_8$}%
}}}}
\put(901,-2626){\makebox(0,0)[lb]{\smash{{\SetFigFont{6}{7.2}{\rmdefault}{\mddefault}{\updefault}{\color[rgb]{0,0,0}$S_9$}%
}}}}
\put(1756,-3076){\makebox(0,0)[lb]{\smash{{\SetFigFont{6}{7.2}{\rmdefault}{\mddefault}{\updefault}{\color[rgb]{0,0,0}$S_{10}$}%
}}}}
\put(2746,-2581){\makebox(0,0)[lb]{\smash{{\SetFigFont{6}{7.2}{\rmdefault}{\mddefault}{\updefault}{\color[rgb]{0,0,0}$P$}%
}}}}
\put(3106,-3571){\makebox(0,0)[lb]{\smash{{\SetFigFont{6}{7.2}{\rmdefault}{\mddefault}{\updefault}{\color[rgb]{0,0,0}$S_3$}%
}}}}
\end{picture}%

%% file: Poly_43B.pstex_t
\begin{picture}(0,0)%
\includegraphics{Poly_43B.pstex}%
\end{picture}%
\setlength{\unitlength}{1658sp}%
\begingroup\makeatletter\ifx\SetFigFont\undefined%
\gdef\SetFigFont#1#2#3#4#5{%
  \reset@font\fontsize{#1}{#2pt}%
  \fontfamily{#3}\fontseries{#4}\fontshape{#5}%
  \selectfont}%
\fi\endgroup%
\begin{picture}(5972,4584)(2123,-8033)
\put(2626,-3886){\makebox(0,0)[lb]{\smash{{\SetFigFont{7}{8.4}{\rmdefault}{\mddefault}{\updefault}{\color[rgb]{0,0,0}$f_D(d)$}%
}}}}
\put(6376,-3886){\makebox(0,0)[lb]{\smash{{\SetFigFont{7}{8.4}{\rmdefault}{\mddefault}{\updefault}{\color[rgb]{0,0,0}$P=(4,3)$}%
}}}}
\put(3121,-6736){\makebox(0,0)[lb]{\smash{{\SetFigFont{7}{8.4}{\rmdefault}{\mddefault}{\updefault}{\color[rgb]{0,0,0}Density of $D$ when $X$ is}%
}}}}
\put(2926,-7111){\makebox(0,0)[lb]{\smash{{\SetFigFont{7}{8.4}{\rmdefault}{\mddefault}{\updefault}{\color[rgb]{0,0,0}uniformly distributed in the}%
}}}}
\put(2926,-7486){\makebox(0,0)[lb]{\smash{{\SetFigFont{7}{8.4}{\rmdefault}{\mddefault}{\updefault}{\color[rgb]{0,0,0}polyhedron of the left figure}%
}}}}
\end{picture}%